\documentclass[letterpaper,11pt]{article}
\title{Efficient and Near-Optimal Online Portfolio Selection}

\newcommand{\hsps}{\hspace{0.04cm}}

\author{
R\'emi J\'ez\'equel\hsps\thanks{Equal contribution.}
\thanks{Capital Fund Management (CFM), Paris, France. 
	\texttt{remi.jezequel@inria.fr}.}
\and
Dmitrii M.~Ostrovskii\hsps\samethanks[1]
\thanks{Georgia Insitute of Technology, School of Mathematics \& H.~Milton Stewart School of Industrial and Systems Engineering (ISyE), Atlanta, USA.
	\texttt{ostrov@gatech.edu}.}
\and
Pierre Gaillard\hsps\thanks{INRIA Grenoble--Rh\^one-Alpes and Universit\'e~Grenoble Alpes, Grenoble, France. 
	\texttt{pierre.gaillard@inria.fr}.}
}
\date{}
\usepackage[utf8]{inputenc}
\usepackage{color}
\usepackage{amsfonts, amsmath, amsthm}
\usepackage{comment}

\usepackage{fullpage}
\usepackage[numbers]{natbib}
\usepackage{amsmath,amsthm,amssymb,amsfonts}
\usepackage{dsfont}
\usepackage{mathrsfs}
\usepackage{url}
\usepackage[dvipsnames]{xcolor}
\usepackage[colorlinks]{hyperref}
\hypersetup{
	colorlinks=true,
    linkcolor=red,
    citecolor=OliveGreen,
    filecolor=black,
    urlcolor=black
}
\usepackage{framed}

\newcommand*\samethanks[1][\value{footnote}]{\footnotemark[#1]}

\usepackage{comment}
\usepackage{xfrac}
\usepackage{enumitem}

\usepackage{microtype}
\usepackage{graphicx}
\usepackage{hyperref}

\usepackage{pifont}

\usepackage{footnote}
\makesavenoteenv{table}
\makesavenoteenv{tabular}

\newcommand{\bSigma}{\boldsymbol{\Sigma}}
\newcommand{\bPi}{\boldsymbol{\Pi}}
\newcommand{\bPsi}{\boldsymbol{\Psi}}
\newcommand{\bLambda}{\boldsymbol{\Lambda}}

\newcommand{\gfrak}{\mathfrak{g}}
\newcommand{\gbold}{\boldsymbol{g}}

\newcommand{\Capital}{\mathsf{Cap}}

\newcommand{\Qmtx}{\boldsymbol{Q}}
\newcommand{\Smtx}{\boldsymbol{S}}
\newcommand{\Rmtx}{\boldsymbol{R}}

\newcommand{\Hmtx}{\boldsymbol{H}}
\newcommand{\Imtx}{\boldsymbol{I}}
\newcommand{\Jmtx}{\boldsymbol{J}}
\newcommand{\Mmtx}{\boldsymbol{M}}

\newcommand{\Dmtx}{\boldsymbol{D}}
\newcommand{\Zmtx}{\boldsymbol{Z}}

\newcommand{\Amtx}{\boldsymbol{A}}

\newcommand{\xHmtx}{\hspace{-0.12cm}\xt{\Hmtx}}
\newcommand{\xHmtxInv}{\xt{\Hmtx}\hspace{-0.12cm}{\vphantom{\Hmtx}}^{-1}}
\newcommand{\sfc}{\mathsf{c}}

\newcommand{\sM}{\textsf{\textit{M}}}

\newcommand{\VolBias}{\textsf{VolBias}}
\newcommand{\LogBias}{\textsf{LogBias}}

\usepackage{algorithm}
\usepackage{algpseudocode}

\algnewcommand{\LineComment}[1]{\State \(\triangleright\) #1}

\newcommand{\OurAlgo}{\textsf{VB-FTRL}}
\newcommand{\OurAlgoImp}{\textsf{VB-FTRL-qN}}

\usepackage{pdfpages}
\usepackage{multirow}
\usepackage{dsfont,makecell}
\usepackage{color, colortbl}

\usepackage{amsopn}
\DeclareMathOperator*{\Argmin}{Argmin}

\DeclareMathOperator*{\argmin}{argmin}

\newtheorem{theorem}{Theorem}[section]
\newtheorem{lemma}{Lemma}[section]
\newtheorem{definition}{Definition}
\newtheorem{proposition}{Proposition}[section]

\newtheorem{remark}{Remark}[section]
\newtheorem{corollary}{Corollary}[section]

\newcommand{\E}{\mathds{E}}
\newcommand{\R}{\mathds{R}}
\newcommand{\N}{\mathds{N}}

\newcommand{\Supp}{\mathsf{Supp}}
\newcommand{\Cov}{\textup{Cov}}
\newcommand{\Ent}{\textup{Ent}}

\newcommand{\ind}{\mathbf{1}}
\newcommand{\res}{\ones}
\newcommand{\Aff}{\mathds{A}}

\newcommand{\cE}{\mathcal{E}}

\newcommand{\Sym}{\mathds{S}}
\newcommand{\cN}{\mathcal{N}}
\newcommand{\cS}{\mathcal{S}}
\newcommand{\cH}{\mathcal{H}}

\newcommand{\Vol}{\textup{Vol}}

\newcommand{\Area}{\textup{Area}}

\newcommand{\ones}{\mathds{1}}

\newcommand{\Tr}{\textup{tr}}

\newcommand{\ri}{\textup{ri}}
\newcommand{\veps}{\varepsilon}

\newcommand{\Diag}{\textup{Diag}}

\newcommand{\Regret}{\mathscr{R}}

\newcommand{\Lag}{\mathsf{Miss}}
\newcommand{\Off}{\mathsf{Gain}}
\newcommand{\Bias}{\mathsf{Bias}}
\newcommand{\Err}{\mathsf{Error}}

\newcommand{\Dec}{\mathsf{Decr}}
\newcommand{\uDec}{\underline{\mathsf{Decr}}}

\newcommand{\trc}{\mathsf{trc}}

\newcommand{\wt}{\widetilde}

\newcommand{\widebar}[1]{%
  \hbox{%
    \vbox{%
      \hrule height 0.5pt 
      \kern0.3ex
      \hbox{%
        \kern-0.05em
        \ensuremath{#1}%
        \kern-0.05em
      }%
    }%
  }%
} 

\newcommand{\wb}{\widebar}

\usepackage{stackengine}
\newcommand\barbelow[1]{\stackunder[1.4pt]{$#1$}{\rule{1.1ex}{.06ex}}}
\newcommand{\wbb}{\barbelow}

\newcommand{\xbar}[1]{%
  \hbox{%
    \vbox{%
      \hrule height 0.5pt 
      \kern0.3ex
      \hbox{%
        \kern-0.20em
        \ensuremath{#1}%
        \kern-0.10em
      }%
    }%
  }%
} 

\newcommand{\xb}{\xbar}
\newcommand{\xt}[1]{\stackrel{\sim}{\smash{#1}\rule{0pt}{1.1ex}\kern0.20em}}

\newcommand{\half}{\frac{1}{2}}
\newcommand{\lsim}{\lesssim}
\newcommand{\gsim}{\gtrsim}

\newcommand{\lang}{\left\langle}
\newcommand{\rang}{\right\rangle}

\renewcommand{\le}{\leqslant}
\renewcommand{\ge}{\geqslant}

\renewcommand{\preceq}{\preccurlyeq}
\renewcommand{\succeq}{\succcurlyeq}

\def\argmin{\mathop{\hbox{\rm argmin}}}

\def\Argmin{\mathop{\hbox{\rm Argmin}}}

\newcommand{\bDelta}{\boldsymbol{\Delta}}
\newcommand{\bNabla}{\boldsymbol{\nabla}}

\newcommand{\sE}{\textsf{E}}

\newcommand{\Dom}{\textup{dom}} 
 
\newcommand{\Int}{\textup{int}}

\newcommand{\proofstep}[1]{$\boldsymbol{{#1}^o}$}

\newcommand{\odima}[1]{{#1}}

\begin{document}


\maketitle

\begin{abstract}
In the problem of online portfolio selection as formulated by Cover (1991)~\cite{cover1991universal}, the trader repeatedly distributes her capital over~$d$ assets in each of~$T > 1$ rounds, with the goal of maximizing the total return.
Cover proposed an algorithm, termed Universal Portfolios, that performs nearly as well as the best (in hindsight) static assignment of a portfolio, with an~$O(d\log(T))$ logarithmic regret. Without imposing any restrictions on the market this guarantee is known to be worst-case optimal, and no other algorithm attaining it has been discovered so far.
Unfortunately, Cover's algorithm crucially relies on computing certain~$d$-dimensional integral, which must be approximated in any implementation;
this results in a prohibitive~$\tilde O(d^4(T+d)^{14})$ per-round runtime for the fastest known implementation due to Kalai and Vempala (2002)~\cite{kalai2002efficient}.
We propose an algorithm for online portfolio selection that admits essentially the same regret guarantee as Universal Portfolios---up to a constant factor and replacement of~$\log(T)$ with~$\log(T+d)$---yet has a drastically reduced runtime of~$\tilde O(d^2(T+d))$ per round.
The selected portfolio minimizes the observed logarithmic loss regularized with the log-determinant of its Hessian -- equivalently, the hybrid {logarithmic-volumetric barrier} of the polytope specified by the asset return vectors.
As such, our work reveals surprising connections of online portfolio selection with two classical topics in optimization theory: cutting-plane and interior-point algorithms.
\end{abstract}
\section{Introduction}
\label{sec:intro}

In the problem of {portfolio selection} in its most general form, the goal is to allocate capital over a set of assets in such a way as to maximize the aggregate return; see e.g.~\cite{markowitz1991foundations,rubinstein2002markowitz} and references therein.
At a high level, one can distinguish two rather different approaches to this general problem. 

\begin{itemize}
\item[$(a)$] Markowitz's {\em mean-variance} theory~\cite{markowitz1952} in which the tradeoff between the expected return and its variance allows to account for the market uncertainty in a one-time investment scenario.
\item[$(b)$] The {\em capital growth theory} ascending to~\cite{kelly1956new}---see~\cite{maclean2011kelly} for a modern treatment---which focuses on the expected {\em log return} of a portfolio and naturally addresses multiple-period investment.
\end{itemize}
We shall focus here on the {\em online portfolio selection} framework introduced by T.~Cover~\citep{cover1991universal,cover1996universal}. 
While being close in spirit to the second approach and addressing the case of multi-period investment, Cover's framework allows to avoid making any stochastic assumptions about the market. This results, on the one hand, in a robust theory well suited for dynamic environment and, on the other hand, in a plethora of algorithms with great practical performance and strong guarantees~\cite{li2018online}.

Cover's framework can be summarized as follows. 
Given the unit initial capital~$\Capital_0 = 1$, in each round~$t \in \{1,2,...,T\}$ the trader chooses an allocation~$w_t \in \Delta_d$ of her previously earned capital~$\Capital_{t-1}$ over~$d$ assets; here~$\Delta_d$ is the standard simplex in~$\R^d$ and~$w_t[i]$---the~$i^{\text{th}}$ entry of~$w_t$---is the share of capital invested into asset~$i \in \{1,2,..,d\}$ at this round.
Then the {\em returns}---the ratios of the closing and opening prices in this round---are revealed in the form of~$x_t \in \R_+^d$ (hereafter~$\R^d_+$ is the nonnegative orthant of~$\R^d$, and~$\R^d_{++}$ is its interior), and the trader's capital is updated as
\[
\Capital_t =  \Capital_{t-1} \, x_t{}^\top w_t.
\]
%
%
%
By Cover, the performance of a strategy that selected portfolios~$w_{1:T} := (w_1, ..., w_T)$ for the market given by~$x_{1:T} := (x_1,...,x_T)$ is quantified by comparing the final capital~$\Capital_T^{\vphantom\top} = \prod_{t=1}^T x_t^\top w_t^{\vphantom\top}$ against
\[
\Capital_T^{\odima{\ast}} := \max_{w \in \Delta_d} \, \prod_{t = 1}^T x_t{}^\top w,
\]
the ``idealized'' final capital attained by the best (in hindsight for~$x_{1:T}$) ``static'' strategy constrained to select the same portfolio in all rounds. Due to the multiplicative structure of~$\Capital_T^{\vphantom\ast}$ and~$\Capital_T^{\odima{\ast}}$, it is natural to define the {\em regret} of~$w_{1:T}$ on market~$x_{1:T}$ as the negative logarithm of~$\Capital_T^{\vphantom\ast}/\Capital_T^{\odima{\ast}}$, i.e.
\begin{equation}
\label{def:regret} 
    \Regret_T(w_{1:T}|x_{1:T}) := \sum_{t \in [T]} \ell_t(w_t) - \min_{w \in \Delta_d} \sum_{t \in [T]} \ell_t(w)
\end{equation}
where we let~$[T] := \{1, 2, ..., T\}$ for~$T \in \N$ with the convention~$[0] = \emptyset$, and define the instantaneous (logarithmic) loss~$\ell_t: \Delta_d \to \R$ by
\begin{equation}
\label{def:log-loss} 
\ell_t(w) := -\log(x_t^\top w).
\end{equation}
Since the instantaneous losses are convex, this formulation falls into the framework of {online convex optimization}, with~\eqref{def:regret} matching the standard notion of regret for an online optimization algorithm~\cite[Sec.~1.1]{hazan2016introduction}; thus, we can benefit from an abundance of techniques developed by the corresponding community~\cite{hazan2016introduction}. 
Despite conveniently putting us into the framework of online convex optimization, comparison against the class of ``static'' strategies, also called {\em constantly rebalanced portfolios} (CRP), might seem restrictive. 
Yet, there are at least two other arguments in its favor.
\begin{itemize}
\item 
On the one hand, CRP strategies are known to be optimal in the i.i.d.~stochastic~setup where~$x_t$'s, instead of being adversarially chosen in each round, come from \odima{a static} distribution unknown to the trader; see~\cite[Theorem~15.3.1]{cover1999elements}. 
Clearly, a uniform upper bound for the regret~\eqref{def:regret}, i.e.~one valid for any sequence~$x_{1:T}$, also applies in the i.i.d.~setup.
Such models are adequate for markets with moderate observation horizons; hence, we may expect good {\em practical} performance for a strategy admitting a uniform regret bound against the best~CRP. This intuition is verified by extensive real-data experiments reported in~\cite{li2014online}.
\item 
On the other hand, the regret measured against the best {\em sequence of portfolios} can be as large as~$T \log(d)$, i.e.~grow linearly with time horizon~$T$. \odima{(Indeed: fix~$w_{1}, ..., w_{T} \in \Delta_d$ and consider~$x_{1:T}$ with~$x_t = e_{i(t)}$ where~$i(t) \in \Argmin_{1 \le j \le d} w_t[j]$ and~$e_1, ..., e_d$ are the canonical basis vectors in~$\R^d$.
Then~$-\log(\Capital_T) = -\log(\prod_{t \le T} w_t[i(t)]) \ge -\log[(\frac{1}{d})^T] = T \log d$; on the other hand,~$\sum_{t \in [T]} \inf_{w \in \Delta_d} \ell_t(w) \ge  \sum_{t \in [T]} \ell_t(e_{i(t)}) = 0$.)}
In contrast, the regret measured against the best CRP strategy, cf.~\eqref{def:regret}, is expected to be {\em sublinear in~$T$} for any reasonable algorithm (in particular, this turns out to be the case for all algorithms discussed further on).
\end{itemize}


In addition to introducing the framework, in his seminal work~\cite{cover1991universal} Cover proposed and studied an algorithm for online portfolio selection, termed {\em Universal Portfolios,}  which consists in selecting
\begin{equation}
\label{def:universal-portfolio}
w_t = \int_{\Delta_d} w \phi_{t}(w) dw
\end{equation}
where~$\phi_{t}(w)$ is the probability density on~$\Delta_d$ proportional to the capital the trader would have earned by the current round~$t$ by selecting portfolio~$w$ in all the previous rounds---in other words, 
\begin{equation}
\label{def:exp-density}
\phi_{t}(w) \propto \exp\left(-\sum_{\tau \in [t-1]} \ell_{\tau}(w) \right).
\end{equation}
The idea here is to give preference to portfolios for which the observed cumulative loss is relatively small without completely discarding portfolios with large losses.
In fact, the same procedure under different names (e.g.,~the method of {exponential weights}, {exponentially weighted averaging forecaster) arises in other contexts: prediction with expert advice~\cite{cesa2006prediction}, online learning  in metric spaces~\cite{paris2021online}, and statistical estimation~\cite{leung2006information,dalalyan2008aggregation,dalalyan2012sharp,rigollet2012sparse,bellec2017optimal,bellec2018optimal}. 
Cover justified his algorithm by showing that for any market realization~$x_{1:T}$, the regret of portfolio sequence~$w_{1:T}$ produced in~\eqref{def:universal-portfolio} satisfies
%
%
%
%
\begin{equation}
\label{eq:regret-bound}
\Regret_T(w_{1:T}|x_{1:T}) = (d-1) \log(T+1)
\end{equation}
see~\cite[Thm.~4]{cover1991universal} or~\cite[Thm.~10.3]{cesa2006prediction}, \odima{and also~\cite[Theorem~7]{hazan2007logarithmic} for a short self-contanted proof with an additional constant term}.
This result justifies the name ``Universal Portfolios:"~\eqref{def:universal-portfolio} admits a strong---nearly independent of~$T$---regret guarantee for {\em any} market.
Moreover, this regret bound turns out to be optimal up to a constant factor: it can be shown that the worst-case over~$x_{1:T}$ regret for {\em any} algorithm is at least~$\frac{d-1}{2}\log(\frac{\pi T}{d}) + \veps_d(T)$ where~$\veps_d(T) \to 0$ as~$T \to \infty$; see~\cite[p.~282]{cesa2006prediction}. 

Unfortunately, Cover's algorithm has a major drawback: it relies on integration over~$\Delta_d$, which cannot be performed explicitly; thus, one has to approximate the integral with sufficient accuracy to obtain a tractable implementation. 
\odima{To obtain such an approximation,} one could {\em sample} from~$\phi_{t}$---a log-concave distribution, which can be done via Markov Chain Monte Carlo type techniques. 
This approach was carried out by Kalai and Vempala in~\cite{kalai2002efficient}, resulting in a~$\wt O(T^5(T+d)^{9}d^4)$ per-round runtime.\footnote{In~\cite{van2020open} this result is miscited as~$\wt O(T^3 (T+d)d^2)$, seemingly due to neglecting that~$\frac{1}{\delta_0^2} = T(T+d)^2d$ in~\cite[Thm.~2]{kalai2002efficient}.} 
(Hereafter,~$\wt O(\cdot)$ hides a polylogarithmic factor in~$T+d$.)
Albeit polynomial in~$d$ and~$T$, this runtime essentially rules out any practical application scenario.
Even though this can likely be improved with recent advances in log-concave sampling 
(see e.g.~\cite{brosse2017sampling,narayanan2017efficient,dwivedi2018log,bubeck2018sampling,wu2021minimax}), \odima{we are not aware of such improvements in the existing literature.}
Meanwhile, no algorithm {\em except Universal Portfolios} is known to admit a regret bound matching~\eqref{eq:regret-bound} up to a constant factor.
In other words, the following problem stands:
\begin{framed}
\begin{center}
\vspace{-0.2cm}
\begin{quote}
\hspace{-1.05cm}{\em \mbox{Propose a \textbf{computationally feasible\vphantom{gp}} and \textbf{near-optimal} online portfolio selection algorithm.}}
\end{quote}
\vspace{-0.2cm}
\end{center}
\end{framed}
We shall address this problem by presenting such an algorithm, proving that its regret admits a bound matching~\eqref{eq:regret-bound} up to a constant factor and replacement of~$\log(T)$ with~$\log(T+d)$, and proposing an efficient implementation for it.
In the remaining part of this section, we shall informally state our main result, briefly summarize our contributions, discuss the related work, and outline the subsequent sections of the paper.
Let us now give an abridged formulation of our main result.
\begin{framed}
\begin{theorem}[cf.~Theorem~\ref{th:volumetric-newton}]
\label{th:intro}
There is an algorithm for online portfolio selection which has an~$O(d\log(T+d))$~regret, runs in~$\wt O(d^2T + d^3)$ per round, and utilizes~$O(dT+d^2)$ memory.
\end{theorem}
\end{framed}
\noindent We see that the improvement of runtime compared to the algorithm of Kalai and Vempala~\cite{kalai2002efficient}---the only one known so far with a regret guarantee of the same order---is drastic; in particular, our proposed algorithm can be run on a personal computer for~$d$ and~$T$ in the range of thousands. 
Moreover, in addition to favorable computational guarantees the algorithm is conceptually easy to implement: in each round one has to minimize a convex potential function with a readily available second-order oracle, so one can simply use Newton's method. In fact, the potential we use belongs to the class of {\em self-concordant} functions---``canonical'' objectives for Newton's method \cite{nesterov2003introductory}; this allows us to implement the algorithm while performing only~$O(\log(T))$ Newton steps in each round.


We shall now briefly summarize the contributions of our paper and outline its organization.

\paragraph{Contributions and outline of further sections.}
In Section~\ref{sec:algorithm} we present our key contributions.
\begin{itemize}
\item 
We introduce a new algorithm for online portfolio selection, called~\OurAlgo{}---the acronym stands for {\bf Volumetric-Barrier enhanced Follow-The-Regularized-Leader},---and state Theorem~\ref{th:volumetric} showing that the regret of~\OurAlgo{} admits an~$O(d \log(T+d))$ upper bound. 
\item 
We then investigate an intriguing question of how~\OurAlgo{} is related to Cover's Universal Portfolios algorithm. 
To this end, we first pass to a variational formulation of Cover's update by finding a functional over distributions supported on~$\Delta_d$ which is minimized by~$\phi_t(\cdot)$, cf.~\eqref{def:exp-density}. 
We then show that the update of~\OurAlgo{} corresponds to solving this minimization problem approximately, by passing to the class of appropriately truncated Gaussian distributions with unknown expectation and covariance, and solving for covariance with a given expectation.
Moreover, we quantify the accuracy of this approximation (Propositions~\ref{prop:Dikin-relaxation-tightness} and~\ref{prop:Gaussian-relaxation-tightness}); this result might be of independent interest and find use beyond the context of portfolio selection.
\end{itemize}
Section~\ref{sec:volumetric-proof} is devoted to analyzing the regret of~\OurAlgo{} and proving Theorem~\ref{th:volumetric}. 
Finally, in Section~\ref{sec:implementation} we present a procedure implementing~\OurAlgo{} and quantify its time and memory costs.







\subsection{Overview of related work}
\label{sec:related-work}
Since Cover's work~\cite{cover1991universal}, the online portfolio selection framework has been studied extensively due to its practical usefulness and theoretical challenges. 
Here we focus on the line of research done in the online convex optimization community and concerned with designing efficient methods with provable regret guarantees; for reader's convenience we recap these results in~Table~\ref{tab:rates}.
For a more application-centered perspective and experimental comparison of algorithms one may refer to~\cite{li2018online}.


To the best of our knowledge, the authors of~\cite{helmbold1998line} were the first to realize that Cover's problem of competing with the best CRP can be reformulated as online convex optimization on the simplex~$\Delta_d$ with logarithmic instantaneous losses~$\ell_t(w) = -\log(x_{t}^\top w)$, cf.~\eqref{def:regret}.
They advocated the exponentiated gradient (EG) method~\cite{kivinen1997exponentiated} over the basic online gradient descent~\cite{zinkevich2003online}, due to EG being better adapted to the simplex type geometry. Unfortunately, both methods require the loss gradients to be bounded, which is a restrictive assumption. 
Indeed, letting~$\|\cdot\|_p$ be~$\ell_p$-norm on~$\R^d$ we have that
\begin{equation}
\label{eq:loss-grad}
\| \nabla \ell_{t}(w) \|_p = \frac{\|x_{t}\|_p}{x_t^\top w}. 
\end{equation}
For~EG, the assumption is that~$\| \nabla \ell_{t}(w) \|_\infty \le G_\infty$ for some~$G_\infty < \infty$; 
by~\eqref{eq:loss-grad} this can be enforced in either of two ways: 
(a) by restricting the set of portfolios to~$\{ w \in \Delta_d: \min_{i} w[i] \ge {1}/{G_{\infty}}\}$, or
(b) by bounding the dynamic range of asset returns, i.e. requiring that~$\frac{\max_i x_{t}[i]}{\min_i x_{t}[i]} \le G_{\infty}$ for all~$1 \le t \le T$. 
Another issue with these methods is that the regret is suboptimal in~$T$, growing as~$\sqrt{T}$, cf.~Table~\ref{tab:rates}. 




\begin{table}[!t]
\centering
\begin{tabular}{|c|c|c|c|}
	\hline
	{\bf Algorithm} & {\bf Regret} & {\bf Runtime (per round)}  &  {\bf Sources} \vphantom{$A^{A^A}$} \\
	\hline 

	{Universal Portfolios} & ${d\log(T)}$ & $d^4 T^{14}$  & \cite{cover1991universal,kalai2002efficient} \vphantom{$A^{A^A}$} \\
	\hline

	Online Gradient~Descent & $G_2 \sqrt{T}$ & $d$  & \cite{zinkevich2003online} \vphantom{$A^{A^{A^A}}$} \\
	Exponentiated~Gradient & ${G_{\infty}}\sqrt{T\log(d)}$ & $d$  & \cite{kivinen1997exponentiated,helmbold1998line} \vphantom{$A^{A^A}$} \\
	\hline

	Online Newton Step (ONS) & ${G_{\infty}} d\log(T)$ & $d^2$
& \cite{agarwal2006algorithms,hazan2007logarithmic} \vphantom{$A^{A^{A^A}}$} \\
	Soft-Bayes & $\sqrt{d{T}\log(d)}$ & $d$  & \cite{orseau2017soft} \vphantom{$A^{A^A}$} \\
	\hline

	Ada-BARRONS & ${d^2 \log^4(T)}$ & $d^{2.5} T$  & \cite{luo2018efficient} \vphantom{$A^{A^{A^A}}$} \\
	BISONS & ${d^2 \log^2(T)}$ & $\text{poly}(d)$ & \cite{zimmert2022pushing} \vphantom{$A^{A^A}$} \\
	AdaMix$+$DONS & ${d^2 \log^5(T)}$ & ${d^3}$  & \cite{mhammedi2022damped} \vphantom{$A^{A^A}$} \\
	\hline

    {\bfseries \OurAlgo} & $ {d\log(T)}$ & ${d^2 T}$ & {\bf our paper} \vphantom{$A^{A^{A^A}}$}\\
    \hline 
\end{tabular}
\caption{Regret guarantees and per-round runtime for various online portfolio selection algorithms. 
$G_{p}$ is an upper bound on~$\max_{1 \le t \le T}\|\nabla \ell_t(w)\|_{p}$ whenever such a bound is assumed.
For brevity we omit a constant factor for the regret and a~polylogarithmic factor for the runtime, and assume~$T \ge d$.}
\label{tab:rates}
\end{table}

The two issues afflicting the previous methods were addressed---in isolation---in~\cite{orseau2017soft} and~\cite{agarwal2006algorithms,hazan2007logarithmic}.
Namely, the Soft-Bayes algorithm from~\cite{orseau2017soft} achieves an~$O(\sqrt{d T \log(d)})$ regret without assuming boundedness of the loss gradients.
Meanwhile, Online Newton Step (ONS), an algorithm proposed in~\cite{agarwal2006algorithms,hazan2007logarithmic}, has an~$O(G_{\infty} d \log(T))$ regret under the same assumption of~$\ell_{\infty}$-bounded loss gradients as for~EG, which matches the guarantee~\eqref{eq:regret-bound} for Universal Portfolios in the regime of a constant~$G_{\infty}$.
All these algorithms can be implemented in a fully incremental fashion, with per-round runtime depending only on~$d$, namely~$\wt O(d)$ for the first-order methods (EG, Soft-Bayes) and~$O(d^2)$ for ONS.
\odima{Notably, the latter runtime seems to be the best one could hope for without sacrificing the regret optimality: indeed, it corresponds to the overall runtime~$O(Td^2)$ of fitting (batch) {\em ordinary least-squares\em} for~$T \ge d$ observations in~$\R^d$ dimensions, and it seems unreasonable to expect the class of (self-concordant) logarithmic losses~$\{-\log(\lang x, \cdot \rang) \;|\; x \in \R^d_+\}$ over~$\R^d_+$ to be more tractable, in the context of online optimization, than that of quadratic losses~$\{(\lang x, \cdot \rang - y)^2 \;|\; x \in \R^d, y \in \R\}$ over~$\R^d$.}

The next improvement for regret was achieved in~\cite{luo2018efficient}. 
In their Ada-BARRONS algorithm, the authors managed to reach a polylogarithmic in~$T$ regret without a gradient boundedness assumption.
This was done by combining ONS with a logarithmic barrier regularizer, a clever strategy of adaptive stepsize selection, and the use of self-concordance in the analysis.
However, the regret remained suboptimal, scaling as~$O(d^2 \log^4(T))$. In addition, the per-round runtime deteriorated to~$\wt O(d^{2.5}T)$, thus becoming~$T$-dependent.
Very recently, two competing works~\cite{zimmert2022pushing,mhammedi2022damped} achieved a~$T$-independent per-round runtime while preserving the~$\wt O(d^2)$ regret guarantee of Ada-BARRONS. 
In both cases, the crucial step was to combine Ada-BARRONS with an appropriate scheme of adaptive restarts. 
\odima{However, none of these algorithms attains a regret better than~$\wt O(d^2)$.
In Appendix~\ref{apx:adabarron}, focusing on Ada-BARRONS, we discuss why~$\wt O(d^2)$ could be unimprovable for such ``ONS-derived'' algorithms.}

Despite all these efforts and a plethora of computationally feasible algorithms resulting from them, Universal Portfolios has remained the only algorithm so far with an {optimal}~$O(d\log(T))$ regret.
The challenge of providing a practical algorithm with a similar regret guarantee is well known. 
The interest to it was reignited in~\cite{van2020open}, where it was conjectured that Follow-The-Leader regularized with a logarithmic barrier (LB-FTRL)---an algorithm coninciding with~\OurAlgo{} without a volumetric regularizer (i.e.~with~$\mu=0$)---is regret-optimal. 
Since such an algorithm could be implemented with~$\wt O(d^2 T)$ per-round runtime via Newton's method, this would result in a regret-optimal and computationally feasible procedure. 
Yet, \cite{zimmert2022pushing} recently disproved this conjecture: as it turns out, the regret of LB-FTRL can be of order~$2^d\log(T) \log(\log(T))$ whenever~$T > \textup{poly}(2^d)$.  
In addition to the arguments we give in Section~\ref{sec:algorithm}, this negative result demonstrates that adding a volumetric-barrier regularization term is crucial for obtaining a {regret-optimal} FTRL-type strategy.

\paragraph{Connections with cutting-plane and interior point methods.}
An interesting byproduct of our work is shedding new light on the {\em volumetric barrier of a polytope}, an object first studied by P.~Vaidya in his seminal work~\cite{vaidya1989new} in the context of cutting-plane methods for convex optimization.
Vaidya's goal was to supercede the ellipsoid method~\cite{shor1971minimization,khachiyan1979polynomial,nemirovskii1983problem} as the state-of-the-art cutting plane method, and he used the volumetric barrier as the key component of his algorithm. 
Once the general theory of interior-point methods (IPM) had been developed by Nesterov and Nemirovskii~\cite{nesterov1994interior}, the volumetric barrier was used by K.~Anstreicher~\cite{anstreicher1997volumetric} as a self-concordant barrier for a polytope, replacing the logarithmic barrier in this role; this resulted in a faster-converging IPM for linear programming. 
Remarkably, in both these scenarios further improvements were obtained using the ``hybrid'' volumetric-logarithmic barrier (see~\cite{vaidya1989new,anstreicher1997volumetric}), which mimics the potential used in~\OurAlgo{} (see~Section~\ref{sec:volumetric-proof}) except for a different choice of regularization parameter~$\mu$ (namely,~\cite{anstreicher1997volumetric} prescribes~$\mu = \frac{T+d-1}{d-1}$). 
To our best knowledge, the volumetric and hybrid barriers have never been used in online learning. 
Their emergence in this context, and in such a vital role, came to us as a suprise and challenged us \odima{to show that \OurAlgo{} can be seen as a variational approximation of Universal Portfolios with a guaranteed approximation accuracy, as we demonstrate in Section~\ref{sec:VB-FTRL-vs-Cover}.}
\section{Proposed algorithm and its connection with Universal Portfolios}
\label{sec:algorithm}

In this section, we present \OurAlgo{} and derive it as an approximation of Universal Portfolios~\cite{cover1991universal}.

First of all, let us establish the notation and briefly recap the setup of online portfolio selection.\vspace{-0.2cm}

\paragraph{Notation.} 
\odima{We introduce~$\wb \R := \R \cup \{+\infty\}$, i.e.~the upper-extended real line, and denote with~$\R_+$ (resp.~$\R_{++}$) the sets of nonnegative (resp.~positive) reals.} 
We denote with~$\R^d_+$ the nonnegative orthant in~$\R^d$; with~$\Delta_d$ the standard simplex in~$\R^d$ and with~$\ri(\Delta_d)$ its relative interior.
We let~$e_i$ be the~$i^\text{th}$ canonical basis vector in~$\R^d$, and~$x[i]$ be the~$i^\text{th}$ entry of~$x$.
$\ones_d$ be the vector of all ones in~$\R^d$, and~$\Imtx_d$ is the~$d \times d$ identity matrix; we also write~$\Imtx$ when the dimension is clear from the context.
$[n]$ be the set of the first~$n$ positive integers;~$[0]$ is the empty set.
We use standard asymptotic notation: given two functions~$f,g > 0$ of the same argument~$a > 1$, notation~$g = O(f)$ tells that~$g(a) \le cf(a)$ for some generic constant~$c > 0$ and all~$a$ in the common domain of~$f$ and~$g$; 
$g = \wt O(f)$ means that~$g(a) \le cf(a)\log^c(ea)$. 
We let~$\|\cdot\|_{p}$,~$p \ge 1$, be the~$\ell_p$-norm on~$\R^n$.
We let~$\Sym^n,\Sym_+^n,\Sym_{++}^n$ be the sets of~$n \times n$ symmetric, positive-semidefinite, and positive-definite matrices, and we use boldfaced capital letters for such matrices.
For~$\Mmtx \in \Sym_+^n$ and~$u,v \in \R^n$, we let~$\|u\|_{\Mmtx} := \sqrt{u^\top \Mmtx u}$ and~$\lang u,v \rang_{\Mmtx} := u^\top \Mmtx v$.
We let~$\| M \|$ the operator norm of a matrix~$M$.
We use the ``Matlab'' matrix notation:~$[A_1,\,A_2]$ (resp.~$[A_1;\,A_2]$) is the horizontal (resp.~vertical) concatenation of two matrices~$A_1, A_2$ with compatible dimensions.\vspace{-0.1cm}

\paragraph{Affine reparametrization.}
The simplex~$\Delta_d \subset \R^d$ has empty interior, which prevents us from using the machinery of self-concordant functions. 
To circumvent this issue, for~$d > 1$ we define  
\begin{equation}
\label{def:solid-simplex}
\bDelta_{d-1} := \{v \in \R_+^{d-1}: \ones_{d-1}^\top v \le 1\},
\end{equation}
a ``solid'' simplex in~$\R^{d-1}$, and note that the affine mapping~$v \mapsto Av + e_d$ with a~$d \times (d-1)$ matrix\vspace{-0.2cm}
\begin{equation}
\label{def:coordinate-matrix}
A := \begin{bmatrix} \Imtx_{d-1} \\ -\ones_{d-1}^\top \end{bmatrix}\vspace{-0.2cm}
\end{equation}
bijectively maps~$\bDelta_{d-1}$ onto~$\Delta_{d}$, and~$\R^{d-1}$ onto the affine span of~$\Delta_d$, i.e.~onto the hyperplane\vspace{-0.2cm}
\begin{equation}
\label{def:affine-subspace}
\Aff_d := \{ w \in \R^{d}: w^\top \ones_d = 1\}.
\end{equation}
As such, restricting a function~$f:\R^{d} \to \wb\R$ to~$\Aff_d$ and reparametrizing results in~$f^{\res}: \R^{d-1} \to \wb\R$,
\begin{equation}
\label{def:affine-reparametrization}
f^\res(v) := f(Av + e_d),
\end{equation}
with the gradient and Hessian as follows:
\begin{equation}
\label{def:affine-reparametrization-grad-and-hess}
\nabla f^\res(v) = A^\top \nabla f(Av + e_d) 
\quad \text{and} \quad 
\nabla^2 f^\res(v) = A^\top \nabla^2 f(Av + e_d) A.
\end{equation}
Since~$A$ has full column rank,~$f^{\res}$ is strictly convex (on~$\R^{d-1}$) whenever~$f$ is strictly convex on~$\Aff_d$. 
Finally, one can easily check that the mapping~$v \mapsto Av + e_d$ of~$\R^{d-1}$ to~$\Aff_{d}$ has the inverse given by
\begin{equation}
\label{def:affine-reparametrization-inv}
w \mapsto A^+(w-e_d)
\end{equation}
where~$A^+ = (A^\top A)^{-1}A^\top$, $A^+ \in \R^{(d-1) \times d}$, is the left pseudoinverse of~$A$, i.e.~such that~$A^+ A = \Imtx_{d-1}$. 
Note that the case~$d = 1$ is trivial:~$w_{t} \equiv 1$ is a unique possible strategy;
\odima{correspondingly, all our subsequent regret bounds are proportional to~$d-1$, and indeed vanish when~$d=1$. 
For simplicity, in the proofs we use the above affine reparametrization assuming that~$d > 1$, and omit the case~$d=1$.}

\paragraph{Summary of the setup.}
Online portfolio selection is a game between the {\em learner} (i.e.~the trader) and her {\em adversary} (the~market) played over~$T \ge 1$ rounds according to the following protocol:
\begin{framed}
\hspace{-0.5cm} In each round~$t \in [T]$:
\begin{enumerate}
\item
The learner selects a new portfolio, i.e.~a distribution~$w_{t} \in \Delta_d$.
\item
The adversary observes~$w_t$ and selects a new vector of asset returns~$x_{t} \in \R^d_+$.
\item 
The learner observes~$x_t$ and suffers the loss~$\ell_{t}(w_t)$ where~$\ell_{t}(w) := -\log(x_t^\top w)$ for~$w \in \Delta_d$.
\end{enumerate}
\end{framed}
\noindent Formally, an online portfolio selection~{\em algorithm} or {\em strategy} is a sequence of mappings~$(\cS_{0}, \cS_1, ... \cS_{T-1})$ where~$\cS_{t-1}$ maps the history~$(w_\tau, x_\tau)_{\tau \in [t-1]}$ on~$\Delta_d$, that is~$w_t = \cS_{t-1}(w_1, x_1,...,w_{t-1}, x_{t-1})$ for~$t > 1$ and~$w_1 = \cS_0(\emptyset)$.
Pitting it against an adversary, as per the above protocol, results in two sequences
\begin{equation}
\label{def:sequence-notation}
\begin{aligned}
w_{1:T} := (w_1,  \, ..., \; w_T), \quad 
x_{1:T} := (x_1,\,\, ..., \; x_T).
\end{aligned}
\end{equation}
%
As previously discussed, we quantify the performance of an algorithm that produced a sequence of portfolios~$w_{1:T}$ for the market realization~$x_{1:T}$, in terms of its {\em regret} against the best CRP for~$x_{1:T}$,
\[
\Regret_T (w_{1:T}|x_{1:T}) := \sum_{t \in [T]} \ell_{t}(w_t) - \min_{w \in \Delta_d} \sum_{t \in [T]} \ell_{t}(w).
\]
Following~\cite{cover1991universal} and subsequent works, we aim at constructing an algorithm for which~$\Regret_T (w_{1:T}|x_{1:T})$ admits an upper bound of the form~\eqref{eq:regret-bound} {\em uniformly} over all choices of~$x_1,..., x_T \in \R^d_+$. 
Note that such a bound would imply that at any round~$t \in [T]$ the adversary is allowed to know the learner's {\em algorithm}---i.e.~the whole sequence~$(\cS_0,...,\cS_{T-1})$ rather than just the portfolios~$w_{1:t}$ selected so far.

%
%

We are now ready to present our proposed algorithm and a regret guarantee that it admits.

\subsection{Algorithm and main result}
\label{sec:algorithm-and-main-result}

Our algorithm relies on some auxiliary functions. 
Let~$\lambda,\mu > 0$ be two regularization parameters whose values will be specified a bit later.
The {\em logarithmic barrier}~$R: \R^d \to \wb\R$ of~$\R^d_+$ is defined by
\begin{equation}
\label{def:log-barrier}
R(w) := 
\left\{
\begin{aligned}
&-\sum_{i \in [d]} \log(w[i]) \;\; &&\text{if} \;\; w \in \R^d_{++},& \\
&+\infty \quad &&\text{otherwise}.
\end{aligned}
\right.
\end{equation}
Note that,~$R(w)$ is proper, lower semicontinuous, and strictly convex. 
Moreover,~$R^\res(v)$, cf.~\eqref{def:solid-simplex}--\eqref{def:affine-reparametrization}, is a barrier on~$\bDelta_{d-1}$---in other words,~$\Dom(R^\res) = \Int(\bDelta_{d-1})$, and~$R^\res(v)$ diverges as~$v$ approaches the boundary of~$\bDelta_{d-1}$.
Now, let~$L_t(w)$ be the observed cumulative loss regularized by~$R(w)$, namely
\begin{equation}
\label{def:LB-potential}
L_t(w) := \sum_{\tau \in [t]} \ell_\tau(w) \; + \; \lambda R(w)
\end{equation}
for~$\lambda > 0$ to be specified later; note that~$L_0(w) = \lambda R(w)$. 
Define the {\em volumetric barrier}~$V_t:\R^d \to \wb\R$:
\begin{equation}
\label{def:VB-regularizer}
V_t(w) := 
\left\{
\begin{aligned}
&\tfrac{1}{2} \log\det(A^\top \nabla^2 L_t(w) A) \;\; &&\text{if} \;\; w \in \R^d_{++},& \\
&+\infty \quad &&\text{otherwise}.
\end{aligned}
\right.
\end{equation}
\odima{This function was initially introduced by Vaidya in~\cite{vaidya1989new} in the context of cutting-plane algorithms, and subsequently used by Nesterov and Nemirovskii~\cite{nesterov1994interior} in the context of interior-point methods. 
Geometrically,~$V_t(w)$ is the logarithm of the inverse $(d-1)$-dimensional volume of the {\em Dikin ellipsoid} 
\[
\{w' \in \Aff_d: \|w'-w\|_{\nabla^2 L(w)} < 1\}
\]
of~$L_t$ (at~$w$) restricted to the affine subspace~$\Aff_d$, cf.~\eqref{def:affine-reparametrization}--\eqref{def:affine-reparametrization-grad-and-hess}. 
(See Appendix~\ref{apx:self-conc} for the background on self-concordant functions.)
As such,~$V_t$ has larger values at points with larger {\em curvature} of~$L_t$, and vice versa; in particular,~$V_t^\ones$ is also a barrier for~$\bDelta^{d-1}$. 
On the other hand,~$V_t$ does not take into account {\em first-order} information~(i.e.~$\nabla L_t(w)$), which suggests to use it as a regularizer.
In accordance with this suggestion,} for~$\mu > 0$ to be speficied later, we define the potential function~$P_t: \R^d \to \wb\R$ by
\begin{equation}
\label{def:VR-potential}
P_t(w) := L_t(w) + \mu V_t(w).
\end{equation}
Our proposed algorithm, called~\OurAlgo{}, amounts to iteratively minimizing this potential, that is:
\begin{framed}
\begin{equation}
\label{def:VB-FTRL}
\tag{\textsf{\OurAlgo}}
    w_{t} = \argmin_{w \in \Delta_d} P_{t-1}(w). 
\end{equation}
\end{framed}
Correctness of the update~\eqref{def:VB-FTRL}---i.e.~that the update~$w_{t}$ therein exists and is unique---follows from the strict convexity of~$V_t$ for all~$t \in [T]$. This fact---\odima{following, in a nontrivial fashion, from the algebraic structure of~$L_t$---}is well known (\cite{vaidya1989new,anstreicher1997volumetric}) in the case~$\lambda = \mu = 1$; the general case is analogous and addressed in Appendix~\ref{apx:derivatives}.
In fact, direct differentiation allows to obtain explicit formulae for the gradient and Hessian of~$V_t$ and~$P_t$; following~\cite{anstreicher1997volumetric}, one can then show that~$P_t^\res(v)$ is a self-concordant function on~$\bDelta_{d-1}$. 
(We provide these conceptually straightforward, but tedious calculations in Appendix~\ref{apx:derivatives}, and give some background on self-concordant functions in~Appendix~\ref{apx:self-conc}.)
Such functions are known as the ``natural'' class of objectives for Newton's method, for which this method admits affine-invariant global convergence guarantees~(see~\cite[Chapter~4]{nesterov2003introductory}). 
This naturally leads to an algorithmic implementation of~\OurAlgo{} via Newton's method; we shall present this implementation and discuss its tractability in Section~\ref{sec:implementation}.
%
%
Postponing further discussion of the computational matters to~Section~\ref{sec:implementation}, let us now state and discuss a regret guarantee for~\OurAlgo{}.
}

\begin{framed}
\begin{theorem}[Regret of~\OurAlgo{}]
\label{th:volumetric}
For any~$T \in \N$ and market realization~$x_{1:T}$, one has that
\[
\Regret_T(w_{1:T}|x_{1:T}) \le (\lambda + 2\mu)(d-1)\log(T+\lambda d)
\]
for the sequence~$w_{1:T}$ produced by~\eqref{def:VB-FTRL} with~$\lambda, \mu$ satisfying
$
\frac{1}{\lambda} \big( 1 + \frac{2\mu}{\lambda} \big)^2 \le \min \left\{\frac{1}{4},  \frac{5\mu}{8(1+\lambda)} \right\}
$
and~$\lambda \ge 2e$.
In particular, with~$\lambda = 16$ and~$\mu = 7$ we guarantee that
\[
\Regret_T(w_{1:T}|x_{1:T}) 
\le 30 \hspace{0.03cm} (d-1) \log(T + 16 d).
\]
\end{theorem}
\end{framed}
%
%
%
%
%
%
%
%
%

On the one hand, the regret bound for~\OurAlgo{} matches the bound for Universal Portfolios, albeit with a worse constant factor (cf.~\eqref{eq:regret-bound}), and also matches the worst-case lower bound from~\cite{cesa2006prediction}. In other words, the regret optimality of Universal Portfolios is preserved in the proposed algorithm.

On the other hand, the update in~\OurAlgo{} reduces to solving a {\em convex optimization} problem, and thus can be efficiently implemented; as such, we address the challenge put forward in Section~\ref{sec:intro}. This is in stark contrast with~Universal Portfolios where one has to compute a multivariate integral in each round. In fact, the update in Universal Portfolios can be seen as minimizing a certain convex functional on the space of {\em probability measures supported on}~$\Delta_d$, which is way more challenging, from an algorithmic standpoint, than convex optimization on~$\Delta_d$. 
In Section~\ref{sec:VB-FTRL-vs-Cover} we discuss this point in more detail, and shed light on the connection between~\OurAlgo{} and Universal Portfolios.

\paragraph{\odima{Why the volumetric barrier?}}
Putting~$\mu = 0$ in~\eqref{def:VB-FTRL} we recover the {\em logarithmic barrier} regularized follow-the-leader, the algorithm conjectured to have an~$O(d \log(T))$ regret in~\cite{van2020open}. As we have already observed in Section~\ref{sec:related-work}, the recent refutation of this conjecture in~\cite{zimmert2022pushing} indicates that \odima{one has to employ some {\em time-dependent} regularization} to attain the optimal regret with an~FTRL-type strategy. 
\odima{One may ask, then, why the volumetric barrier is suitable for the task at hand.
While a full and rigorous answer to this question can hardly be summarized in one paragraph---after all, our whole paper is, essentially, our best attempt at providing it---let us give some preliminary intuition.

\begin{itemize}
\item[--] First, the regret analysis of LB-FTRL, skectched in~\cite{van2020open}, indicates that the main challenge is the accumulation of the terms~$\pi_t = \| \nabla_t \|^2{}_{\Hmtx_t^{-1}}$, where~$\nabla_t = A^\top \nabla \ell_t(w_t)$ and~$\Hmtx_t = A^\top \nabla^2 L_t(w_t) A$ (i.e.~$\pi_t$ is the observed leverage score of the last observation): indeed, choosing~$\lambda = O(1)$ only allows to bound each~$\hat\pi_t$ with a constant, so~$\sum_{t \in [T]} \hat\pi_t$ can grow linearly with~$T$. As we show in more detail in Section~\ref{sec:volumetric-proof} (cf.~\eqref{eq:offset-bound-sketch} and the ambient discussion), using the volumetric barrier augments the regret decomposition with {\em negative} terms of the form~$\mu[V_{t-1}(w_t) - V_t(w_t)]$, where
\[
\begin{aligned}
2\mu[V_{t-1}(w_t) - V_t(w_t)]
= \mu\log\det\big(\Imtx_d - \Hmtx_{t}{}^{-1/2}\nabla_t \nabla_t{}^\top \Hmtx_{t}{}^{-1/2}\big)
= \mu\log(1-\pi_t) 
\le -\mu\pi_t,
\end{aligned}
\]
while also bringing an~$O(\mu d\log(T))$ error.
Hence,~$\mu = O(1)$ results in an~$O(d \log(T))$ regret.
\item[--]
Second,~\OurAlgo{} can be derived from Universal Portfolios: in a nutshell, one replaces the Gibbs distribution~\eqref{def:exp-density} with its local (truncated) Gaussian approximation at~$w \in \Delta_d$, and selects the expectation~$w_t$ to minimize the Gibbs functional over this class of distributions. 
In Section~\ref{sec:VB-FTRL-vs-Cover}, we put this brief explanation at a rigorous footing and quantify the accuracy of such Gaussian approximation in terms of the Gibbs functional; this is of independent interest.
\end{itemize}
}



%
%
%
%
%

\subsection{\OurAlgo{} as an approximation of Universal Portfolios}
\label{sec:VB-FTRL-vs-Cover}

In this section, our goal is to show that~\eqref{def:VB-FTRL} naturally arises as an approximation of the Universal Portfolios update~\eqref{def:universal-portfolio}; moreover: informally speaking, the accuracy of this approximation can be controlled in natural terms. 
Although arguably quite natural, the argument about to be presented is a rather delicate one, and the corresponding accuracy bound depends on some auxiliary objects that have to be introduced first. 
For the sake of exposition, we shall split this argument into three logical steps that are briefly outlined below, and defer some technical details to Appendix~\ref{apx:cover-correspondence-proofs}.

\begin{itemize}

\item[\proofstep{1}.] We first incorporate two parameters~$(\lambda, \mu)$ into~\eqref{def:universal-portfolio}, so that ``vanilla'' Cover's update corresponds to~$\lambda = 1, \mu = 0$. We then establish~Proposition~\ref{prop:cover-general-regret} stating that the generalized strategy is optimal for any (constant)~$\lambda$ and~$\mu$.
Besides, we observe that the corresponding distribution~$\phi_t$ admits a variational formulation as a minimizer of the Gibbs' potential~$F_{t-1}[\cdot]$,
\begin{equation}
\label{eq:variational-v0}
\phi_t \in \argmin_{\phi \in \Supp(\Delta_d)} F_{t-1}[\phi];
\end{equation}
here~$\Supp(\Delta_d)$ is the set of distributions supported on~$\Delta_d$, and~$F_{t-1}[\cdot]$ shall be defined in~\eqref{eq:universal-portfolios-gibbs}. 

\item[\proofstep{2}.] 
Next we focus on minimization problem~\eqref{eq:variational-v0}. Using self-concordance~\cite{nesterov2003introductory}, we show that
\begin{equation}
\label{eq:sandwiching-v0}
c\,\Tr(\nabla^2 L_{t-1}(\hat w) \, \Cov[\phi]) - \mu \Ent[\phi]
\le 
F_{t-1}[\phi] - L_{t-1}(\hat w)
\le 
C\,\Tr(\nabla^2 L_{t-1}(\hat w) \, \Cov[\phi]) - \mu \Ent[\phi]
\end{equation}
for any distribution~$\phi$ with expectation~$\hat w \in \Delta_{d}$, covariance~$\Cov[\phi]$, differential entropy~$\Ent[\phi]$, and supported on rescaled Dikin ellipsoid~$\cE_{t-1, 1/2}(\hat w) := \{ w \in \Aff_d: \|w - \hat w\|_{\nabla^2 L_{t-1}(\hat w)} \le 1/2\}$. It is well known that~$\cE_{t-1,1}(w)$ is contained in~$\Delta_d$ for any~$w \in \Delta_d$~\cite{nesterov2003introductory}; as such, 
a natural idea is to approximate~$\phi_t$ by minimizing the upper bound for~$F_{t-1}[\phi]$, as per~\eqref{eq:sandwiching-v0}, over such distributions. 
This relaxation is indeed reasonable: the corresponding minimizer~$\bar\phi_t$ satisfies
\begin{equation}
\label{eq:Dikin-relaxation-tightness}
F_{t-1}[\wb\phi_{t}] 
\le \min_{\phi \; \in \; \Supp(\Delta_d)} F_{t-1}[\phi] + O(\mu d \log(T+\lambda d)).
\end{equation}
That is,~$F_{t-1}[\cdot]$ is minimized up to an error of the order of the desired regret.
This result, rigorously formulated as Proposition~\ref{prop:Dikin-relaxation-tightness}, is proved via self-concordance machinery and some volume estimates, the proof being rather technical; we believe it to be of independent interest.

\item[\proofstep{3}.]
We then focus on the relaxation of~\eqref{eq:variational-v0} described in step~\proofstep{2}, i.e.~replacing~$F_{t-1}[\phi]$ with its upper bound from~\eqref{eq:sandwiching-v0}, and imposing the constraint~$\phi \in \Supp(\cE_{t-1,c}(\hat w))$ for some~$c \le 1/2$. 
We relax this ``hard'' support constraint into a ``soft'' covariance constraint, namely
\[
\phi \in \Supp(\Aff_{d}), \quad \Cov[\phi] \preceq c^2 L_{t-1}(\hat w)^{-1}.
\]
As it happens, the minimum is attained on the~$\Aff_d$-marginal of the Gaussian distribution with covariance~$\nabla^2 L_{t-1}(\hat w)^{-1}$ and expectation~$w_{t}$---precisely as in~\eqref{def:VB-FTRL}. 
The correspondence between~\eqref{def:VB-FTRL} and (generalized) Cover's update is thus specified. 
Moreover, we show (see Proposition~\ref{prop:Gaussian-relaxation-tightness}) that the truncation of this Gaussian distribution to the appropriate Dikin ellipsoid around~$w_{t}$ is an approximate minimizer of~$F_{t-1}[\cdot]$ with accuracy~$O(\mu d \log(ed))$---in other words, the Gaussian relaxation does not lead to a further loss of accuracy beyond~\eqref{eq:Dikin-relaxation-tightness}. 
In our view, this result---just as~\eqref{eq:Dikin-relaxation-tightness}---may find uses beyond the context of portfolio selection.
\end{itemize}
We are now about to implement these steps. 
Before doing so, let us make one remark. 
One might expect that the above approximation result---the one representing~\eqref{def:VB-FTRL} update as an approximation of generalized Universal Portfolios with the ``right'' accuracy of order~$O(d\log(T))$---would directly lead to a shorter and/or simpler proof of Theorem~\ref{th:volumetric} than the one to be presented in Section~\ref{sec:volumetric-proof}, via some formal reduction to the regret bound for~$(\lambda,\mu)$-Universal Portfolios in Proposition~\ref{prop:cover-general-regret}.
However, we were unable to find such an alternative proof.
The actual proof of Theorem~\ref{th:volumetric} in~Section~\ref{sec:volumetric-proof} has very little to do (if anything at all) with the proofs of Propositions~\ref{prop:Dikin-relaxation-tightness}--\ref{prop:Gaussian-relaxation-tightness}; in particular, the former proof is non-probabilistic and does not use the variational formulation~\eqref{eq:variational-v0}.


\paragraph{Step \proofstep{1} :~$(\lambda,\mu)$-generalized Universal Portfolios update and its variational formulation.} 
Recall that Cover's original method~\eqref{def:universal-portfolio}--\eqref{def:exp-density} amounts to playing the expectation 
over the density proportional to
$
\exp \big(-\sum_{\tau \in [t-1]} \ell_{\tau}(w)\big),
$
$w \in \Delta_d$, at round~$t$. 
Now, consider the generalized strategy
%
\begin{align}
w_{t} = \E_{w \sim \phi_{t}}[w] 
\quad \text{where} \quad
\mathsf{\phi}_{t}(w) 
	\propto \exp \left(-\frac{1}{\mu} L_{t-1}(w)\right), \quad w \in \Delta_{d},
\label{def:cover-general}
\end{align}
where we introduced two parameters~$\lambda \ge 0$ and~$\mu \ge 1$, so that~\eqref{def:exp-density} corresponds to~$\lambda = 0$,~$\mu = 1$.  
Here, parameter~$\mu$ plays the role of temperature: the larger it is, the closer~$\phi_{t}$ is to the {\em uniform distribution on~$\Delta_d$}.
Meanwhile,~$\lambda$ biases~$\phi_{t}$ towards {\em the Dirac measure on~$\frac{1}{d} \ones_{d}$, the uniform portfolio.}
%
We first observe that the regret optimality of Cover's algorithm extends to the whole range of~$\lambda, \mu$. 

\begin{proposition}
\label{prop:cover-general-regret}
For any~$T \in \N$ and any market realization~$x_{1:T}$, the sequence of portfolios~$w_{1:T}$ constructed by~$(\lambda,\mu)$-generalized Universal Portfolios~\eqref{def:cover-general} with parameters~$\lambda \ge 0, \mu \ge 1$ satisfies:
\[
\Regret_T(w_{1:T}|x_{1:T}) \le \mu (d-1) \log(T+1) + \lambda (d-1) \log \left(4e \, \frac{T + \lambda d}{\lambda d}\right) + \mu - \lambda \log(d).
\]
\end{proposition}
The argument, given in Appendix~\ref{apx:cover-correspondence-proofs}, extends the proof of~\cite[Theorem 7]{hazan2007logarithmic} to allow for~$\lambda > 0$. 
The key observation is that~$\exp(-\frac{1}{\mu} \ell_t(w)) = (x_t^\top w)^{1/\mu}$ is a concave function whenever~$\mu \ge 1$. 
By Jensen's inequality, this leads, through a telescoping argument, to the bound as follows:
\begin{equation}
\label{eq:cover-jensen-bound}
\sum_{t \in [T]} \ell_t(w)
\le -\mu \log \left[ \int_{\Delta_d} \exp \left(-\frac{1}{\mu} L_T(w) \hspace{-0.0cm} \right) dw \right] + \mu \log \left[ \int_{\Delta_d} \exp \left(-\frac{\lambda}{\mu} R(w) \hspace{-0.0cm} \right) dw \right].
\end{equation}
The right-hand side is then related to~$\min_{w \in \Delta_d} \sum_{t \in [T]} \ell_t(w)$ through the rescaling technique from~\cite{blum1999universal}. 

For what is to follow, it is crucial to make a couple of observations. 
First,~$\phi_{t}$, as defined in~\eqref{def:cover-general}, is the unique optimal solution to the following optimization problem (see e.g.~\cite[Lemma~4.10]{van2016probability}):
\begin{equation}
\label{eq:universal-portfolios-gibbs}
\min_{\phi \, \in \, \Supp(\Delta_d)} \; F_{t-1}[\phi] := \E_{w\sim\phi}[L_{t-1}(w)] - \mu \Ent[\phi]
\end{equation}
where~$\Supp(\Delta_d)$ is the set of probability densities supported on~$\Delta_d$, and~$\Ent[\phi] := \E_{w \sim \phi}[-\log\phi(w)]$ is the differential entropy of~$\phi$. 
Moreover,~$F_{t-1}[\phi_t]$ is the (rescaled) negative log-partition function:
\[
F_{t-1}[\phi_t] = -\mu \log\int_{\Delta_d} \exp \left( -\frac{1}{\mu} L_{t-1}(w) \right) dw.
\]
In particular, we recognize the first term in the right-hand side of~\eqref{eq:cover-jensen-bound} as nothing else but~$F_T[\phi_{T+1}]$.

\paragraph{Step \proofstep{2} : restriction to the Dikin ellipsoid.}
Fix~$\hat w \in \ri(\Delta_d)$, and for~$r \ge 0$ and~$t \in [T]$ define 
\begin{equation}
\label{def:Dikin-ellipse}
\cE_{t-1,r}(\hat w) := \big\{w \in \Aff_d: \; \left\| w - \hat w \right\|_{\nabla^2 L_{t-1}(\hat w)} < r \big\},
\end{equation}
the $r$-Dikin ellipsoid~\cite{nesterov2003introductory} of~$L_{t-1}(\cdot)$ at~$\hat w$ restricted to hyperplane~$\Aff_d$, cf.~\eqref{def:affine-subspace}.
It can be routinely checked (see Appendix~\ref{apx:self-conc}) that when~$\lambda \ge 1$, the function~$L_{t-1}^{\res}(v) = L_{t-1}(Av + e_d)$ on~$\R^{d-1}$, cf.~\eqref{def:affine-reparametrization}, is self-concordant with domain~$\Int(\bDelta_{d-1})$. 
By~\cite[Theorems 4.1.5]{nesterov2003introductory}, the ``unit'' ellipse~$\cE_{t-1,\,1}(\hat w)$ is contained in~$\ri(\Delta_d)$. Moreover, on smaller ellipses one can control the remainder of the quadratic expansion of~$L_{t}$: from~\cite[Theorems 4.1.6--4.1.8]{nesterov2003introductory} it follows (see~Lemma~\ref{lem:SC-sandwich}) that
\begin{equation}
\label{eq:SC-upper-for-cover}
\begin{aligned}
\frac{1}{5} \left \| w - \hat w \right\|_{\nabla^2 L_{t-1}(\hat w)}^2
\le
L_{t-1}(w) - L_{t-1}(\hat w) - \nabla L_{t-1}(\hat w)^\top (&w - \hat w) 
\le 
\frac{4}{5} \left \| w - \hat w \right\|_{\nabla^2 L_{t-1}(\hat w)}^2\\
\text{for all} \; w \in \cE_{t-1,\,1/2}(&\hat w).
\end{aligned}
\end{equation}
Hence, any distribution~$\phi$ with mean~$\E_{\phi}[w] = \hat w \in \ri(\Delta_d)$ and supported in~$\cE_{t-1,\,1/2}(\hat w)$, satisfies
\begin{equation}
\label{eq:expected-quadratic-bound}
\frac{1}{5}\Tr\left( \nabla^2 L_{t-1}(\hat w) \, \Cov[\phi] \,\right) - \mu \Ent[\phi] 
\le 
F_{t-1}[\phi] - L_{t-1}(\hat w)
\le 
\frac{4}{5}\Tr\left( \nabla^2 L_{t-1}(\hat w) \, \Cov[\phi] \,\right) - \mu \Ent[\phi] 
\end{equation}
where~$\Cov[\phi]$ is the covariance matrix of~$\phi$.
The upper bound suggests us to approximate~\eqref{eq:universal-portfolios-gibbs} with
\begin{equation}
\label{eq:Dikin-upper-approx}
\begin{aligned}
\min_{
\scriptsize
	\begin{aligned} 
		&\hat w \in \Delta_d, \; \E_{\phi}[w] = \hat w, \\
	 	&\phi \in \Supp(\cE_{t-1,\,1/4}(\hat w))
	\end{aligned}
	} 
\underbrace{L_{t-1}(\hat w) + \frac{4}{5}\Tr\left( \nabla^2 L_{t-1}(\hat w) \, \Cov[\phi] \,\right) - \mu \Ent[\phi]}_{:= \; \wb F_{t-1}[\phi]}
\end{aligned} 
\end{equation}
---in other words, to focus on distributions supported in a small Dikin ellipsoid around the expectation. 
Now: since~$\cE_{t-1,\,1/4}(\hat w) \subset \Delta_d$, any optimal solution~$\wb\phi_{t}$ to~\eqref{eq:Dikin-upper-approx} is feasible in~\eqref{eq:universal-portfolios-gibbs}; hence~\eqref{eq:Dikin-upper-approx} overestimates~\eqref{eq:universal-portfolios-gibbs}. 
Moreover, the following result gives the accuracy of this upper approximation:
\begin{proposition}
\label{prop:Dikin-relaxation-tightness}
When~$\lambda \ge 1$, any optimal solution~$\wb\phi_{t}$ to~\eqref{eq:Dikin-upper-approx} satisfies the following inequalities: 
\[
F_{t-1}[\wb\phi_{t}] 
\le \wb F_{t-1}[\wb\phi_{t}] 
\le \min_{\phi \; \in \; \Supp(\Delta_d)} F_{t-1}[\phi] \; + \; 1.5 \mu (d-1) \log(T+\lambda d) + 3.9\mu(d+1) + 0.1.
\]
\end{proposition}
\noindent 
The proof of this result is relegated to Appendix~\ref{apx:cover-correspondence-proofs} and proceeds in two steps. 
First, we show that
\[
F_{t-1}[\wb\phi_{t}] \le 
\wb F_{t-1}[\wb\phi_{t}] \le
\min_{
\scriptsize
	\begin{aligned} 
		&\hat w \in \Delta_d, \; \E_{\phi}[w] = \hat w, \\
	 	&\phi \in \Supp(\cE_{t-1,\,1/2}(\hat w))
	\end{aligned}
	} 
	F_{t-1}[\phi] \;\; + \mu (d-1)\log(2).
\]
Here the first inequality is trivial, and the second one follows by comparing the covariance and entropy for the random variables~$w \sim \phi$ and~$\hat w + \half(w-\hat w)$.
The second step consists in showing that~$(\lambda,\mu)$-generalized Universal Portfolios distribution~$\phi_{t}$, as defined in~\eqref{def:cover-general}, can be truncated to the Dikin ellipsoid around its mode at the expense of an~$O(\mu (d-1) \log(T+\lambda d))$ increase in~$F_{t-1}$. 
This is done by expressing~$F_{t-1}[\phi_t]$ as the log-partition function, controlling the integrand via self-concordance, and lower-bounding the volume of the Dikin ellipsoid by arguing that the mode of~$\phi_{t}$ cannot be too close to the relative boundary of~$\Delta_d$ thanks to the logarithmic-barrier regularizer.

\paragraph{Step~\proofstep{3} : reduction to the Gaussian distribution.}
Departing from~\eqref{eq:Dikin-upper-approx}, we note that its objective~$\wb F_{t-1}[\phi]$ depends on~$\phi$ only through the mean, covariance, and differential entropy of~$\phi$.
However, the support constraint, namely~$\phi \in \Supp(\cE_{t-1,\,1/4}(\hat w))$, interferes with higher moments.
A natural idea is then to replace the ``hard'' support constraint with a ``soft'' covariance one, namely
\begin{equation}
\label{eq:cov-constraints}
\begin{aligned}
\phi \in \Supp(\Aff_d), \quad
\Cov[\phi] \preceq \text{{\footnotesize$\frac{1}{16}$}} \nabla^2 L_{t-1}(\hat w)^{-1}.
\end{aligned}
\end{equation}
This results in a relaxation of~\eqref{eq:Dikin-upper-approx}. 
Indeed,~$\E_{w\sim\phi}[w] = \hat w \in \Delta_d$ and~$\phi \in \Supp(\cE_{t-1,\,1/4}(\hat w))$ together imply that~$\ones_d^\top (w - \hat w) = 0$ ($\phi$-a.s.), that is~$\phi \in \Supp(\Aff_d)$; 
on the other hand, by Jensen's inequality
\begin{equation}
\label{eq:soft-feasibility-check}
\left\| \nabla^2 L_{t-1}(\hat w)^{1/2} \Cov[\phi] \, \nabla^2 L_{t-1}(\hat w)^{1/2} \right\| 
\le \E_{w \sim \phi} \left[ \left\| w - \hat w \right\|_{\nabla^2 L_{t-1}(\hat w)}^2 \right] \le \text{{\footnotesize$\frac{1}{16}$}}.
\end{equation}
We can now satisfy the support constraint in~\eqref{eq:cov-constraints} ``automatically'' by reformulating it as follows:
\begin{equation}
\label{def:marginal-correspondence}
w \sim \phi \;\; \text{is such that} \;\; w = Av + e_d \;\; \text{for} \;\; v \sim \phi^{\res} \in \Supp(\R^{d-1})
\end{equation}
where the mean~$\hat v \in \bDelta_{d-1}$, covariance, and entropy of~$\phi^\res$ are related to those of~$\phi$ according to
\[
\begin{aligned}
\hat w = A \hat v + e_d, \quad 
\Cov[\phi] = A \Cov[\phi^\res] A^\top, \quad
\Ent[\phi] = \Ent[\phi^\res] - \tfrac{1}{2}\log\det(A^\top A) 
\end{aligned}
\]
where~$\det(A^\top A) = d$.
The matrix constraint in~\eqref{eq:cov-constraints} translates to 
$
\Cov[\phi^{\res}] \preceq \frac{1}{16} (A^\top \nabla^2 L_{t-1}(\hat w) A)^{-1} 
$
by simple linear algebra.
As such, in terms of~$\hat w$,~$\Cov[\phi^\res]$, and~$\Ent[\phi^\res]$, the objective in~\eqref{eq:Dikin-upper-approx} becomes
\begin{equation}
\label{eq:Dikin-upper-approx-obj-reformulated}
\wb F_{t-1}[\phi] = L_{t-1}(\hat w) + \frac{4}{5} \Tr \left[ \Cov[\phi^\res] \, A^\top \nabla^2 L_{t-1}(\hat w) A \right] - \mu \Ent[\phi^\res].
\end{equation}
Clearly,~$\cN(\hat v, \bSigma)$ maximizes~$\Ent[\phi^{\res}]$ given the constraints~$\E_{\phi^\res}[v] = \hat v$ and~$\Cov[\phi^\res] = \bSigma$, with the value
$\Ent[\cN(\hat v, \bSigma)] = \frac{1}{2} \log\det(\bSigma) + c_d$ where~$c_d$ depends only on~$d$.
This gives a relaxation of~\eqref{eq:Dikin-upper-approx},
\begin{equation}
\label{eq:relaxation-post-max-ent}
\min_{
	\scriptsize
	\begin{aligned}
		\hat w \in \Delta_{d}, \quad 
		0 \preceq \bSigma \preceq \frac{1}{16}^{\vphantom{A^A}} (A^\top \nabla^2 L_{t-1}(\hat w)A)^{-1}
	\end{aligned}		
	}  
L_{t-1}(\hat w) + \frac{4}{5} \Tr\left[ \bSigma A^\top \nabla^2 L_{t-1}(\hat w) A \right] - \frac{\mu}{2} \log\det(\bSigma),
\end{equation}
that depends on~$\phi$ only through the mean and covariance.
Finally, in~Appendix~\ref{app:gaussian-reduction-solve-in-sigma} we show that, when~$\mu > 0.1$, the minimum over~$\bSigma$ is attained at~$\hat \bSigma = \frac{1}{16} (A^\top \nabla^2 L_{t-1}(\hat w) A )^{-1}$, and~\eqref{eq:relaxation-post-max-ent} reduces to
\begin{equation}
\label{eq:vb-ftrl-derived-from-cover}
\min_{
	\hat w \in \Delta_d
	}  
L_{t-1}(\hat w) + \frac{\mu}{2} \log\det(A^\top \nabla^2 L_{t-1}(\hat w) A).
\end{equation}
This is nothing else but~\eqref{def:VB-FTRL}.
To summarize,~\eqref{def:VB-FTRL} appears after relaxing the ``hard'' support constraint in~\eqref{eq:Dikin-upper-approx}, so that the optimum is attained on a Gaussian distribution.
Moreover: the result below shows that, essentially, this relaxation comes at no extra cost (cf.~Proposition~\ref{prop:Dikin-relaxation-tightness}).


\begin{proposition}
\label{prop:Gaussian-relaxation-tightness}
Assume that~$\mu > 0.1$. Let~$\wb\phi_{t}$ be optimal in~\eqref{eq:Dikin-upper-approx},~$w_{t}$ be as in~\eqref{def:VB-FTRL},
~$\bSigma_{t} = {\frac{1}{16}} (A^\top \nabla^2 L_{t-1}(w_{t}) A)^{-1}$,
and let~$g_{t}^\trc$ be the truncation of~$\cN(w_{t}, A\bSigma_{t}A^\top)$ to~$\cE_{t-1,\,1/2}(w_{t})$. 
Then
\[
F_{t-1}^{\vphantom\trc}[g_{t}^\trc] 
\le 
\wb F_{t-1}^{\vphantom\trc}[\wb\phi_{t}^{\vphantom\trc}] + 0.5 \mu (d-1) \log(350\hspace{0.03cm}d) + 0.01\mu. 
\]
\end{proposition}
\noindent The proof of this result relies on the feasibility of~$\wb\phi_t$ in~\eqref{eq:relaxation-post-max-ent} and an estimate of the decrease of entropy when truncating a Gaussian distribution to its covariance ellipsoid (see Lemma~\ref{lem:Gaussian-trunc-entropy}).
\vspace{0.1cm}

Finally, by combining Propositions~\ref{prop:Dikin-relaxation-tightness} and~\ref{prop:Gaussian-relaxation-tightness} we establish a quantitative correspondence between~\eqref{def:VB-FTRL} and generalized Cover's update in its variational form~\eqref{eq:variational-v0}. 
Namely, portfolio~$w_t$ computed in~\eqref{def:VB-FTRL} generates a distribution---precisely,~$\cN(w_t, \frac{1}{16}A(A^\top \nabla^2 L_{t-1}(w_t)^{-1} A)^{-1}A^\top)$ truncated to the Dikin ellipsoid~$\cE_{t-1,\,1/2}(w_t)$---that solves~\eqref{eq:variational-v0} up to the accuracy of~$O(\mu d \log(T))$.

\section{Regret analysis \odima{for~\OurAlgo{}}}
\label{sec:volumetric-proof}

In this section, we start by giving an outline of the proof of Theorem~\ref{th:volumetric} and discussing the role of volumetric regularization. 
After that, we are going to present the full proof of the theorem.
\subsection{Overview of the proof and the role of volumetric regularization}
\paragraph{Key decomposition.}
We shall specify the values of~$\lambda, \mu$ in the later stages of the proof; for now we consider general~$\lambda, \mu > 0$. 
Let us decompose the regret of the sequence generated in~\eqref{def:VB-FTRL} through telescoping:
\begin{align}
\Regret_T(w_{1:T}|x_{1:T}) 
        &\le \sum_{t \in [T]} \big[\ell_t(w_t) + \min_{w \in \Delta_d} P_{t-1}(w) - \min_{w \in \Delta_d} P_t(w)\big] - \min_{w \in \Delta_d} \sum_{t \in [T]} \ell_t(w) + \min_{w \in \Delta_d} P_T(w) \notag\\
        &=
            \sum_{t \in [T]} \big[ P_t(w_t) + \mu V_{t-1}(w_t) - \mu V_{t}(w_t) - P_t(w_{t+1}) \big]
			 - \min_{w \in \Delta_d} \sum_{t \in [T]} \ell_t(w)
            + \min_{w \in \Delta_d} P_T(w). 
\label{eq:regret-decomp-pre}
\end{align}
Here the inequality holds since~$P_0(w) = \lambda R(w) \ge 0$ on~$\Delta_d$, cf.~\eqref{def:log-barrier}.
In other words, we have that
\begin{equation}
\label{eq:regret-decomp}
\Regret_T(w_{1:T}|x_{1:T}) 
\le \Bias_T + \sum_{t \in [T]} [\Lag_t + \Off_t]
\end{equation}
where the three kinds of differences~$\Bias_T,\Lag_t,\Off_t$ are defined as follows:
\begin{equation*}
\hspace{-0.05cm}
\begin{array}{c}
\Lag_t := P_t(w_t) \hspace{-0.05cm} - \hspace{-0.05cm} P_t(w_{t+1}),
\;\;\,
\Off_t := \mu [V_{t-1}(w_t) \hspace{-0.05cm} - \hspace{-0.05cm} V_{t}(w_t)], 
\displaystyle
\;\;\,
\Bias_T := \min_{w \in \Delta_d} P_T(w) \hspace{-0.05cm} - \hspace{-0.07cm} \min_{w \in \Delta_d} \sum_{t \in [T]} \ell_t(w).
\end{array}
\end{equation*}
Here,~$\Bias_T$ is the ``bias'' term arising since~\eqref{def:VB-FTRL} minimizes~$P_{t-1}(w)$ instead of~$\sum_{\tau \in [t-1]} \ell_\tau(w)$, the observed sum of losses---in other words, due to our use of regularization ``{per se.}'' 
$\Lag_t$ is the ``lagged'' term that arises since we only ``follow'' the (regularized) leader, i.e.~minimize~$P_{t-1}$ instead of~$P_t$.
Finally,~$\Off_t$ appears due to our use of a {\em time-varying} volumetric regularizer.
Since~$P_{t-1}$ in~\eqref{def:VB-FTRL} includes~$V_{t-1}$ rather than~$V_{t}$, we expect~$\Off_t$ to be {\em negative} and {\em decrease} the regret.

\paragraph{Big picture.} 
Our proof proceed as follows.
On the one hand, we shall bound the ``bias'' term as
\begin{equation}
\label{eq:bias-bound}
\Bias_T
\le (\lambda + 2\mu) \, (d-1) \log(T + \lambda d)
\end{equation}
under the assumption that~$\lambda \ge 2e$. 
This is not too involved, and can be done via a combination of standard techniques. 
Way more interesting is another step, in which we show that the inequalities 
\begin{equation}
\label{eq:invariant}
\Lag_t + \Off_t  \le 0, \quad \forall t \in [T],
\end{equation}
hold when~$\lambda,\mu$ are as in the premise of the theorem. 
The result then follows by combining~\eqref{eq:regret-decomp}--\eqref{eq:invariant}.

\paragraph{Volumetric regularization mechanism.}
First, concavity of~$\log\det(\cdot)$ on~$\Sym_{+}^{n}$ implies that
\begin{equation}
\label{eq:offset-bound-sketch}
\Off_t 
\le -\frac{\mu}{2} \pi_t(w_t) \quad \text{with} \quad
\pi_t(w) := \left\| A^\top \nabla\ell_{\tau}(w) \right\|_{ (A^\top \nabla^2 L_t(w) A)^{-1}}^2,
\end{equation}
see Lemma~\ref{lem:upper_bound_trace}.
Under the affine reparametrization in~\eqref{def:affine-reparametrization},~$\pi_t(w_t)$ becomes the leverage score of~$x_t$ at~$w_t$, cf.~\eqref{def:affine-reparametrization-grad-and-hess}. 
On the other hand,~$\Lag_{t}$ is the suboptimality gap of~$w_t$ in terms of~$P_t^{\vphantom{\res}}$ or~$P_t^\res$:
\[
\Lag_{t} = P_t(w_t) - \min_{w \in \Delta_d} P_t(w) = P_t^\res(v_t) - \min_{w \in \Delta_d} P_t^\res(v).
\] 
Now, self-concordance of~$P_t^\res$ allows to bound~$\Lag_{t}$ in terms of the squared Newton decrement:
\begin{equation}
\label{eq:SC-decrement-bound-sketch}
\Lag_{t} \le C \Dec_t^2 \quad \text{with} \quad \Dec_t = \left\|A^\top \nabla P_t(w_t)\right\|_{(A^\top\nabla^2 P_t(w_t) A)^{-1}},
\end{equation}
for some~$C > 0$, once we have shown that~$\Dec_t^2 \le c < 1$ (see Lemma~\ref{lem:bound-minimum-SC} or~\cite[Theorem~4.1.13]{nesterov2003introductory}). 
Finally, the optimality condition for~$w_t$ implies that
\[
A^\top \nabla P_t(w_t) = A^\top \nabla \ell_t(w_t) + \mu A^\top [\nabla V_t(w_t) - \nabla V_{t-1}(w_t)].
\]
Departing from this identity and exploting the algebraic structure of leverage scores, we manage to control the contribution of the term~$\mu A^\top [\nabla V_t(w_t) - \nabla V_{t-1}(w_t)]$ to~$\Dec_t^2$, as~$\mu$ increases, by simultaneously increasing~$\lambda$. More precisely, we show that
$
\Dec_t^2
\le \frac{1+\lambda}{\lambda} \big(1+\frac{2\mu}{\lambda}\big)^2 \pi_t(w_t), 
$
while~$\pi_t(w_t) \le \frac{1}{1+\lambda}$ by Lemma~\ref{lem:bound_leverage_score}. 
Thus, selecting~$\lambda$ appropriately we can enforce~$\Dec_t^2 \le c < 1$ for {\em any}~$\mu > 0$. 
But then, combining the bound on~$\Dec_t^2$ with~\eqref{eq:offset-bound-sketch}--\eqref{eq:SC-decrement-bound-sketch} and selecting~$\mu$ appropriately, we enforce~\eqref{eq:invariant}.

\begin{remark} 
In the actual proof, instead of~$\Dec_t$ we focus on the Newton decrement~$\uDec_t$ of 
\[
\wbb P_t^\res(v) = L_t^\res(v) + \mu\big[V_t^\res(v_t) + \nabla V_t^\res(v_t)^\top(v-v_t) \big],
\]
a lower bound for~$P_t^\res(v)$ tight at~$v = v_t$. 
As it turns out, while~$\uDec_t$ is somewhat larger than~$\Dec_t$, focusing on~$\uDec_t$ allows to avoid explicit use of the self-concordance of~$P_t^\res$, using self-concordance of~$L_t^\res$ instead.
(Of course, we use the differential properties of~$V_t$, in particular the algebraic structure of its derivatives, when bounding~$\uDec_t$.)
The point is that~$P_t^\res$ is only~$21$-self-concordant (see Lemma~\ref{prop:SC-of-P}), rather than~$1$-self-concordant as~$L_t^\res$, so working with~$\uDec_t, \wbb P_t^\res$ instead of~$\Dec_t, P_t^\res$ allows to reduce the constants.
Meanwhile, in the proof of Theorem~\ref{th:volumetric-newton} on the regret of~\eqref{def:VB-FTRL-Newton}, a procedure implementing~\eqref{def:VB-FTRL}, 
we have to use self-concordance of~$P_t^\res$ directly because of a quasi-Newton method applied to~$P_t^\res$; the resulting constant factor is an order of magnitude larger.
\end{remark}



\subsection{Proof of Theorem~\ref{th:volumetric}}

\proofstep{1}.
Our goal in this step is to establish~\eqref{eq:bias-bound}. 
Let us further decompose the ``bias'' term as
\begin{equation}
\label{eq:bias-decomp}
\Bias_T = \VolBias_T + \LogBias_T
\end{equation}
where the two terms correspond to the two regularizers~$V_T(w)$ and~$R(w)$ respectively:
\[
\begin{aligned}
\VolBias_T := \min_{w \in \Delta_d} P_T(w)- \min_{w \in \Delta_d}  L_T(w),
\quad \quad
\LogBias_T := \min_{w \in \Delta_d} L_T(w) - \min_{w \in \Delta_d} \sum_{t \in [T]} \ell_t(w). 
\end{aligned}
\]
Defining~$w_T^\star := \argmin_{w \in \Delta_d} L_T (w)$ we then observe that
\begin{align}
\VolBias_T \le P_T^{\vphantom\star}(w_T^\star) - L_T^{\vphantom\star}(w_T^\star) 
= \mu V^{\vphantom\star}_T(w_T^\star)
&\le \frac{\mu}{2} (d-1) \log(\|A^\top \nabla^2 L_T^{\vphantom\star}(w_T^\star) A\|) \notag\\
&\le \frac{\mu}{2} (d-1) \log(d\|\nabla^2 L_T^{\vphantom\star}(w_T^\star)\|)
\label{eq:bias-V|R-first}
\end{align}
where in the last step we used that~$\|A\| = \sqrt{d}$. 
On the other hand, we can estimate~$\|\nabla^2 L_T^{\vphantom\star}(w_T^\star)\|$ using that the entries of~$w_T^\star$ cannot be too small. 
Indeed, according to~Lemma~\ref{lem:entries-cutoff} we have that
\begin{equation}
\label{eq:entries-cutoff}
w_T^\star[i] \ge \frac{\lambda}{T + \lambda d}, \quad \forall i \in [d]
\end{equation}
---a variation of the well-known result for the logaritmic-barrier regularizer (see~e.g.~\cite{luo2018efficient}).
But then
\[
\begin{aligned}
\nabla^2 L_T^{\vphantom\star}(w_T^\star) 
= \sum_{t \in [T]} \frac{x_t^{\vphantom\top} x_t^\top}{(x_t^\top w_T^\star)^2} + \sum_{i \in [d]} \frac{\lambda e_i^{\vphantom\top} e_i^\top}{w_T^\star[i]^2} 
\preceq \frac{(T+\lambda d)^2}{\lambda^2} 
        \Bigg( \sum_{t \in [T]} \frac{x_t^{\vphantom\top} x_t^\top}{\|x_t\|_1^2} + \lambda \sum_{i \in [d]}  e_i^{\vphantom\top} e_i^\top \Bigg) 
\preceq \frac{(T+\lambda d)^3}{\lambda^2} I_d,
\end{aligned}
\]
that is
\begin{equation}
\label{eq:Hessian-bound-from-entries-magnitude}
\|\nabla^2 L_T^{\vphantom\star}(w_T^\star) \| \le \frac{(T+\lambda d)^3}{\lambda^2}.
\end{equation} 
Returning to~\eqref{eq:bias-V|R-first} and using that~$\lambda \ge 1$ to simplify the argument of the logarithm, we find that
\begin{equation}
\label{eq:bias-V|R}
\VolBias_T \le 2\mu (d-1) \log(T + \lambda d).
\end{equation}
Now, in order to bound~$\LogBias$ take any~$w_T^o \in \Argmin_{w \in \Delta_d} \sum_{t \in [T]} \ell_t(w)$ and define~$w_{T,\alpha}^o \in \Delta_d$ by
\[
w_{T,\alpha}^o = (1-\alpha) w_T^o + \frac{\alpha}{d} \ones_d
\quad \text{with} \quad
\alpha = \frac{\lambda (d-1)}{T + \lambda(d-1)}.
\]
On the one hand, we have that
\[
\begin{aligned}
R(w_{T,\alpha}^o) 
\le \max_{w \in \Delta_d} R \left((1-\alpha) w + \frac{\alpha}{d} \ones_d \right)  
= (d-1) \log \left(\frac{d}{\alpha} \right) 
\le (d-1) \log \left( 2 \frac{T + \lambda d}{\lambda} \right)
\end{aligned}
\] 
where we used that a convex function with a compact domain is maximized at an extremal points, and also that~$\frac{d}{d-1} \le 2$ \odima{for~$d \ge 2$.}
On the other hand, by Lemma~\ref{lem:cutoff-regret} proved in appendix,~$\forall w \in \Delta_d$
\[
\sum_{t \in [T]} \ell_{t}\left((1-\alpha)w + \frac{\alpha}{d} \ones_d\right)  - \ell_{t}(w) 
\le \frac{\alpha}{1-\alpha} T 
= \lambda(d-1), 
\]
whence
$
\sum_{t \in [T]} \ell_t(w_{T,\alpha}^o) - \ell_t(w_T^o) \le \lambda (d-1).
$
Combining these two results, we conclude that
\begin{equation}
\label{eq:bias-R}
\LogBias_T
\le L_T^{\vphantom\star}(w_{T,\alpha}^o) - \sum_{t \in [T]} \ell_t(w_T^o) 
\le \lambda R^{\vphantom\star}(w_{T,\alpha}^o) + \sum_{t \in [T]} \ell_t(w_{T,\alpha}^o) - \ell_t(w_T^o)
\le \lambda (d-1) \log \left( 2e \frac{T + \lambda d}{\lambda} \right).
\end{equation}
We arrive at the desired estimate, cf.~\eqref{eq:bias-bound}, by plugging~\eqref{eq:bias-V|R} and~\eqref{eq:bias-R} into~\eqref{eq:bias-decomp} and using that~$\lambda \ge 2e$.

\proofstep{2}.
Next we turn to proving~\eqref{eq:invariant}. To streamline the exposition, we define~$x_{-i} = e_i$ for~$i \in [d]$, let~$\tau$ vary over the extended index set~$[t]^+ := [t] \cup \{-1, \,\dots, -d\}$, and define the abridged notation:
\begin{equation}
\label{def:abridged-hess-and-grad}
\begin{aligned}
\Hmtx_t(w) &:= A^\top \nabla^2 L_t(w) A, \quad &\bNabla_\tau(w) &:= A^\top \nabla \ell_\tau(w), \\
\hat\Hmtx_t &:= \Hmtx_t(w_t), \quad &\hat\bNabla_\tau &:= \bNabla_\tau(w_t).
\end{aligned}
\end{equation}
Moreover, with a slight abuse of notation we let
\begin{equation}
\label{def:lambda-seq}
\lambda_\tau := \lambda^{\ind\{\tau < 0\}}, \quad \tau \in [t]^+, 
\end{equation}
so that
$
L_t(w) = \sum_{\tau \in [t]^+} \lambda_{\tau} \ell_{\tau}(w)
$
and, as a consequence,
\begin{equation}
\label{eq:hess-as-a-sum}
\Hmtx_t(w) = \sum_{\tau \in [t]^+} \lambda_\tau^{\vphantom\top} \bNabla_\tau^{\vphantom\top}(w) \bNabla_\tau^\top(w), 
\quad\quad
\hat\Hmtx_t = \sum_{\tau \in [t]^+} \lambda_\tau^{\vphantom\top} \hat\bNabla_\tau^{\vphantom\top} \hat\bNabla_\tau^\top. 
\end{equation}
Finally, we define the Gram matrix~$\bPi(w) \in \Sym^{t+d}_{+}$ of the extended dataset, with entries given by
\begin{equation}
\label{def:Gram-entries}
\pi_{\tau,\nu}(w) := \sqrt{\lambda_\tau \lambda_\nu} \big\langle \bNabla_\tau(w), \bNabla_\nu(w) \big\rangle_{\Hmtx_t(w)^{-1}}, \quad \forall \tau,\nu \in [t]^+,
\end{equation}
rows/columns indexed on~$[t]^+$, and dependence on~$t$ omitted for brevity.
Also, let~$\hat \bPi := \bPi(w_t)$, i.e.
\begin{equation}
\label{def:Gram-hat-entries}
\hat\pi_{\tau,\nu} := \sqrt{\lambda_\tau \lambda_\nu} \big\langle \hat\bNabla_\tau, \hat\bNabla_\nu \big\rangle_{\hat\Hmtx_t^{-1}}, \quad \forall \tau,\nu \in [t]^+.
\end{equation}
For brevity, we refer to the diagonal entries with a single index:~$\pi_\tau(w) := \pi_{\tau,\tau}(w)$ and~$\hat\pi_\tau := \hat\pi_{\tau,\tau}$.
When~$w \in \Delta_{d}$,~$\pi_{\tau}(w)$ is the squared leverage score for the reparametrized and $\sqrt{\lambda_\tau}$-rescaled losses: 
\[
\pi_{\tau}(Av + e_d) = \lambda_{\tau}^{\vphantom\res} \| \nabla\ell_{\tau}^\res(v) \|_{\nabla^2 L_t^\res(v)^{-1}}^2.
\]
In fact,~$\bPi(w)$ is a projection matrix, i.e.~$\bPi(w) = \bPi(w)^2$; see~\cite{vaidya1989new}. 
We will use a weaker property:
\begin{equation}
\label{eq:projection-property}
\pi_{\tau}(w) = \sum_{\nu \in [t]^+} \pi_{\tau,\nu}(w)^2, \quad \forall \tau \in [t]^+.
\end{equation}
For the paper to be self-contained, this and other properties of~$\bPi(w)$ are proved in Appendix~\ref{apx:derivatives}.

\proofstep{3}.
Now, in terms of the notation we have just introduced, for each term~$\Off_t$ we have:
\begin{equation}
\label{eq:offset-bound}
\begin{aligned}
\Off_t 
= \frac{\mu}{2} \log \left( \frac{\det \left(\hat\Hmtx_t^{\vphantom\top} - \hat\bNabla_t^{\vphantom\top} \hat\bNabla_t^\top \right)}{\det \big( \hat\Hmtx_{t} \big)} \right)
\le -\frac{\mu}{2} \Tr \left[ \hat\bNabla_t^{\vphantom\top} \, \hat\bNabla_t^\top \hat\Hmtx_t^{-1} \right] 
= -\frac{\mu}{2} \hat\pi_t.
\end{aligned}
\end{equation}
Here the inequality relies on the fact that~$\log\det(\cdot)$ is a convex function on~$\Sym^n_{+}$ (see Lemma~\ref{lem:upper_bound_trace}). 
As such, to verify~\eqref{eq:invariant} it suffices to show that, for the values of~$\lambda,\mu$ as in the premise of the theorem,
\begin{equation}
\label{eq:lag-estimate}
\Lag_t 
\le \frac{\mu}{2} \hat\pi_t.
\end{equation}
%
To that end, let us define~$\wbb P_t: \R^{d}_{++} \to \R$ as follows:
\begin{equation}
\label{eq:potential-affine-lower-bound}
    \wbb P_t(w) = L_t(w) + \mu\big[V_t(w_t) + \nabla V_t(w_t)^\top(w-w_t) \big].
\end{equation}
From~\cite{vaidya1989new} it is known that~$V_t$ is convex. 
(For the paper to be self-contained, we reprove this result in Lemma~\ref{lem:volumetric-grad-and-hess}.) 
As such,~$\wbb P_t(w) \le P_t(w)$ for all~$w$. Since~$\wbb P_t(w_t) = P_t(w_t),$ we obtain the estimates
\begin{equation}
\label{eq:lagged}
\Lag_t
\le \wbb P_t(w_t) - \wbb P_t(w_{t+1}) 
\le \wbb P_t(w_t) - \min_{w \in \Delta_d} \wbb P_t(w).
\end{equation}
Next, let~$\wbb P_t^\res: \R^{d-1} \to \wb\R$ be the affine reparametrization (cf.~\eqref{def:affine-reparametrization}) of~$\wbb P_t$ restricted to~$\Aff_d$, namely:
\begin{equation}
\label{eq:reparametrization}
\wbb P_t^\res(v) := \wbb P_t(Av + e_d).
\end{equation}
Clearly,~$\wbb P_t^\res$ has~$\Int(\bDelta_{d-1})$, cf.~\eqref{def:solid-simplex}, as its domain, and~$v \mapsto Av+e_d$ is a bijection from~$\Delta_d$ to~$\bDelta_{d-1}$. 
As a result,~$\min_{w \in \Delta_{d}} \wbb P_t(w) = \min_{v \in \bDelta_{d-1}} \wbb P_t^\res(v)$; recalling~\eqref{eq:lagged}, this results in
\begin{equation}
\label{eq:lagged-gap-linear}
\Lag_t 
\le \wbb P_t^\res(v_t) - \min_{v \in \bDelta_{d-1}} \wbb P_t^\res(v).
\end{equation}
Function~$\wbb P_t^{\res}(\cdot)$ is standard self-concordant in the sense of~\cite[Section 4.1]{nesterov2003introductory}, or~$1$-self-concordant in terms of our Definition~\ref{def:SC-function} in Appendix~\ref{apx:self-conc}. (The reader may consult Appendix~\ref{apx:self-conc} for the essential results on self-concordance.)
As such, upper-bounding its Newton decrement
\begin{equation}
\label{eq:newton-decrement}
\uDec_t := \left\|\nabla \wbb P_t^\res(v_t)\right\|_{{\nabla^2 \wbb P_t^\res(v_t)}^{-1}}
\end{equation}
with some~$c < 1$ would suffice for showing that
$
\wbb P_t^\res(v_t) - \min_{v \in \bDelta_{d-1}} \wbb P_t^\res(v) 
\le C \uDec_t^2,
$
where~$C$ depends only on~$c$; in particular, Lemma~\ref{lem:bound-minimum-SC} shows that~$c = 0.5$ allows for~$C = 0.8$.
Therefore, it remains to prove that
$
\uDec_t \le 0.5,
$
while lower-bounding the right-hand side of~\eqref{eq:lag-estimate} with~$0.8\,\uDec_t^2$.

\proofstep{4}.
Since~$\nabla^2 \wbb P_t^\res(v) = A^\top \nabla^2 \wbb P_t(Av+e_d) A$, cf.~\eqref{def:affine-reparametrization-grad-and-hess}, and~$\nabla^2 \wbb P_t(w_t) = \nabla^2 L_t(w_t)$, cf.~\eqref{eq:potential-affine-lower-bound}, one has
\begin{equation}
\label{eq:A-hessian}
\begin{aligned}
\nabla^2 \wbb P_t^\res(v_t) = \hat\Hmtx_t.
\end{aligned}
\end{equation}
Meanwhile, since~$v_t = \argmin_{v \in \bDelta_{d-1}} P_{t-1}^{\res}(v)$ belongs to~$\Int(\bDelta_{d-1})$, we get
$
\nabla P_{t-1}^\res(v_t^{\vphantom\res}) = 0.
$
Defining
\begin{equation}
\label{def:volumetric-increment}
D_t(w) := V_{t}(w) - V_{t-1}(w) 
\end{equation}
we thence arrive at
\begin{align}
\nabla \wbb P_t^\res(v_t^{\vphantom\res})
\stackrel{\eqref{eq:potential-affine-lower-bound}}{=} \nabla P_t^\res(v_t^{\vphantom\res})
= \nabla P_t^\res(v_t^{\vphantom\res}) - \nabla P_{t-1}^\res(v_t^{\vphantom\res})
&= A^\top \left[\nabla P_t(w_t) - \nabla P_{t-1}(w_t)\right] \notag\\
&= A^\top \left[\nabla L_t(w_t) - \nabla L_{t-1}(w_t) + \mu \nabla D_t(w_t) \right] \notag\\
&= \hat\bNabla_{t} + \mu A^\top  \nabla D_t(w_t).
\label{eq:A-gradient}
\end{align}
Plugging~\eqref{eq:A-hessian} and~\eqref{eq:A-gradient} into~\eqref{eq:newton-decrement} we conclude that
\begin{equation}
\label{eq:dec-projected}
\uDec_t^2 = \left\| \hat\bNabla_t + \mu A^\top \nabla D_t(w_t) \right\|_{\hat\Hmtx_t^{-1}}^2.
\end{equation}
Let us now focus on~$\nabla D_t(w_t)$. 
Observe that, for any~$w \in \R^d_{++}$,
%
\begin{equation}
\label{eq:volumetric-increment-alt-form}
    -D_t(w)
    = \frac{1}{2}\log \det \left( \Imtx_{d-1}  - \bNabla_t(w) \bNabla_t(w)^\top \Hmtx_{t}(w)^{-1} \right)
    = \frac{1}{2}\log \left( 1 - \pi_t(w) \right)
\end{equation}
where we used that~$\bNabla_t(w) \bNabla_t(w)^\top$ is rank-one.
%
%
Whence by the composition formula:
\begin{equation}
\label{eq:xi-grad}
\nabla D_t(w) = \frac{1}{2(1 - \pi_{t}(w))} \nabla \pi_{t}(w).
\end{equation}
Finally, a rather tedious computation allows to differentiate~$\pi_t(w)$ explicitly (see~Lemma~\ref{lem:leverage-score-grad}):
\begin{equation}
\label{eq:sigma-grad}
\frac{1}{2} \nabla \pi_{t}(w) 
= \pi_t(w) \nabla \ell_t(w) - \sum_{\tau \in [t]^+} \pi_{t,\tau}(w)^{2} \nabla \ell_\tau^{\vphantom 2}(w).
\end{equation}
Collecting~\eqref{eq:dec-projected}--\eqref{eq:sigma-grad} 
we arrive at the following estimate:
\begin{equation}
\label{eq:decrement-bound-pre}
\begin{aligned}
\uDec_t^2
\le 
    \frac{1}{(1 - \hat\pi_t)^2}
    \left\| 
    	\left[ 1 + (\mu -1)\hat\pi_{t}\right] \hat\bNabla_t
		- \mu \left(  \sum_{\tau \in [t]^+} {\hat\pi_{t,\tau}}^2 \, \hat\bNabla_\tau^{\vphantom 2} \right)
    \right\|_{\hat\Hmtx_t^{-1}}^2.
\end{aligned}
\end{equation}

\proofstep{5}.
Let~$\hat\alpha_t := 1 + (\mu - 1)\hat\pi_{t}$, then developing the square in the right-hand side of~\eqref{eq:decrement-bound-pre} results in
\begin{equation}
\label{eq:square-expanded}
\begin{aligned}
\uDec_t^2
\le \frac{\hat\alpha_t^2 \hat\pi_{t}  + 2 \mu \hat\alpha_t |\hat\sE_1| + \mu^2 \hat\sE_2}{(1 - \hat\pi_t)^2} 
\end{aligned}
\end{equation}
with~$\hat\sE_1$ and~$\hat\sE_2$ as follows (cf.~\eqref{def:lambda-seq}--\eqref{def:Gram-hat-entries}):
\begin{align}
\label{eq:E1-def}
\hat\sE_1 
&\;\left[ := \sum_{\tau \in [t]^+} \big\langle \hat \bNabla_t, \hat \bNabla_\tau \big\rangle_{\hat\Hmtx_t^{-1}} \; \hat\pi_{t,\tau}^2 \right]\;
= \sum_{\tau \in [t]^+} \frac{1}{\sqrt{\lambda_{\tau}}} \hat\pi_{t,\tau}^3, \\
\label{eq:E2-def}
\hat\sE_2
&\;\left[ := \sum_{\tau,\nu \in [t]^+} \big\langle \hat\bNabla_\tau, \hat\bNabla_\nu \big\rangle_{\hat\Hmtx_t^{-1}} \; \hat\pi_{t,\tau}^2 \hat\pi_{t,\nu}^2\right]\; 
= \sum_{\tau,\nu \in [t]^+} \frac{\hat\pi_{\tau,\nu}}{\sqrt{\lambda_\tau \lambda_\nu}}\hat\pi_{t,\tau}^2 \hat\pi_{t,\nu}^2. 
\end{align}
Now, by Lemma \ref{lem:bound_leverage_score} proved in the appendix, we have the following bounds for the entries of~$\hat{\bPi}$:
\begin{equation}
\label{eq:bounds-for-leverage-scores}
|\hat\pi_{t,\tau}| \le \frac{1}{1+\lambda} \quad \forall \tau \in [t]; 
\quad\quad
|\hat\pi_{t,\tau}| \le \frac{1}{\sqrt{1+\lambda}} \quad \forall \tau < 0; 
\quad\quad
\frac{|\hat\pi_{\tau,\nu}|}{\sqrt{\lambda_\tau \lambda_{\nu}}} \le \frac{1}{\lambda} \quad \forall \tau, \nu \in [t]^+. 
\end{equation}
The first bound implies that~$\hat\alpha_t \le \frac{\lambda + \mu}{1 + \lambda}$. 
Combining the first two bounds in~\eqref{eq:bounds-for-leverage-scores} with~\eqref{eq:projection-property} we get
\begin{equation}
\label{eq:E1-bound}
|\hat\sE_1|
\le \frac{1}{\lambda} \sum_{\tau \in [t]^+} \hat\pi_{t,\tau}^2 
\stackrel{\eqref{eq:projection-property}}{=} \frac{\hat\pi_t}{\lambda}.
\end{equation}
Finally, applying the third and then the first bound in~\eqref{eq:bounds-for-leverage-scores}, we have that
\begin{equation}
\label{eq:E2-bound}
\hat\sE_2 
\le \frac{1}{\lambda} \left( \sum_{\tau \in [t]^+} \hat\pi_{t,\tau}^2 \right)^2
\stackrel{\eqref{eq:projection-property}}{=} \frac{\hat\pi_t^2}{\lambda}
\le 
\frac{\hat\pi_t}{\lambda(1+\lambda)}.
\end{equation}
Plugging~\eqref{eq:E1-bound}--\eqref{eq:E2-bound} and the bound for~$\hat\alpha_t$ back into~\eqref{eq:square-expanded} leads to
\begin{align}
\uDec_t^2 
\le \left(\frac{1+\lambda}{\lambda}\right)^2 
\left[ \left( \frac{\lambda+\mu}{1+\lambda} \right)^2 + \frac{2\mu(\lambda+\mu)}{\lambda(1+\lambda)} + \frac{\mu^2}{\lambda(1+\lambda)} \right] \hat\pi_t
&\le \frac{(1+\lambda)(\lambda + 2\mu)^2}{\lambda^3} \hat\pi_t. 
\label{eq:decrement-score-bound}
\end{align}
From the first bound in~\eqref{eq:bounds-for-leverage-scores} we get~$\uDec_t^2 \le \frac{1}{\lambda} \big( 1 + \frac{2\mu}{\lambda} \big)^2$, thus
$\uDec_t^2 \le \min \big\{ \tfrac{1}{4}, \frac{5\mu}{8} \hat\pi_t \big\}$ under the premise of the theorem. 
By Lemma~\ref{lem:bound-minimum-SC} and~\eqref{eq:lagged-gap-linear}, this gives
$
\Lag_t
\le 0.8 \, \uDec_t^2
\le \frac{\mu}{2} \hat\pi_t,
$
cf.~\eqref{eq:lag-estimate},
leading to~\eqref{eq:invariant}. The theorem is proved.
\qed

\section{Efficient implementation}
\label{sec:implementation}

We are now about to present an implementation of~\eqref{def:VB-FTRL} based on a quasi-Newton method run from the previous portfolio.
Let us first explain why this choice is natural in our context.

First, as we verify in Appendix~\ref{apx:derivatives}, function~$P_{t}^\res$, to be minimized in each round of~\OurAlgo{} after the reparametrization~\eqref{def:affine-reparametrization}, is self-concordant. 
Such functions were initially proposed in~\cite{nesterov1994interior} as the ones for which Newton's method has a global convergence guarantee. 
Moreover, it converges quadratically when initialized at a point with a small Newton decrement. 
On the other hand, when proving Theorem~\ref{th:volumetric} we have shown (when bounding~$\Dec_{t}$) that portfolio~$w_{t}$ gives a small Newton decrement in terms of~$P_{t-1}^\res$. As such, Newton's method run from~$\wt w_{t} \approx w_{t} [= \argmin_{w \in \Delta_d} P_{t-1}(w)]$ will produce an accurate approximation~$\wt w_{t+1}$ of~$w_{t+1}$ in just a few steps, enforcing that~$\wt w_{t+1} \approx w_{t+1}$.

Furthermore, the convergence guarantees for Newton's method, when run on a self-concordant function, are {\em affine-invariant}---just as the regret bound established in Theorem~\ref{th:volumetric}. 
This would not be so if we used, instead, a first-order method such as~gradient descent. Given substantial efforts spent on obtaining affine-invariant regret guarantees for various algorithms \odima{(specifically, on getting rid of the gradient norm parameters~$G_2, G_{\infty}$; cf.~Table~\ref{tab:rates} and the accompanying discussion),} it would be quite unsatisfactory if such a regret bound---even an optimal one---came only at the price of a computational guarantee that lacks affine-invariance. 

Finally, we have at our disposal a computationally cheap approximate second-order oracle for~$P_t$ (and thus for~$P_t^\res$ as well).
Namely, it can be verified (see Lemma~\ref{lem:volumetric-grad-and-hess}) that the gradient of~$V_t(w)$ is
\begin{equation}
\label{eq:V-grad-impl}
\nabla V_t(w) 
=  \sum_{\tau \in [t]^+} \pi_{\tau}(w) \, \nabla \ell_\tau(w),
\vspace{-0.1cm}
\end{equation}
and the Hessian of~$V_t$ satisfies the bounds~$\Qmtx_t(w) \preceq \nabla^2 V_t(w) \preceq 3\Qmtx_t(w)$ with~$\Qmtx_t(w)$ as follows:
\begin{equation}
\label{eq:V-hess-bound-impl}
\Qmtx_t(w) 
= \sum_{\tau \in [t]^+} \pi_{\tau}(w) \, \nabla^2 \ell_\tau(w).
\end{equation}
Now, the point here is that~$\Qmtx_t(w)$ can be computed in~$O(d^2 (T+d))$ arithmetic operations (a.o.'s), whereas computing~$\nabla^2 V_t(w)$ exactly would take~$O(d^2(T+d)^2)$ a.o.'s. (The latter can be seen from the explicit representation of~$\nabla^2 V_t$ as the sum of~$(T+d)^2$ outer-product terms, cf.~\eqref{eq:V-hess} in Appendix~\ref{apx:derivatives}.)
On the other hand, it can be shown that a quasi-Newton method in which the exact Hessian of a self-concordant objective is replaced with such an approximation---i.e., the one with a constant relative error as is the case for~$\nabla^2 V_t(w)$ and~$\Qmtx_t(w)$---admits strong convergence guarantees---namely, local linear convergence in a Dikin ellipsoid of the optimum; see Lemma~\ref{lem:quasi-Newton-update}. 
%
%
\vspace{-0.1cm}
\paragraph{Conceptual procedure.}
The above discussion suggests to replace~\eqref{def:VB-FTRL} with the update shown in Fig.~\ref{fig:VB-FTRL-Newton}, and referred to as~\eqref{def:VB-FTRL-Newton} from now on.
In~\eqref{def:VB-FTRL-Newton}, the next portfolio~$\wt w_t$ is obtained by running~$S$ steps of a quasi-Newton method for the minimization problem in~\eqref{def:VB-FTRL}, under the affine reparametrization of~$\Delta_d$ onto~$\bDelta_{d-1}$, cf.~\eqref{def:affine-reparametrization}--\eqref{def:affine-reparametrization-inv}, starting from the previously chosen portfolio~$\wt w_{t-1}$, and approximating the Hessian~$\nabla^2 P_t(w)$ with~$\nabla^2 L_t(w) + 3\mu\Qmtx_t(w)$. 
\begin{figure}[H]
\label{fig:VB-FTRL-Newton}
\begin{framed}
\vspace{-0.4cm}
\begin{quote}
\label{eq:Newton-conceptual}
\begin{center}
\end{center}
{\bf Input:}~$\wt w_{t-1} \in \Delta_{d}$.
\begin{enumerate}
\item Initialize~$w_t^{(0)} = \wt w_{t-1}^{\vphantom{(0)}}$.
\item For~$s = 0,1,..., S-1$ do:
\begin{equation}
\begin{aligned}
\Mmtx_{t-1}^{(s)} 
	&= A^\top \big[ \nabla^2 L_t \big(w_{t}^{(s)}\big) + 3\mu\Qmtx_t \big(w_{t}^{(s)}\big) \big] A, \\
w_{t}^{(s+1)} 
	&= w_{t}^{(s)} - A\big[\Mmtx_{t-1}^{(s)} \big]^{-1} A^\top \nabla P_t^{\vphantom{(s)}} \big(w_{t}^{(s)}\big).
\end{aligned}
\label{def:VB-FTRL-Newton}
\tag{\OurAlgo{}\textsf{-qN}}
\end{equation}
\end{enumerate}
{\bf Output:}~$\wt w_t^{\vphantom{(S)}} = w_{t}^{(S)}$.
\end{quote}
\vspace{-0.4cm}
\end{framed}
\vspace{-0.3cm}
\caption{Conceptual implementation of~\eqref{def:VB-FTRL} via a quasi-Newton method.}
\end{figure}

\paragraph{Computational guarantees.}
Algorithm~\ref{alg:efficient-ftrl-vr} is a concrete implementation of the update rule in~\eqref{def:VB-FTRL-Newton}, with~$O(dT + d^2)$ overall memory use, and~$O(d^2T + d^3)$ runtime per one Newton step. As we are about to see next (in Theorem~\ref{th:volumetric-newton}), it suffices to perform~$O(\log (T+d))$ Newton steps to match the regret guarantee of Theorem~\ref{th:volumetric} up to a constant factor; thus, the overall runtime is~$\wt O(d^2T + d^3)$ per round. 
\odima{(More precisely, round~$t$ takes~$O(d^2t+d^3)$ a.o.'s, but this gives~$O(d^2T + d^3)$ per round \em{on average} over~$T$ rounds.)} The runtime and memory \odima{cost} estimates we have just announced can be verified by a line-by-line inspection of Algorithm~\ref{alg:efficient-ftrl-vr}. 
Indeed, in each round~$t$ we start by translating~$w_{t-1}$ into~$v_{t-1} = v_{t-1}^{(0)}$ as per \eqref{def:affine-reparametrization-inv} to translate the problem on~$\bDelta_{d-1}$ a set with nonempty interior; clearly, this can be done in~$O(d)$.
Next, when given the previous iterate~$v_t^{(s)}$, we compute the individual loss gradients~$\bNabla_{\tau} = \nabla \ell_\tau^\res(v_t^{(s)})$ for each datapoint~$\tau \in [t-1]^+$ (including the ``fictional'' datapoints~$\{e_1,...,e_d\}$ corresponding to the logarithmic-barrier term); this takes~$O(dt+d^2)$ a.o.'s in total, and requires~$O(dt+d^2)$ of storage space. 
These data are then used to compute the Hessian~$\Hmtx_{t-1}^{(s)} = \nabla^2 L_{t-1}^\res\big(v_t^{(s)}\big)$ as the sum of~$O(t+d)$ one-rank matrices of the form~$\lambda^{\ind\{\tau < 0\}} \bNabla_{\tau}^{(s)} \big[\bNabla_{\tau}^{(s)}\big]^\top$,~in~$O(t d^2 + d^3)$ a.o.'s and~$O(d^2)$ additional storage. 
After that, we compute the inverse of~$\Hmtx_{t-1}^{(s)}$ with $O(d^3)$ runtime and~$O(d^2)$ additional storage, 
and use~$[\Hmtx_{t-1}^{(s)}]^{-1}$ to compute~$O(t+d)$ leverage scores~$\pi_{\tau}^{(s)}$, cf. line~\ref{alg-line:leverage-scores} of the algorithm, in~$O(d^2)$ a.o.'s per each~$\pi_{\tau}^{(s)}$. 
Using these, we then compute~$\nabla P_{t-1}^\res\big(v_t^{(s)}\big)$, compute and invert~$\Mmtx_{t-1}^{(s)}$ (cf.~line~\ref{alg-line:hessian-model}) that satisfies
\begin{equation}
\label{eq:hess-approx-validity}
\frac{1}{3} \Mmtx_{t-1}^{(s)}  \preceq \nabla^2 P_{t-1}^\res \big( v_{t}^{(s)} \big) \preceq \Mmtx_{t-1}^{(s)},
\end{equation}
and perform a matrix-vector multiplication (line~\ref{alg-line:Newton-step}).
This takes~$O(td^2+d^3)$ time and~$O(d^2)$ space.


\begin{center}
\begin{algorithm}[t!]
\caption{ \OurAlgoImp{}\; : Efficient implementation of \OurAlgo{} via a quasi-Newton method}
\label{alg:efficient-ftrl-vr}
\begin{algorithmic}[1]
\Require{$\lambda > 0, \mu > 0, S \in \N$}
\State $\wt w_0 := \frac{1}{d} \ones_d$ 
\Comment{{\em $\wt w_0 = \argmin_{w \in \Delta_d} P_{0}(w)$}}
\For{$t \in [T]$}
    \State $v_t^{(s)} := A^+ (\wt w_{t-1} - e_d)$
    \Comment{map~$\wt w_{t-1} \in \Delta_{d}$ to~$\bDelta_{d-1}$, cf.~\eqref{def:affine-reparametrization-inv}}
	\For{$s \in \{0, ..., S-1\}$}
		\For{$\tau \in  [t-1]^+$} 		\Comment{{\em $[t-1]^+ = [t-1] \cup \{-1, ..., -d\}$}}
		\State~$\lambda_{\tau} := \lambda^{\ind\{ \tau < 0\}}$
			\Comment{{\em recall that~$x_{-j} = e_j$ for~$j \in [d]$}}
		\State $\bNabla_{\tau}^{(s)} := \textstyle-\frac{1}{x_\tau^{\top} \big(A v_{t}^{(s)\vphantom\top} + e_d \big)} A^\top x_{\tau}$ 
			\Comment{{\em $\bNabla_{\tau}^{(s)} = \nabla \ell_{\tau}^\res \big(v_t^{(s)}\big)$}}
		\vspace{0.1cm}	
		\EndFor
		\vspace{0.1cm}
		\State $\Hmtx_{t-1}^{(s)} := \displaystyle \sum_{\tau \in [t-1]^+}  \lambda_{\tau}^{\vphantom\top}  \bNabla_{\tau}^{(s)} \big[\bNabla_{\tau}^{(s)} \big]^\top$
		\Comment{{\em $\Hmtx_{t-1}^{(s)} = \nabla^2 L_{t-1}^\res \big(v_t^{(s)}\big)$}}
		\vspace{0.1cm}	
		\State $\bSigma_{t-1}^{(s)} \;:= \big[\Hmtx_{t-1}^{(s)}\big]^{-1}$	
		\vspace{0.1cm}	
		\For{$\tau \in  [t-1]^+$} 
		\label{alg-line:leverage-scores}
			\State $\pi_{\tau}^{(s)} := \lambda_{\tau} \left\| \bNabla_{\tau}^{(s)} \right\|^2_{\bSigma_{t-1}^{(s)}}$
		\EndFor
		\vspace{0.1cm}		
		\State $\gbold_{t-1}^{(s)} \;\;\, := \displaystyle \sum_{\tau \in [t-1]^+} \big(\lambda_{\tau}^{\vphantom{(s)}} + \mu \pi_{\tau}^{(s)} \big) \bNabla_{\tau}^{(s)}$
		\Comment{$\gbold_{t-1}^{(s)} = \nabla P_{t-1}^\res \big( v_{t}^{(s)} \big)$}
		\vspace{0.1cm}		
		\State $\Mmtx_{t-1}^{(s)}  := \Hmtx_{t-1}^{(s)} + 3 \mu \displaystyle \sum_{\tau \in [t-1]^+} \pi_{\tau}^{(s)} \bNabla_{\tau}^{(s)} \big[\bNabla_{\tau}^{(s)}\big]^\top$
		\label{alg-line:hessian-model}
		\vspace{0.1cm}		
        \State $v_t^{(s+1)} := v_t^{(s)} - \big[\Mmtx_{t-1}^{(s)} \big]^{-1} \gbold_{t-1}^{(s)}$ \Comment{{\em quasi-Newton step for~$P_{t-1}^\res(\cdot)$}}
		\label{alg-line:Newton-step}
	\EndFor
    \State {Predict} $\wt w_t := Av_{t}^{(S)} + e_d$  \Comment{{\em map~$v_{t}^{(S)} \in \bDelta_{d-1}$ back to~$\Delta_{d}$}}
    \State {Receive} $x_t \in \R^d_+$
\EndFor
\end{algorithmic}
\end{algorithm}
\end{center}

\subsection{Regret guarantee for~\OurAlgoImp{}}
\label{sec:regret-Newton}

As we are about to see, performing a logarithmic in~$T+d$ number of quasi-Newton steps, instead of full minimization as per~\eqref{def:VB-FTRL}, suffices to match the regret bound of Theorem~\ref{th:volumetric}. More precisely, we have the following result for the regret of~\OurAlgoImp{} (a.k.a.~Algorithm~\ref{alg:efficient-ftrl-vr}).

\begin{theorem}
\label{th:volumetric-newton}
For any~$T \in \N$ and market realization~$x_{1:T}$, Algorithm~\ref{alg:efficient-ftrl-vr} produces~$\wt w_{1:T}$ such that
\[
\Regret_T(\wt w_{1:T}|x_{1:T}) \le (\lambda + 2\mu) \, \left[\hspace{0.03cm} (d-1) \log(T+\lambda d) + 1 \hspace{0.03cm}\right]
\]
if we choose~$\lambda,\mu, S$ satifying the following conditions:~$\lambda \ge 2e$,
\begin{equation}
\label{eq:volumetric-newton-conditions}
\frac{1}{\lambda} \left( 1 + \frac{2\mu}{\lambda} \right)^2 \le \min \left\{\frac{1}{556}, \; \frac{5\mu}{8(1+\lambda)} \right\},
\quad
S \ge 9 \log \left( 2 \max \left\{ (T + d + 1)^2, 10^4 \sqrt{1+3\mu} \right\} \right). 
\end{equation}
In particular, choosing~$\lambda = 560$,~$\mu = 2$, and~$S = 18 \lceil \log(T+d+164) \rceil + 7$ allows to guarantee that
\[
\Regret_T(\wt w_{1:T}|x_{1:T}) 
\le 564 \hspace{0.04cm} d \log(T + 560 \hspace{0.02cm} d).
\]
\end{theorem}

\begin{proof}
As previously noted, Algorithm~\ref{alg:efficient-ftrl-vr} precisely corresponds to the updates in~\eqref{def:VB-FTRL-Newton}. 
Let~$\wt w_{1:T}$ and~$w_{1:T}$ be two sequences produced by~\eqref{def:VB-FTRL-Newton} and~\eqref{def:VB-FTRL} correspondingly {\em on the same~$x_{1:T}$}. 
Let also~$w_{0} = \frac{1}{d}\ones_d$, and observe that~$\wt w_{0} = w_0 = w_1$ is the exact minimizer of~$P_0$.
The following technical result claims, in a nutshell, that~$\wt w_{t}$ remains very close to~$w_t$ for all~$t \in [T]$.
\begin{lemma}
\label{lem:approx-stability}
Let~$\lambda, \mu, S$ be as in the premise of Theorem~\ref{th:volumetric-newton}. Then for all~$t \in \{0\} \cup [T]$, one has
\begin{equation}
\label{eq:Newton-stability-invariant}
\| \wt w_t - w_t \|_{\nabla^2 L_{t}(w_t)} \le \min \left\{ (T+d+1)^{-2}, \; 10^{-4} (1 + 3\mu)^{-1/2}\right\}.
\end{equation}
\end{lemma}
\noindent Deferring the proof of Lemma~\ref{lem:approx-stability} to Appendix~\ref{apx:stability-lemma} we shall now focus on proving the theorem.

\proofstep{1}.
Mimicking~\eqref{eq:regret-decomp} we get
\begin{align}
\label{eq:regret-decomp-Newton}
\Regret_T(\wt w_{1:T}|x_{1:T}) 
        &\le \Bias_T + \sum_{t \in [T]} \big[ \ell_t(\wt w_t) + \min_{w \in \Delta_d} P_{t-1}(w) - \min_{w \in \Delta_d} P_t(w)\big] \notag\\
        &\le \Bias_T + \sum_{t \in [T]} \big[ \ell_t(\wt w_t) + P_{t-1}(\wt w_t)  - P_t(w_{t+1})\big] \notag\\
		&\le \Bias_T + \sum_{t \in [T]} [\wt \Lag_t + \wt \Off_t]        
\end{align}
where the term~$\Bias_T$ is the same as in the proof of Theorem~\ref{th:volumetric}, and the terms summed over~$t$ are
\begin{equation}
\label{def:regret-terms-Newton} 
\wt\Lag_t := P_t(\wt w_t) - P_t(w_{t+1}),
\quad 
\wt\Off_t := \mu [V_{t-1}(\wt w_t) - V_{t}(\wt w_t)].
\end{equation}
Since~$\lambda \ge 2e$,~$\Bias_T$ admits the same bound~\eqref{eq:bias-bound} as before. Moreover, for~$\Lag_t = P_t(w_t) - P_t(w_{t+1})$ and~$\Off_t = \mu [V_{t-1}(w_t) - V_{t}(w_t)]$ we still have that $\Lag_t + \Off_t \le 0$ for all~$t \in [T]$, 
simply because these terms, just like~\eqref{eq:bias-bound}, depend only on~$w_{1:T}$ (and not~$\wt w_{1:T}$), and the constraint on~$\lambda,\mu$ imposed in~\eqref{eq:volumetric-newton-conditions} is stronger than its counterpart in Theorem~\ref{th:volumetric}. 
Thus, it remains to show that~$\sum_{t \in [T]} \wt \Err_t \le \lambda + 2\mu$ for the sum of approximation errors~$\wt \Err_t := \wt \Lag_t - \Lag_t + \wt \Off_t - \Off_t$.

\proofstep{2}. 
Clearly,
\[
\wt \Err_t = L_t(\wt w_t) - L_t(w_t) + \mu [V_{t-1}(\wt w_t) - V_{t-1}(w_t)]. 
\]
By Lemma~\ref{lem:approx-stability}, Lemma~\ref{lem:SC-hessian-bounds}, and since~$L_t^\res(\cdot)$ and~$L_{t-1}^\res(\cdot)$ are~$1$-self-concordant, cf.~Corollary~\ref{cor:SC-of-L},
\[
\begin{aligned}
(1-\veps)^2 \Hmtx_t(w_t) \preceq \Hmtx_t&(\wt w_t) \preceq \frac{1}{(1-\veps)^2} \Hmtx_t(w_t),\\
(1-\veps)^2 \Hmtx_{t-1}(w_t) \preceq \Hmtx_{t-1}&(\wt w_t) \preceq \frac{1}{(1-\veps)^2} \Hmtx_{t-1}(w_t)
\end{aligned}
\]
where~$\veps$ is the right-hand side of~\eqref{eq:Newton-stability-invariant}, and the second line is due to~$\|\wt w_t - w_t \|_{\nabla^2 L_{t-1}(w_t)} \le \veps$. 
Thus,
\[
\begin{aligned}
\mu[V_{t-1}(\wt w_t) - V_{t-1}(w_t)] 
= \frac{\mu}{2} \log\det \left(\Hmtx_{t-1}(w_t)^{-1/2} \Hmtx_{t-1}(\wt w_t) \Hmtx_{t-1}(w_t)^{-1/2} \right) 
&\le \mu(d-1) \log\left( \frac{1}{1-\veps}\right)\\
&\le 2\mu(d-1)\veps.
\end{aligned}
\]
where in the end we used that~$\frac{1}{1-u} \le 1+2u$ for~$u \in [0,\frac{1}{2}]$.
On the other hand, applying Lemma~\ref{lem:SC-sandwich} to~$L_t^\res(\cdot)$ and noting that~$\|\nabla L_t(w)\|_{\nabla^2 L_t(w)^{-1}} \le t+d$ by triangle inequality, we get
\[
L_t(\wt w_t) - L_t(w_t) 
\le \nabla L_t(w_t)^\top (\wt w_t - w_t) + 0.8 \, \|\wt w_t - w_t\|_{\nabla^2 L_t(w_t)}^2 
\le (T+d)\veps + 0.8 \veps^2 \le (T+d+1)\veps.
\]
As such, we have a (very conservative) bound
$
\sum_{t \in [T]} \wt \Err_t \le (\lambda + 2\mu) \left(T +d+1\right)^2 \veps \le \lambda + 2\mu. 
$
\end{proof}

\section*{Acknowledgments}
Dmitrii M.~Ostrovskii acknowledges support by the National Science Foundation grant CIF-1908905. 

\appendix
%
\section{Auxiliary results}
\label{apx:aux-results}

\subsection{Logarithmic barrier and smoothed portfolios} 

\begin{lemma}
\label{lem:entries-cutoff}
Let~$w^\star_T = \argmin_{w \in \Delta_d} L_T(w)$, then
\begin{equation}
\label{eq:entries-cutoff-app}
w^\star_T[i] \ge \frac{\lambda}{T + \lambda d}, \quad \forall i \in [d].
\end{equation}
\end{lemma}
\begin{proof}
Let~$1/w$ be the vector with entries~$w[i] = 1/w[i]$. 
The first-order optimality conditions read
\[
(w - w^\star_T)^\top \Bigg( \frac{\lambda}{w_T^\star} + \sum_{t \in [T]}  \frac{x_t}{x_t^\top w^\star_T} \Bigg) 
\le 0, \quad \forall w \in \Delta_d.
\]
Rearranging, we conclude that
\[
\lambda \sum_{i \in [d]} \frac{w[i]}{w^\star_T[i]} + \sum_{t \in [T]}  \frac{x_t^\top w}{x_t^\top w^\star_T} \le \lambda d + T.
\]
The second sum is nonnegative; choosing~$w = e_{i}$ for~$i \in [d]$ violating~\eqref{eq:entries-cutoff-app} gives a contradiction.
\end{proof}

\begin{lemma}
\label{lem:cutoff-regret}
\mbox{For all~$w \in \Delta_d$, the smoothed portfolio~$[w]_{\alpha} := (1-\alpha) w + \frac{\alpha}{d} \ones_d$, $\alpha \in [0,1)$, satisfies}
\[
\sum_{t \in [T]} \ell_{t}([w]_\alpha)  - \ell_{t}(w) \le \frac{\alpha}{1-\alpha} T.
\]
In particular, for~$\alpha = \frac{\lambda (d-1)}{T+\lambda (d-1)}$ we have that
\[
\sum_{t \in [T]} \ell_{t}([w]_{\alpha})  - \ell_{t}(w) \le \lambda (d-1).
\]

\end{lemma}

\begin{proof}
Observe that~$w - [w]_{\alpha} = \frac{\alpha}{1-\alpha} ([w]_{\alpha} - \frac{1}{d} \ones_d)$. 
Whence by convexity we get
\[
\begin{aligned}
\sum_{t \in [T]} \ell_{t}([w]_\alpha)  - \ell_{t}(w) 
\le \sum_{t \in [T]} \frac{x_t^\top (w - [w]_\alpha)}{x_t^\top [w]_{\alpha}}  
= \frac{\alpha}{1-\alpha}  
	\sum_{t \in [T]} 
    \left( 1 - \frac{1}{d} \cdot \frac{\ones_d^\top x_t }{x_t^\top [w]_\alpha} \right) 
\le 
\frac{\alpha}{1-\alpha} T.
\end{aligned}
\qedhere
\]
\end{proof}

\begin{lemma}
\label{lem:rank-one-comparison}
For all~$w \in \R^d_{++}$ and~$x \in \R^d_{+}$, function~$R(\cdot)$ in~\eqref{def:log-barrier} satisfies 
$
\nabla^2 R(w) \succeq \frac{1}{(x^\top w)^2} x x^\top.
$
\end{lemma}

\begin{proof}
Equivalently, we must prove that, for~$w,x$ as in the premise of the lemma, and for any~$u \in \R^d$,
\[
(x^\top w)^2 \sum_{i \in [d]} \frac{u[i]^2}{w[i]^2} \ge {(x^\top u)^2}.
\]
Clearly, it suffices to consider~$u \in \R^d_+$.
Let~$r \in \R^d_+$ have entries~$r[i] = \frac{u[i]}{w[i]}$, then we have that
\[
\sqrt{\sum_{i \in [d]} \frac{u[i]^2}{w[i]^2}} \, x^\top w = \|r\|_2 \, x^\top w  \ge \| r \|_{\infty} \, x^\top w  
\ge \sum_{i \in [d]} r[i] \, x[i] \, w[i] = x^\top u.
\qedhere
\]
\end{proof}

\subsection{Facts from linear algebra}

\begin{lemma}[{\cite[Lemma~12]{hazan2007logarithmic}}]
\label{lem:upper_bound_trace}
For any~$\Hmtx \succ \Hmtx' \succeq 0$, it holds that
$
\log\left(\frac{\det(\Hmtx)}{\det(\Hmtx-\Hmtx')}\right) \ge \Tr(\Hmtx^{-1} \Hmtx').
$
In particular, if~$\Hmtx'$ is rank-one, i.e.~$\Hmtx = u u^\top$ for some vector~$u$, then~$\log\left(\frac{\det(\Hmtx)}{\det(\Hmtx-uu^\top)}\right) \ge \|u\|_{\Hmtx^{-1}}^2.$
\end{lemma}
\begin{proof}
\noindent The result follows by using that~$\log \det(\cdot)$ is concave on the positive-semidefinite cone~(see~\citep{boyd2004convex}):
\begin{align*}
	\log \left(\frac{\det(\Hmtx)}{\det(\Hmtx-\Hmtx')}\right)
    &= \log \det(\Hmtx(\Hmtx-\Hmtx')^{-1}) 
    \ge \Tr(\Imtx - (\Hmtx-\Hmtx') \Hmtx^{-1})
    = \Tr(\Hmtx^{-1} \Hmtx').
\qedhere
\end{align*}
\end{proof}

\begin{lemma}
\label{lem:inv-hessian-derivative}
Let~$\Hmtx \in \Sym^{n}_{++}$ and~$u,v \in \R^n$. Then
$
\nabla_{\Hmtx} ( \lang u,v \rang_{\Hmtx^{-1}}) 
= -\frac{1}{2} \Hmtx^{-1} (u v^\top + v u^\top) \Hmtx^{-1}. 
$
In particular,
\[
\nabla_{\Hmtx}  ( \|  u \|_{\Hmtx^{-1}}^2 ) = -\Hmtx^{-1} u u^\top \Hmtx^{-1}.
\]
\end{lemma}
\begin{proof}
Fix an arbitrary symmetric matrix~$\Mmtx \in \R^{n \times n}$, and consider~$t \in \R$ for which~$\Hmtx - t\Mmtx \succ 0$. 
In terms of the Cholesky factorization~$\Hmtx = RR^\top$ we have that, as~$t \to 0$,
\[
(\Hmtx + t \Mmtx)^{-1}
= R^{-\top} (\Imtx - t R^{-1} \Mmtx R^{-\top}) R^{-1} + o(t)
= \Hmtx^{-1} - t\Hmtx^{-1} \Mmtx \Hmtx^{-1} + o(t).
\]
Whence~$u^\top \Hmtx^{-1} u - u^\top (\Hmtx + t \Mmtx)^{-1} u = t \, \Tr(\Mmtx \Hmtx^{-1} u u^\top \Hmtx^{-1}) + o(t)$, and the formula for~$\nabla_{\Hmtx}(\|  u \|_{\Hmtx^{-1}}^2)$ is verified.
More generally,
\[
\begin{aligned}
u^\top \Hmtx^{-1} v  - u^\top (\Hmtx + t \Mmtx)^{-1} v 
= t \Tr(\Mmtx \Hmtx^{-1} v u^\top \Hmtx^{-1}) + o(t) 
= \frac{t}{2}  \Tr(\Mmtx \Hmtx^{-1} (u v^\top + v u^\top ) \Hmtx^{-1}) + o(t). 
\end{aligned}
\] 
Since~$\Hmtx^{-1} (u v^\top + v u^\top ) \Hmtx^{-1}$ is a symmetric matrix, this verifies the formula for~$\nabla_{\Hmtx}  ( \lang u,v \rang_{\Hmtx^{-1}})$.
\end{proof}

\subsection{Self-concordant functions}
\label{apx:self-conc}

In this section, we recall the notion of self-concordant functions and the key properties of such functions.
This class of functions was thoroughly studied in the seminal work of Nesterov and Nemirovski~\cite{nesterov1994interior} in connection with interior-point methods; 
earlier, Vaidya used similar tools in the context of cutting-plane methods~\cite{vaidya1989new}.
The results presented next are well known; their proofs can be found in~\cite[Section~4]{nesterov2003introductory}, see also~\cite[Proposition B.2]{ostrovskii2021finite} for a more concise proof of Lemma~\ref{lem:SC-hessian-bounds}.


\begin{definition}
\label{def:SC-function}
Let~$f: \R^n \to \wb\R$ be convex, proper, and three times differentiable on its domain~$\Dom(f) := \{z: f(z) < \infty\}$ which is open, has nonempty interior, and does not contain straight lines. 
Then~$f$ is called $\sM$-{\em self-concordant} if it is a barrier on its domain---i.e.~$f(z) \to_{z \to z_0} \infty$ for any~$z_0$ in the boundary of~$\Dom(f)$---and the following inequality holds for all~$z \in \Dom f$ and~$u \in \R^n$:
\begin{equation}
\label{def:SC-bound}
\left|\nabla^3 f(z) [u,u,u] \right|  \le 2 \sM \left( \nabla^2 f(z)[u,u] \right)^{3/2}. 
\end{equation}
\end{definition}

The following property of self-concordant functions is well-known, see e.g.~\cite[Section 4.1]{nesterov2003introductory}.
\begin{lemma}[{\cite[Theorem 4.1.1]{nesterov2003introductory}}]
\label{lem:SC-sum}
Let~$f_1$ and~$f_2$ be~$\sM_1$- and~$\sM_2$-self-concordant, respectively. 
For~$\lambda_1,\lambda_2 > 0$, the function~$\lambda_1 f_1 + \lambda_2 f_2$ has domain~$\Dom(f_1) \cap \Dom(f_2)$ and is~$\sM$-self-concordant,
\[
\sM = \max \left\{\frac{\sM_1}{\sqrt{\lambda_1}}, \frac{\sM_2}{\sqrt{\lambda_2}}  \right\}.
\]
\end{lemma}

Lemma~\ref{lem:SC-sum} implies self-concordance of the reparametrized logarithmic-barrier function~$L_t^\res(\cdot)$.
%
%
\begin{corollary}
\label{cor:SC-of-L}
Assume~$\lambda \ge 1$, then the restriction~$L_t^\res: \R^{d-1} \to \wb\R$ of~$L_t$ to~$\Aff_d$ as in~\eqref{def:affine-subspace}, namely 
\begin{equation*}
L_t^{\res}(v) = L_t(Av + e_d) 
\end{equation*}
with~$A$ defined in~\eqref{def:coordinate-matrix}, is~$1$-self-concordant according to~Definition~\ref{def:SC-function} with domain~$\Int(\bDelta_{d-1})$, cf.~\eqref{def:solid-simplex}.
\end{corollary}
\begin{proof}
It is evident that~$\Dom(L_t^{\res}) = \Int(\bDelta_{d-1})$ where~$\bDelta_{d-1}$ is the ``solid'' simplex in~$\R^{d-1}$, cf.~\eqref{def:solid-simplex}.
In particular,~$\Dom(L_t^{\res})$ is open and does not contain straight lines, and~$L_t^{\res} \to +\infty$ on its boundary. 
Thus, the requirements on the domain are verified. 
On the other hand, inequality~\eqref{def:SC-bound} holds with~$\sM=1$ for each of the functions~$\ell_\tau(Av+e_d)$,~$\tau \in [d]$, and with~$\sM = 1/\sqrt{\lambda} \le 1$ for the functions~$-\lambda \log(e_i^\top (Av + e_d))$, $i \in [d]$. Applying Jensen's inequality as in the proof of Lemma~\ref{lem:SC-sum} (cf.~\cite[Theorem 4.1.1]{nesterov2003introductory}) we verify~\eqref{def:SC-bound} with~$\sM=1$ for~$L_t^\ones(\cdot)$. 
That is,~$L_t^{\res}(\cdot)$ is $1$-self-concordant.
\end{proof}

%

Following~\cite[Section 4.1]{nesterov2003introductory} we now define a pair of mutually conjugate functions~$\omega, \psi$ on~$\R^+$ by
\begin{equation}
\label{def:nesterov-functions}
\begin{aligned}
\omega(r) := r - \log(1+r), \quad \psi(r) := -r - \log(1-r). 
\end{aligned}
\end{equation}
Note that~$\omega$ and~$\psi$ are increasing,~$\omega(r) \le \half r^2 \le \psi(r)$, and these bounds are tight at~$r = 0$. 
In fact,~$\omega(r)$ is a Huber-type function, so~$\omega(r) = (1-o(1))r$ as~$r \to \infty$; meanwhile,~$\psi(r) {\to} +\infty$ as~$r \uparrow 1$. 
More specifically, it can be verified that~$\omega(r) \ge 0.3 r^2$ for~$r \le 1$, and~$\psi(r) \le 0.8 r^2$ for~$r \le 0.5$. 

\begin{lemma}[{\cite[Theorems~4.1.10--4.1.11, 4.1.13]{nesterov2003introductory}}]
\label{lem:bound-minimum-SC}
Let~$f$ be~$1$-self-concordant, and some~$\hat z \in \Dom(f)$ satisfies~$\|\nabla f(\hat z)\|_{\nabla^2 f(\hat z)^{-1}} < 1$. 
Then $f$ has a unique minimizer~$z^\star$, and
\[ 
	\omega\left(\|\nabla f(z)\|_{\nabla^2 f(z)^{-1}}\right)
\le 
	f(z) - f(z^\star) 
\le 
	\psi\left(\|\nabla f(z)\|_{\nabla^2 f(z)^{-1}}\right)
\]
where the lower bound holds for all~$z \in \Dom(f)$, and the upper bound requires~$\|\nabla f(z)\|_{\nabla^2 f(z)^{-1}} < 1$. 
In particular,
$
0.3\, \|\nabla f(z)\|_{\nabla^2 f(z)^{-1}}^2 \le f(z) - f(z^\star) \le 0.8 \, \|\nabla f(z)\|_{\nabla^2 f(z)^{-1}}^2
$
if~$\|\nabla f(z)\|_{\nabla^2 f(z)^{-1}} \le 0.5$.
\end{lemma}

\begin{lemma}[{\cite[Theorems~4.1.5, 4.1.7--4.1.8]{nesterov2003introductory}}]
\label{lem:SC-sandwich}
Let~$f$ be~$1$-self-concordant,~$z \in \Dom(f)$, then
\[
	\omega \left(\| z' - z \|_{\nabla^2 f(z)} \right) 
\le 
	f(z') - f(z) - \nabla f(z)^\top (z' - z) 
\le 
	\psi \left(\| z' - z \|_{\nabla^2 f(z)} \right).
\]
Here, the lower bound holds~$\forall z' \in \Dom(f)$ and the upper bound on~$1$-Dikin ellipsoid~$\cE_{f,1}(z)$ where
\begin{equation}
\label{def:Dikin-z}
\cE_{f,r}(z) := \{ z' \in \R^n: \|z'-z\|_{\nabla^2 f(z)} < r\}.
\end{equation}
In particular,~$\cE_{f,1}(z)$ is contained in~$\Dom(f)$. Moreover, for~$z ' \in \cE_{f,\frac{1}{2}}(z)$ we have the following:
\[
\begin{aligned}
0.3 \, \|z' - z\|_{\nabla^2 f(z)}^2 \le f(z') - f(z) - \nabla f(z)^\top (z' - z) \le 0.8 \, \|z' - z\|_{\nabla^2 f(z)}^2.
\end{aligned}
\] 
\end{lemma}

\begin{lemma}[{\cite[Theorem~4.1.6]{nesterov2003introductory}}]
\label{lem:SC-hessian-bounds}
Let~$f$ be~$1$-self-concordant and~$z \in \Dom(f)$, then we have 
\[
\left(1 - \|z' - z\|_{\nabla^2 f(z)}\right)^2
	\, \nabla^2 f(z)
\preceq
\nabla^2 f(z')
\preceq
\frac{1}{\left(1 - \|z' - z\|_{\nabla^2 f(z)}\right)^2} 
	\, \nabla^2 f(z)
\]
for all~$z' \in \cE_{f,1}(z)$. 
In particular, for~$z' \in \cE_{f,\frac{1}{2}}(z)$ we have that
$
0.25 \nabla^2 f(z) \preceq \nabla^2 f(z') \preceq 4 \nabla^2 f(z).
$
\end{lemma}

\subsection{Quasi-Newton method}
\label{apx:quasi-Newton}

Next we focus on a quasi-Newton method in which the actual Hessian is replaced with an approximation with a constant relative accuracy.
Such a method converges linearly if initialized in a constant-radius Dikin ellipsoid of the minimizer of a self-concordant function. 
Analogous results were obtained in~\cite[Lemma 11]{marteau2019globally}, for a modified notion of self-concordance. 
\begin{lemma}[Linear convergence for a quasi-Newton method]
\label{lem:quasi-Newton-update}
Let~$f$ be~$1$-self-concordant,~$z \in \Dom(f)$ be such that~$\|\nabla f(z)\|_{\nabla^2 f(z)^{-1}} \le \frac{\sfc}{3}$ for some~$\sfc \le 1$, 
and consider~$z^+ := z \, - \xHmtxInv \nabla f(z)$ with 
\[
\sfc \xHmtx \preceq \nabla^2 f(z) \preceq \;\; \xHmtx \hspace{-0.2cm}.
\]
Then~$z^+ \in \Dom(f)$, and we have that
$
\|\nabla f(z^+)\|_{\nabla^2 f(z^+)^{-1}} \le (1 - \frac{\sfc}{3}) \|\nabla f(z)\|_{\nabla^2 f(z)^{-1}}.
$
\end{lemma}

\begin{proof}
Define~$\Hmtx(z') := \nabla^2 f(z')$ for any~$z' \in \Dom(f)$, and~$\Jmtx := \Hmtx(z)^{1/2} \xHmtxInv \Hmtx(z)^{1/2}$. 
Then, by the premise of the lemma,
$
\sfc \Hmtx(z)^{-1} \preceq \, \xHmtxInv \hspace{-0.1cm}\preceq \Hmtx(z)^{-1},
$
which implies
\[
\sfc \Imtx \preceq \Jmtx \preceq \Imtx
\quad \text{and} \quad
0 \preceq \Imtx - \Jmtx  \preceq (1-\sfc) \Imtx.
\]
As a result, defining~$\delta(z') := \|\nabla f(z')\|_{\nabla^2 f(z')^{-1}}$ for brevity, we get
\[
\|z^+ - z\|_{\Hmtx(z)}^2 = \| \hspace{-0.1cm} \xHmtxInv \nabla f(z) \|_{\Hmtx(z)}^2 
= \nabla f(z)^\top \Hmtx(z)^{-1/2} \Jmtx^2 \Hmtx(z)^{-1/2} \nabla f(z)
\le \delta(z)^2,
\]
whence (by Lemma~\ref{lem:SC-hessian-bounds})~$z^+ \in \Dom(f)$ and
\begin{equation}
\label{eq:Newton-hess-update-approx}
(1-\delta(z))^2 \Hmtx(z)^{-1} \preceq \Hmtx(z^+)^{-1} \preceq \frac{1}{(1-\delta(z))^2} \Hmtx(z)^{-1}.
\end{equation}
Now, define~$\Hmtx(z) := \nabla^2 f(z)$ and~$z_s := sz^+ + (1-s)z$ for~$s \in [0,1]$. 
Since~$\nabla f(z) \; + \hspace{0.1cm} \xHmtx \hspace{-0.2cm}(z^+-z) = 0$, 
\begin{equation}
\label{eq:Newton-error-decomposition}
\begin{aligned}
\left\|\nabla f(z^+) \right\|_{\Hmtx(z)^{-1}} 
&= \left\|\nabla f(z^+) - \nabla f(z) \; - \hspace{0.1cm} \xHmtx \hspace{-0.2cm}(z^+-z) \right\|_{\Hmtx(z)^{-1}} \\
&= \left\|\int_{0}^1 \big( \Hmtx(z_s) - \hspace{0.1cm} \xHmtx \hspace{-0.1cm} \big) (z^+ - z) \, ds  \right\|_{\Hmtx(z)^{-1}} \\
&\le \int_{0}^1 \left\| \big( \Hmtx(z_s) \hspace{0.1cm} - \hspace{0.1cm} \xHmtx \hspace{-0.1cm} \big) (z^+ - z) \right\|_{\Hmtx(z)^{-1}} ds \\
&\le \int_{0}^1 \left\| \big( \Hmtx(z_s) - \Hmtx(z) \big) (z^+ \hspace{-0.1cm} - \hspace{-0.05cm} z) \right\|_{\Hmtx(z)^{-1}} ds 
	+ \big\| \big( \xHmtx \hspace{-0.1cm}  - \hspace{0.1cm} \Hmtx(z) \big) (z^+ - z) \big\|_{\Hmtx(z)^{-1}} . 
\end{aligned}
\end{equation}
For the last term in the right-hand side, we get
\begin{equation}
\label{eq:Newton-bias-term-bound}
\begin{aligned}
\left\| \big( \xHmtx \hspace{-0.1cm}  - \hspace{0.1cm} \Hmtx(z) \big) (z^+ \hspace{-0.1cm} -  z) \right\|_{\Hmtx(z)^{-1}}^2
&= \left\| \big( \Imtx  - \Hmtx(z) \hspace{-0.1cm}\xHmtxInv \big) \nabla f(z) \right\|_{\Hmtx(z)^{-1}}^2 \\
&= \nabla f(z)^\top \Hmtx(z)^{-1/2} (\Imtx  - \Jmtx)^2 \Hmtx(z)^{-1/2} \nabla f(z)
\le (1-\sfc)^2 \delta(z)^2.
\end{aligned}
\end{equation}
On the other hand, since~$z_s - z = s (z^+ - z)$, similarly to~\eqref{eq:Newton-hess-update-approx}, we have that
\[
(1-s\delta(z))^2 \Hmtx(z)^{-1} \preceq \Hmtx(z_s)^{-1} \preceq \frac{1}{(1-s\delta(z))^2} \Hmtx(z)^{-1},
\]
and since~$(1-u)^2 \le \frac{1}{(1-u)^2}$ for~$u \in (0,1)$, we get
$
\left\|  \Hmtx(z)^{1/2} \Hmtx(z_s)^{-1} \Hmtx(z)^{1/2} - \Imtx \right\| \le \frac{1}{(1-s\delta(z))^2} - 1.
$
As a result, since
\[
\left\| \big( \Hmtx(z_s) - \Hmtx(z) \big) (z^+ \hspace{-0.15cm} - \hspace{-0.07cm} z) \right\|_{\Hmtx(z)^{-1}}^2 
= (z^+ - z)^\top \Hmtx(z)^{\frac{1}{2}} \big( \Hmtx(z)^{\frac{1}{2}} \Hmtx(z_s)^{-1} \Hmtx(z)^{\frac{1}{2}} - \Imtx \big)^2 \Hmtx(z)^{\frac{1}{2}} (z^+ - z),
\]
we get
\[
\left\| \big( \Hmtx(z_s) - \Hmtx(z) \big) (z^+ \hspace{-0.1cm} - \hspace{-0.05cm} z) \right\|_{\Hmtx(z)^{-1}} 
\le \left( \frac{1}{(1-s\delta(z))^2} - 1 \right) \left\| z^+ - z \right\|_{\Hmtx(z)} 
\le \left( \frac{1}{(1-s\delta(z))^2} - 1 \right) \delta(z)
\]
and estimate the integral term:
\begin{equation}
\label{eq:Newton-integral-term-bound}
\int_{0}^1 \left\| \big( \Hmtx(z_s) - \Hmtx(z) \big) (z^+ \hspace{-0.1cm} - \hspace{-0.05cm} z) \right\|_{\Hmtx(z)^{-1}} ds
\le \delta(z) \int_{0}^1 \left( \frac{1}{(1-s\delta(z))^2} - 1 \right) ds
= \frac{\delta(z)^2}{1-\delta(z)}.
\end{equation}
Finally, combining~\eqref{eq:Newton-hess-update-approx}--\eqref{eq:Newton-integral-term-bound}, and using that~$\delta(z) \le \frac{\sfc}{3}$, we conclude that
\[
\delta(z^+) 
\le (1-\sfc) \frac{\delta(z)}{1-\delta(z)} + \frac{\delta(z)^2}{(1-\delta(z))^2} 
\le \left(1-\frac{2\sfc}{3} \right) \frac{\delta(z)}{1-\delta(z)}
\le \frac{3-2\sfc}{3-\sfc} \delta(z)
\le \left(1-\frac{\sfc}{3} \right) \delta(z).
\qedhere
\]
\end{proof}

\begin{corollary}
\label{cor:quasi-Newton-convergence}
Let~$f$ be~$1$-self-concordant and minimized at~$z^\star$, and assume that~$z_0 \in \Dom(f)$ is such that~$\|\nabla f(z_0)\|_{\nabla^2 f(z_0)^{-1}} \le \frac{\sfc}{3}$ for some~$\sfc \le 1$.
Consider the following sequence for~$s \in \{0\} \cup [S-1]$:
\[
z_{s+1}^{\vphantom{-1}} := z_{s}^{\vphantom{-1}} - \xHmtxInv_{s} \nabla f(z_{s}^{\vphantom{-1}})
\quad \text{where} \quad
\sfc \xHmtx_{\hspace{-0.1cm}s} \preceq \nabla^2 f(z_{s}) \preceq \;\; \xHmtx_{\hspace{-0.1cm}s}.
\]
Then, for any~$\veps \le \frac{1}{4}$, we have the following inequalities if~$S \ge \left\lceil \frac{3}{\sfc} \log \left( \frac{1}{\veps} \right) \right\rceil$:
\[
\max \left\{ \tfrac{1}{4} \|z_S- z^\star\|_{\nabla^2 f(z^\star)}^2, \;\; f(z_S) - f(z^\star), \;\; \|\nabla f(z_S)\|_{\nabla^2 f(z_S)^{-1}}^2 \right\} \le \veps^2,
\]
\[
(1-2\veps)^2 \nabla^2 f(z^\star) \preceq \nabla^2 f(z_S) \preceq \frac{1}{(1-2\veps)^2}\nabla^2 f(z^\star).
\]
\end{corollary}

\begin{proof}
Using Lemma~\ref{lem:quasi-Newton-update} sequentially~$S$ times,
$
\|\nabla f(z_S)\|_{\nabla^2 f(z_S)^{-1}} \le \frac{\sfc}{3}  \left( 1-\frac{\sfc}{3} \right)^{S} \le \min \{\veps, \frac{1}{3}\}.
$
Then Lemma~\ref{lem:bound-minimum-SC} gives~$f(z_S) - f(z^\star) \le 0.8 \veps^2$. 
Now, one can verify that~$\omega(r) \ge 0.1 r$ for~$r \ge 0.5$, whence by Lemma~\ref{lem:SC-sandwich} first~$\|z_S - z^\star\|_{\nabla^2 f(z^\star)} \le \max\{0.5, 8 \veps^2\} \le 0.5$, and then~$\|z_S- z^\star\|^2 \le \frac{8}{3} \veps^2$.
\end{proof}

\section{Proofs deferred from Section~\ref{sec:VB-FTRL-vs-Cover}}
\label{apx:cover-correspondence-proofs}

\subsection{Proof of Proposition~\ref{prop:cover-general-regret}}
\label{apx:generalized-cover-proof}

Function~$h_t(w) = e^{-\frac{1}{\mu} \ell_t(w)} \; [= (x_t^\top w)^{1/\mu}]$ is concave on~$\Delta_d$ when~$\mu \ge 1$; thus, by Jensen's inequality,
\[
\begin{aligned}
\ell_t(w_t) 
&\le -\mu \log \E_{w \sim \phi_t} \left[\exp \left(-\frac{1}{\mu} \ell_t(w) \right) \right] \\
&\le -\mu \log \left[ \int_{\Delta_d} \exp \left(-\frac{1}{\mu} L_t(w) \right) dw \right]
	+\mu \log \left[ \int_{\Delta_d} \exp \left(-\frac{1}{\mu} L_{t-1}(w) \right) dw \right].
\end{aligned}
\]
Whence via telescoping and~\eqref{def:LB-potential}:
\begin{equation}
\label{eq:cover-general-telescoping}
\sum_{t \in [T]} \ell_t(w_t) 
\le -\mu \log \left[ \int_{\Delta_d} \exp \left(-\frac{1}{\mu} L_T(w) \right) dw \right] + \mu \log \left[ \int_{\Delta_d} \exp \left(-\frac{\lambda}{\mu} R(w) \right) dw \right].
\end{equation}
Now, let~$w_{T}^\star = \argmin_{w \in \Delta_d} L_T(w)$. 
Define the set~$\Delta_{d,T}^{\star} := \left\{ \frac{T}{T+1} w_T^\star + \frac{1}{T+1} w, \; \forall w \in \Delta_d \right\}$, i.e.~a copy of~$\Delta_d$ shrinked by the factor of~$\frac{1}{T+1}$ and shifted by~$w_T^\star$. Since~$h_t(\cdot)$ is concave and positive, we have:
\[
h_t(w) \ge \frac{T}{T+1} h_t(w_T^\star) \quad \forall w \in \Delta_{d,T}^{\star}.
\]
As such, for all~$w \in \Delta_{d,T}^{\star}$ we have that
\[
\exp \left(-\frac{1}{\mu} L_T(w)\right) 
\ge \exp \left(-\frac{\lambda}{\mu} R(w) \right) \cdot \left(\frac{T}{T+1} \right)^T \prod_{t \in [T]} h_t(w_T^\star) 
\ge \frac{1}{e} \exp \left(-\frac{\lambda}{\mu} R(w) \right) \prod_{t \in [T]} h_t(w_T^\star). 
\]
After taking the logarithm and noting that~$\Delta_{d,T}^\star \subseteq \Delta_d$, this results in
\[
\begin{aligned}
-\mu \log \left[ \int_{\Delta_d} \exp \left(-\frac{1}{\mu} L_T(w) \right) dw \right] 
&\le - \mu \log \left[ \frac{1}{e} \int_{\Delta_{d,T}^\star} \exp \left(-\frac{\lambda}{\mu} R(w) \right) dw \right] + \sum_{t \in [T]} \ell_{t}(w_T^\star). 
\end{aligned}
\]
Returning to~\eqref{eq:cover-general-telescoping}, this results in
\begin{equation}
\label{eq:cover-general-aggregated-bound}
\begin{aligned}
\Regret_T(w_{1:T}|x_{1:T}) 
\le 
&\sum_{t \in [T]} \ell_{t}(w_T^\star) - \min_{w \in \Delta_d} \sum_{t \in [T]} \ell_{t}(w) \\
&+ \mu \log \left[ \int_{\Delta_d} \exp \left(-\frac{\lambda}{\mu} R(w) \right) dw \right] - \mu \log \left[ \frac{1}{e} \int_{\Delta_{d,T}^\star} \exp \left(-\frac{\lambda}{\mu} R(w) \right) dw \right].
\end{aligned}
\end{equation}
Observing that~$R(w) \ge R(\frac{1}{d}\ones_d) = d\log(d)$ on~$\Delta_d$ allows to estimate the penultimate term in~\eqref{eq:cover-general-aggregated-bound}:
\begin{equation}
\label{eq:cover-general-aggregated-bound-1}
\mu \log \left[ \int_{\Delta_d} \exp \left(-\frac{\lambda}{\mu} R(w) \right) dw \right] 
\le  \mu \log[\Area(\Delta_d)]  -\lambda d \log d.
\end{equation}
Meanwhile,~$R(w)$, as a convex function, is maximized on a vertex of~$\Delta_{d,T}^\star$, whence for any~$w \in \Delta_{d,T}^*$:
\[
\begin{aligned}
R(w) 
\le \max_{k \in [d]} R \left(\frac{T}{T+1} w_T^\star + \frac{1}{T+1} e_k\right)
&= \max_{k \in [d]}  \left\{  
\log\left(\frac{T+1}{T w_T^\star[k] + 1} \right) 
+\sum_{i \in [d]\setminus \{k\}}  \log \left(\frac{T+1}{T w_T^\star[i]} \right) 
\right\} \\
&\le R(w_T^\star) + (d-1) \log \left(\frac{T+1}{T} \right).
\end{aligned}
\]
This allows to estimate the last term in~\eqref{eq:cover-general-aggregated-bound}:
\begin{equation}
\label{eq:cover-general-aggregated-bound-2}
\begin{aligned}
-\mu \log \left[ \frac{1}{e} \int_{\Delta_{d,T}^\star} \exp \left(-\frac{\lambda}{\mu} R(w) \right) dw \right] 
&\le \lambda R(w_{T}^\star) + \lambda (d-1) \log\left(\frac{T+1}{T}\right) + \mu - \mu\log[\Area(\Delta_{d,T}^\star)]  \\
&\le \lambda R(w_{T}^\star) + \lambda (d-1) \log\left(2\right) + \mu - \mu\log[\Area(\Delta_{d})] \\ 
& \quad\quad\quad\quad\quad\quad\quad\quad\quad\quad\quad\quad\;\; + \mu (d-1) \log(T+1).
\end{aligned}
\end{equation}
Here we used that~$\Area(\Delta_{d,T}^\star) = \frac{1}{(T+1)^{d-1}} \Area(\Delta_{d})$. 
Finally, proceeding as in step~\proofstep{1} of the proof of Theorem~\ref{th:volumetric} we get
\begin{equation}
\label{eq:cover-general-aggregated-bound-3}
L_T^{\vphantom\star}(w_T^\star) - \min_{w \in \Delta_d} \sum_{t \in [T]} \ell_t(w) \le \lambda \, (d-1) \log\left(2e\frac{T + \lambda d}{\lambda}\right).
\end{equation}
Combining~\eqref{eq:cover-general-aggregated-bound}--\eqref{eq:cover-general-aggregated-bound-3} yields the result.
\qed

\subsection{Proof of Proposition~\ref{prop:Dikin-relaxation-tightness}}

On the one hand, the mere feasibility of~$\wb\phi_{t}$ in surrogate problem~\eqref{eq:Dikin-upper-approx} implies, through~\eqref{eq:expected-quadratic-bound}, that
\begin{equation}
\label{eq:Dikin-relaxation-using-upper-bound}
F_{t-1}[\wb\phi_{t}] \le \wb F_{t-1}[\wb\phi_{t}].
\end{equation}
Now, let us define~$(\wbb w, \wbb\phi_{t})$ as the optimal solution to the following minimization problem:
\begin{equation}
\label{eq:Dikin-lower-relaxation}
\min_{
\scriptsize
	\begin{aligned} 
		&\hat w \in \Delta_d, \; \E_{\phi}[w] = \hat w, \\
	 	&\phi \in \Supp(\cE_{t-1,\,1/2}(\hat w))
	\end{aligned}
	} 
\underbrace{L_{t-1}(\hat w) + \frac{1}{5}\Tr\left( \nabla^2 L_{t-1}(\hat w) \, \Cov[\phi] \,\right) - \mu \Ent[\phi]}_{:= \; \wbb F_{t-1}[\phi]}.
\end{equation}
Moreover, let
$
\wt \phi_{t}(\cdot) := 2 \wbb \phi_{t}(\wbb w + 2(\cdot-\wbb w)),
$
i.e.~$\wt \phi_{t}$ is the distribution of~$\wbb w + \half(w-\wbb w)$ with~$w \sim \wbb\phi_{t}$. 
Then~$\E_{w \sim \wt \phi_{t}}[w] = \wbb w$ and also~$\wt\phi_{t} \in \Supp(\cE_{t-1,\,1/4}(\wbb w))$---in other words,~$(\wbb w, \wt\phi_{t})$ is feasible in~\eqref{eq:Dikin-upper-approx}, so
\begin{equation}
\label{eq:Dikin-relaxation-upper-optimality}
\wb F_{t-1}[\wb\phi_{t}] \le \wb F_{t-1}[\wt\phi_{t}]. 
\end{equation}
Now: observe that~$\Cov[\wt \phi_{t}] = \frac{1}{4}\Cov[\wbb \phi_{t}]$ and~$\Ent[\wt \phi_{t}] = \Ent[\wbb \phi_{t}] - (d-1)\log(2)$, 
whence~\eqref{eq:expected-quadratic-bound} gives
\begin{equation}
\label{eq:Dikin-relaxation-rescaling-error}
\wb F_{t-1}[\wt\phi_{t}] \le \wbb F_{t-1}[\wbb\phi_{t}] + \mu(d-1)\log(2). 
\end{equation}
Finally, let~$(w^o, \phi_{t}^o)$ be the optimal solution to the optimization problem
\begin{equation}
\label{eq:Dikin-middle-relaxation}
\min_{
\scriptsize
	\begin{aligned} 
		&\hat w \in \Delta_d, \; \E_{\phi}[w] = \hat w, \\
	 	&\phi \in \Supp(\cE_{t-1,\,1/2}(\hat w))
	\end{aligned}
	} F_{t-1}[\phi].
\end{equation}
Since~$(w^o, \phi_{t}^o)$ is also a feasible solution to~\eqref{eq:Dikin-lower-relaxation}, we have that
$
\wbb F_{t-1}[\wbb\phi_{t}] 
\le \wbb F_{t-1}[\phi_{t}^o] 
\le F_{t-1}[\phi_{t}^o].
$
In combination with~\eqref{eq:Dikin-relaxation-using-upper-bound} and~\eqref{eq:Dikin-relaxation-upper-optimality}--\eqref{eq:Dikin-middle-relaxation} this results in
\begin{equation}
\label{eq:Dikin-restriction-accuracy}
F_{t-1}[\wb\phi_{t}]
\le 
	\min_{
		\scriptsize
		\begin{aligned} 
			&\hat w \in \Delta_d, \; \E_{\phi}[w] = \hat w, \\
	 		&\phi \in \Supp(\cE_{t-1,\,1/2}(\hat w))
		\end{aligned}
	} F_{t-1}[\phi] \;  + \; \mu(d-1)\log(2).
\end{equation}
Thus, to prove the proposition it suffices to bound the minimum in~\eqref{eq:Dikin-middle-relaxation} from above in terms of the minimum in~\eqref{eq:universal-portfolios-gibbs}. To this end, recall that~$\phi_{t}$, as defined in~\eqref{def:cover-general}, is the optimal solution to~\eqref{eq:universal-portfolios-gibbs}; let~$w^\star = \argmin_{w \in \Delta_d} L_{t-1}(w)$, and let~$\phi_{t}^\trc$ be the truncation of~$\phi_{t}$ to~$\cE_{t-1,\,1/8}(w^\star)$---in other words,
\begin{equation}
\label{def:cover-truncated}
\phi_{t}^\trc(w) = \frac{\exp \left(-\frac{1}{\mu} L_{t-1}(w)\right)}{\normalsize\int_{\cE_{t-1,\,1/8}(w^\star)} \exp \left(-\frac{1}{\mu} L_{t-1}(w')\right) d w'} , \quad w \in \cE_{t-1,\,1/8}(w^\star).
\end{equation}
Clearly,~$\phi_{t}^\trc$ is supported on~$\cE_{t-1,\,1/2}(\hat w^\trc)$ where~$\hat w^\trc = \E_{\phi_{t}^\trc}[w]$ is its expectation---in other words, the pair~$(\hat w^\trc, \phi_{t}^\trc)$ is feasible in~\eqref{eq:Dikin-middle-relaxation}. 
Indeed, for any~$w \in \cE_{t-1,\,1/8}(w^\star)$ we have that
\[
\begin{aligned}
\|w - \hat w^\trc \|_{\nabla^2 L_{t-1}(\hat w^\trc)} 
&\le 2 \|w - \hat w^\trc \|_{\nabla^2 L_{t-1}(w^\star)} \\
&\le 2 \|w - w^\star \|_{\nabla^2 L_{t-1}(w^\star)} + 2\|\hat w^\trc - w^\star \|_{\nabla^2 L_{t-1}(w^\star)} \le 1/2,
\end{aligned}
\]
where we used that function~$L_{t-1}^\res(v)$ is~1-self-concordant whenever~$\lambda \ge 1$---see~Corollary~\ref{cor:SC-of-L} in Appendix~\ref{apx:self-conc} for the proof of this result---and applied Lemma~\ref{lem:SC-hessian-bounds}. 
By~\eqref{eq:Dikin-restriction-accuracy}, this implies that
\begin{equation}
\label{eq:before-cover-truncation-error-lemma}
F_{t-1}[\wb\phi_{t}] \le F_{t-1}[\phi_{t}^\trc]  +  \mu(d-1)\log(2).
\end{equation}
As such, it remains to show that replacing~$\phi_{t}$ with~$\phi_{t}^\trc$ does not lead to a dramatic increase of~$F_{t-1}$.
\begin{lemma}
\label{lem:cover-truncation-error} 
For~$\phi_{t}$,~$F_{t-1}$,~$\phi_{t}^\trc$ as in~\eqref{def:cover-general}--\eqref{eq:universal-portfolios-gibbs} and~\eqref{def:cover-truncated}, the following holds as long as~$\lambda \ge 1$:
\[
F_{t-1}[\phi_{t}^\trc] \le F_{t-1}[\phi_{t}] 
\, + \, 1.5 \mu (d-1) \log \left(T+\lambda d\right) 
\, + \, 3.2\mu(d+1) 
\, + \, 0.1.
\]
\end{lemma}
This lemma, proved in the next section, gives the desired result via~\eqref{eq:before-cover-truncation-error-lemma}, since~$\log(2) < 0.7$.
\qed

\subsection{Proof of Lemma~\ref{lem:cover-truncation-error}}
\label{apx:truncation-error-proof}

For the sake of generality we shall consider truncation to~$\cE_{t-1,r}(w^\star)$ with arbitrary~$r \le 1/2$, and put~$r = 1/8$ post-hoc.
By duality between the negative entropy and log-partition function (cf.~\eqref{eq:universal-portfolios-gibbs}),
\[
\begin{aligned}
	F_{t-1}[\phi_{t}]
&=  -\mu \log \int_{\Delta_d}  \exp \left(-\frac{1}{\mu} L_{t-1}(w)\right) dw, \\
	F_{t-1}[\phi_{t}^{\trc}] 
&=  -\mu \log \int_{\cE_{t-1,r}(w^\star)}  \exp \left(-\frac{1}{\mu} L_{t-1}(w)\right) dw.
\end{aligned}
\]
Here the second integral is larger:~$\cE_{t-1,r}(w^\star) \subseteq \Delta_d$ by Lemma~\ref{lem:SC-sandwich} and Corollary~\ref{cor:SC-of-L}. 
Our goal is to show the reverse inequality up to the appropriate error term.
To this end, we first note that
\begin{equation}
\label{eq:likelihood-ratio}
\begin{aligned}
\exp \left(\frac{1}{\mu} \left(F_{t-1}[\phi_{t}^{\trc}] - F_{t-1}[\phi_{t}]\right) \right)
= \frac{\int_{\Delta_d}  \exp \left(-\frac{1}{\mu} [L_{t-1}(w) - L_{t-1}(w^\star)] \right) dw}{\int_{\cE_{t-1,r}(w^\star)}  \exp \left(-\frac{1}{\mu} [L_{t-1}(w) - L_{t-1}(w^\star)]\right) dw}. 
\end{aligned}
\end{equation}
Since~$L_{t-1}(w) \ge L_{t-1}(w^\star)$, the numerator can be upper-bounded by the~surface area of~$\Delta_d$:
\[
\int_{\Delta_d}  \exp \left(-\frac{1}{\mu} [L_{t-1}(w) - L_{t-1}(w^\star)] \right) dw
\le \Area(\Delta_d) = \sqrt{d} \Vol(\bDelta_{d-1}) = \frac{\sqrt{d}}{\Gamma(d)}. 
\]
By the first-order optimality conditions~$\lang \nabla L_{t-1}(w^\star), w - w^\star \rang = 0$ for~$w \in \Delta_d$, whence by Lemma~\ref{lem:SC-sandwich}
\[
L_{t-1}(w) - L_{t-1}(w^\star) 
\le \frac{4}{5} \|w - w^\star\|_{\nabla^2 L_{t-1}(w^\star)}^2
\quad \forall w \in \cE_{t-1,r}(w^\star),
\]
where we used that~$r \le 1/2$. 
As such, we can lower-bound the denominator in~\eqref{eq:likelihood-ratio} as follows:
\[
\int_{\cE_{t-1,r}(w^\star)}  \exp \left(-\frac{1}{\mu} [L_{t-1}(w) - L_{t-1}(w^\star)]\right) dw 
\ge \exp \left( - \frac{4r^2}{5\mu} \right) r^{d-1} \Area(\cE_{t-1,1}(w^\star)).
\]
Now, as in~Corollary~\ref{cor:SC-of-L} let~$L_{t-1}^{\res}: \R^{d-1} \to \wb \R$ be the reparametrized restriction of~$L_{t-1}$ to~$\Aff_d$. 
Since~$\det(A^\top A) = d$, we have that
\[
\Area(\cE_{t-1,1}(w^\star)) = \sqrt{d} \Vol(\cE_{t-1,1}^{\res}(v^\star))
\]
where~$v^\star$ is such that~$Av^\star  + e_d = w^\star$, and ellipsoid~$\cE_{t-1,r}^{\res}(v^\star)$ similarly corresponds to~$\cE_{t-1,r}(w^\star)$:
\[
\cE_{t-1,r}^{\res}(v^\star) := \{v \in \R^{d-1}: \|v - v^\star\|_{\nabla^2 L_{t-1}^{\res}(v^\star)} \le r\} \subset \R^{d-1}.
\]
Using the expression for the volume of an ellipsoid in~$\R^{d-1}$, we have that
\[
\Vol(\cE_{t-1,1}^{\res}(v^\star)) = \frac{\pi^{\frac{d-1}{2}}}{\Gamma(\frac{d+1}{2}) \sqrt{\det ( \nabla^2 L_{t-1}^{\res}(\wb v) ) } }.
\] 
Moreover, since~$\nabla^2 L_{t-1}^{\res}(v^\star) = A^\top \nabla^2 L_{t-1}(w^\star) A$ and~$A^\top A = \Imtx_{d-1}^{\vphantom\top} + \ones_{d-1}^{\vphantom\top} \ones_{d-1}^\top$, so that~$\|A\| = \sqrt{d}$,
\[
\begin{aligned}
\det(\nabla^2 L_{t-1}^{\res}(v^\star)) 
\le 
\| \nabla^2 L_{t-1}^{\res}(v^\star) \|^{d-1}
\le (d \, \| \nabla^2 L_{t-1}(w^\star)\|)^{d-1}
&\le \left(\frac{d(t-1+\lambda d)^3}{\lambda^2} \right)^{d-1} \\
&\le \left(d (t-1+\lambda d)^3\right)^{d-1};
\end{aligned}
\]
here we bounded~$\| \nabla^2 L_{t-1}(w^\star) \|$ by using that~$\min_{i \in [d]} w^\star[i] \ge \frac{\lambda}{t-1 + \lambda d}$ (cf.~Lemma~\ref{lem:entries-cutoff}) and invoking estimate~\eqref{eq:Hessian-bound-from-entries-magnitude} used in step~\proofstep{1} of the proof of Theorem~\ref{th:volumetric}.
Returning to~\eqref{eq:likelihood-ratio} we arrive at
\[
\begin{aligned}
F_{t-1}[\phi_{t}^{\trc}] -  F_{t-1}[\phi_{t}] 
&\le 
	\frac{3}{2} \mu(d-1)\log \left(\frac{T+\lambda d}{r^{2/3} \pi^{1/3}} \right)  
	+
	\mu \log \left( \frac{\Gamma(\frac{d+1}{2}) d^{\frac{d-1}{2}}}{\Gamma(d)} \right) 
	+ 
	\frac{4}{5} r^2 \\
&\le 
	\frac{3}{2} \mu(d-1) \log \left(\frac{T+\lambda d}{r} \right) +  \mu(d+1) + \frac{4}{5} r^2.
\end{aligned}
\]
Here in the last step we did a series of simple estimates to bound the term depending solely on~$d$. The result follows by plugging in~$r = 1/8$.
\qed

\subsection{Transition from~\eqref{eq:relaxation-post-max-ent} to~\eqref{eq:vb-ftrl-derived-from-cover}}
\label{app:gaussian-reduction-solve-in-sigma}
Recall that~$A^\top \nabla^2 L_{t-1}(\hat w) A = \nabla^2 L_{t-1}^{\res}(\hat v)$, cf.~\eqref{def:affine-reparametrization-grad-and-hess}, and note that the~$\bSigma$~component of the objective gradient in~\eqref{eq:relaxation-post-max-ent} is
$
\frac{4}{5} \nabla^2 L_{t-1}^{\res}(\hat v) - \frac{\mu}{2} \bSigma^{-1}. 
$
The corresponding first-order optimality condition reads:
\[
\Tr\left[(\bSigma - \hat \bSigma) \left(\nabla^2 L_{t-1}^{\res}(\hat v) - \frac{5\mu}{8}\hat\bSigma^{-1} \right)\right] \ge 0 
\quad \text{for all} \quad 
0 \preceq \bSigma \preceq \frac{1}{16} \nabla^2 L_{t-1}^{\res}(\hat v)^{-1}. 
\]
Plugging in~$\hat \bSigma = \frac{1}{16} \nabla^2 L_{t-1}^{\res}(\hat v)^{-1}$ and dividing over~$1-10\mu < 0$, the above condition translates into
\[
\Tr\left[ \bSigma \nabla^2 L_{t-1}^{\res}(\hat v) - \frac{1}{16} \Imtx_{d-1}\right] \le 0 \quad \text{for all} \quad 0 \preceq \bSigma \preceq \frac{1}{16} \nabla^2 L_{t-1}^{\res}(\hat v)^{-1}. 
\]
From the monotonicity of trace with respect to~$\preceq$ ordering, it follows that this condition is indeed satisfied.
Plugging~$\hat \bSigma$ into~\eqref{eq:relaxation-post-max-ent} and neglecting the constant terms, we finally arrive at~\eqref{eq:vb-ftrl-derived-from-cover}.
\qed

\subsection{Proof of Proposition~\ref{prop:Gaussian-relaxation-tightness}}
\label{app:Gaussian-relaxation-tightness-proof}

\proofstep{1}.
Recall that the optimization problem
\begin{equation}
\label{eq:relaxation-pre-max-ent}
\min_{
	\scriptsize
	\begin{aligned}
		\hat w \in \Delta_{d}, \quad 0 \preceq \bSigma \preceq {\frac{1}{16}}^{\vphantom{A^A}} (A^\top \nabla^2 L_{t-1}(\hat w)A)^{-1},\\
		\phi^\res: \; A\E_{v \sim \phi^{\res}}[v] + e_d = \hat w, \; \Cov[\phi^{\res}] = \bSigma
	\end{aligned}		
	}  
\underbrace{L_{t-1}(\hat w) + \frac{4}{5} \Tr \left[ \bSigma \, A^\top \nabla^2 L_{t-1}(\hat w) A \right] - \mu \Ent[\phi^\res]}_{= \, \wb F_{t-1}[\phi] + \frac{\mu}{2} \log(d)}
\end{equation}
is equivalent to~\eqref{eq:relaxation-post-max-ent} in the natural sense:~$(w_{t}, \bSigma_{t}) \in \Delta_d \times \Sym_{+}^{d-1}$ is optimal in~\eqref{eq:relaxation-post-max-ent} whenever~$\cN(v_{t}, \bSigma_{t})$, with~$v_{t} \in \bDelta_{d-1}$ such that~$Av_{t} + e_d = w_{t}$, is optimal in~\eqref{eq:relaxation-pre-max-ent}. 
Now, let~$(\wb w, \wb \phi_{t})$ and~$(w_{t}, \bSigma_{t})$ be optimal in~\eqref{eq:Dikin-upper-approx} and~\eqref{eq:relaxation-post-max-ent} correspondingly, 
so that~$\cN(v_{t},\bSigma_{t})$, with~$v_{t} \in \bDelta_{d-1}$ as above, is optimal in~\eqref{eq:relaxation-pre-max-ent}. 
Observe that~$(\wb w, \wb \phi_{t})$ is feasible in~\eqref{eq:relaxation-pre-max-ent} due to~\eqref{eq:soft-feasibility-check}. 
Since the objective in~\eqref{eq:relaxation-pre-max-ent} amounts to~$\wb F_{t-1}[\phi] + \frac{\mu}{2} \log(d)$ with~$\phi \in \Supp(\Aff_{d})$ associated with~$\phi^{\res}$~via~\eqref{def:marginal-correspondence}, we have, for the distribution
\[
g_{t} := \cN(w_{t},A\bSigma_{t}A^\top),
\] 
that
\begin{equation}
\label{eq:Gaussian-relaxation-tightness-proof-1}
\wb F_{t-1}[g_{t}] \le \wb F_{t-1}[\wb \phi_{t}].
\end{equation}
\proofstep{2}.
Next, let us bound~$\wb F_{t-1}[g_{t}]$ from below.
From our derivation in~Appendix~\ref{app:gaussian-reduction-solve-in-sigma} it follows that 
\begin{equation}
\label{eq:Sigma-t+1}
\bSigma_{t}^{-1} = 16 A^\top \nabla^2 L_{t-1}(w_{t}) A,
\end{equation}
therefore the level sets of~$g_{t}$ are precisely the Dikin ellipses~$\cE_{t-1,r}(w_{t})$ with various~$r \ge 0$, cf~\eqref{def:Dikin-ellipse}.
Now, define~$g_{t}^\trc$ as the truncation of~$g_{t}$ to the specific Dikin ellipse~$\cE_{t-1,\,1/2}(w_{t})$---in other words,
\begin{equation}
\label{eq:Gaussian-truncated-on-cov-ellipse}
g_{t}^\trc(w) = \frac{g_{t}(w)}{\int_{\cE_{t-1,\,1/2}(w_{t})} g_{t}(w') dw'}, \quad w \in \cE_{t-1,\,1/2}(w_{t}). 
\end{equation}
Since~$\E_{g_{t}^\trc}[w] = w_{t}$ and~$g_{t}^\trc \in \Supp(\cE_{t-1,\,1/2}(w_{t}))$, the bounds in~\eqref{eq:expected-quadratic-bound} are valid for~$\phi = g_{t}^\trc$, and so
\begin{equation}
\label{eq:Gaussian-relaxation-tightness-proof-2}
F_{t-1}[g_{t}^\trc] \le \wb F_{t-1}[g_{t}^\trc]. 
\end{equation}
\proofstep{3}.
Let us now bound~$\wb F_{t-1}[g_{t}^\trc]$ from above in terms of~$\wb F_{t-1}[g_{t}]$.
To this end, we first observe that
\begin{equation}
\label{eq:covariance-change-under-truncation}
\Cov[g_{t}^\trc] \preceq \Cov[g_{t}].
\end{equation}
Indeed: truncating~$g_{t}$ can also be seen as transporting the mass from each ellipse~$\cE_{t-1,s}(w_{t}), s > 1/2,$ to~$\cE_{t-1,r}(w_{t})$ for some~$r = r(s) \le 1/2$. 
As such, it only remains to control the change of entropy because of truncation. 
The following lemma, proved in the next section, allows us to do so.
\begin{lemma}
\label{lem:Gaussian-trunc-entropy}
Given~$\hat v \in \R^n,$~$\bSigma \in \Sym_{++}^{n},$ let~$\gfrak(v)$ be the density of~$\cN(\hat v,\bSigma)$. Consider the density
\[
\gfrak^{(r)}(v) = \frac{\gfrak(v)}{\int_{\|v'-\hat v\|_{\bSigma^{-1}} \le r} \gfrak(v') dv'}, \quad \forall v \in \R^n: \|v-\hat v\|_{\bSigma^{-1}} \le r,
\]
i.e. the truncation of~$\gfrak$ to the covariance ellipsoid~$\{v \in \R^n: \|v-\hat v\|_{\bSigma^{-1}} \le r\}$ of radius~$r > 0$. Then
\[
\Ent[\gfrak] - \Ent[\gfrak^{(r)}] \le \frac{n}{2} \log\left(\frac{2e(n+1)}{r^2}\right) + \frac{r^2}{2}. 
\]
\end{lemma}
\noindent While this bound becomes vacuous when~$r \to \infty$, for our purposes it suffices, as in our case~$r < 1$. 
Namely, we apply the lemma to~$\cN(v_{t},\bSigma_{t})$ in the role of~$\gfrak$,
so that its truncation to~$\cE_{t-1,\,1/2}^{\res}(v_{t})$, cf.~\eqref{eq:Gaussian-truncated-on-cov-ellipse}, corresponds to~$\gfrak^{(1/8)}$, cf.~\eqref{eq:Sigma-t+1}. This results in
\[
\Ent[g_{t}] - \Ent[g_{t}^\trc] \le \frac{d-1}{2}  \log\left(128 e d\right) + \frac{1}{128}, 
\]
where we applied the identity~$\Ent[\phi^\ones] = \Ent[\phi] + \half\log(d)$ to~$\phi \in \{g_{t}^{\vphantom\trc}, g_{t}^\trc\}$. 
In view of~\eqref{eq:covariance-change-under-truncation} we get
\[
\wb F_{t-1}[g_{t}^\trc] \le \wb F_{t-1}[g_{t}] + \frac{\mu(d-1)}{2}  \log\left(128 e d\right) + \frac{\mu}{128}. 
\]
The desired result follows by combining this with~\eqref{eq:Gaussian-relaxation-tightness-proof-2},~\eqref{eq:Gaussian-relaxation-tightness-proof-1}, and Proposition~\ref{prop:Dikin-relaxation-tightness}. 
\qed

\subsection{Proof of Lemma~\ref{lem:Gaussian-trunc-entropy}}
Recall that
$
\Ent[\gfrak] = \frac{1}{2} \log\det(\bSigma) + \frac{n}{2}\log(2\pi) + \frac{n}{2}. 
$
On the other hand, by rewriting~$\gfrak^{(r)}$ in the form
\[
\gfrak^{(r)}(v) = \frac{1}{\sqrt{(2\pi)^{n} C_{r}(\bSigma)}} \exp\left(-\half\|v-\hat v\|_{\bSigma^{-1}}^2\right)
\]
where~$C_r(\bSigma)$ does not depend on~$\hat v$, we estimate
$
\Ent[\gfrak^{(r)}] \ge \half \log [C_r(\bSigma)] + \frac{n}{2} \log(2\pi)
$
and arrive at
\begin{equation}
\label{eq:Gaussian-trunc-entropy-difference}
\Ent[\gfrak^{(r)}] - \Ent[\gfrak] \ge \frac{1}{2} \log [C_r(\bSigma)] - \frac{1}{2} \log\det(\bSigma) - \frac{n}{2}. 
\end{equation}
Now, observe that, putting~$\hat v = 0$ w.l.o.g.,
\[
\begin{aligned}
C_r(\bSigma)^{1/2}
= \frac{1}{(2\pi)^{n/2}} \int_{\|v\|_{\bSigma^{-1}} \le r} \exp\left(-\half\|v\|_{\bSigma^{-1}}^2\right) dv 
&\ge \frac{e^{-r^2/2}}{(2\pi)^{n/2}} \det(\bSigma^{1/2}) \frac{\pi^{n/2} r^{n}}{\Gamma(\frac{n}{2}+1)}
\end{aligned}
\]
where we used the expression~$\frac{\pi^{n/2}}{\Gamma(\frac{n}{2}+1)}$ for the volume of the unit~$\ell_2$-ball in~$\R^n$. Whence we arrive at
\[
\begin{aligned}
\half \log C_r(\bSigma) - \half \log\det(\bSigma) 
&\ge -\frac{r^2}{2} -\frac{n}{2} \log(2) + n\log(r) - \log \left[ \Gamma\left(\frac{n}{2}+1\right) \right] \\
&\ge -\frac{r^2}{2} -\frac{n}{2} \log(2) + n\log(r) - \frac{n}{2} \log(n+1),
\end{aligned}
\]
and the desired result follows via~\eqref{eq:Gaussian-trunc-entropy-difference}. 
\qed
\section{Differential properties of leverage scores and barrier functions}
\label{apx:derivatives}

In this section, we recall some differential properties of the Gram matrix~$\bPi(w)$, and barrier functions~$V_t(w), P_t(w)$ used in~\OurAlgo{} and in its analysis.
These results are adaptations or easy corollaries of those in~\cite{vaidya1989new,anstreicher1997volumetric}; 
however, the existing results in these works can only be applied in the special case~$\lambda = \mu = 1$, where~$V_t^\res(\cdot)$ and~$P_t^\res(\cdot)$ become, respectively, the volumetric and ``hybrid'' barriers (in the terminology of~\cite{anstreicher1997volumetric}) for the set of linear constraints corresponding to the set of vectors~$\{A^\top x_{\tau}, \; \tau \in [t]^+\}$ in the notation defined in~\eqref{def:abridged-hess-and-grad}--\eqref{def:Gram-hat-entries}.
\footnote{The polyhedron defined by this set of constraints is~$\bDelta_{d-1}$, same as without~$x_{\tau}$'s for~$\tau \in [t]$. 
But functions~$L_t, V_t, P_t$ do depend on these~$x_\tau$'s, so their names adopted in~\cite{anstreicher1996large}---e.g.~``volumetric barrier'' for~$V_t$---should be used with caution.}
Formal reduction to this special case from~$\lambda \ge 1, \mu > 0$ proved to be elusive.
For this reason, and to make the paper self-contained, here we reprove these results assuming only that~$\lambda, \mu > 0$, and in some cases that~$\lambda \ge 1$.


\paragraph{Preliminaries and notation.}
To streamline the exposition, in this section we fix~$t \in [T]$ and use the notation defined in~\eqref{def:abridged-hess-and-grad}--\eqref{def:Gram-hat-entries}, which we are now repeating here for convenience.
For any~$\tau \in [t]^+$:
\begin{equation}
\label{def:abridged-hess-and-grad-apx}
\begin{aligned}
\Hmtx_t(w) &:= A^\top \nabla^2 L_t(w) A, \quad &\bNabla_\tau(w) &:= A^\top \nabla \ell_\tau(w), \\
\hat\Hmtx_t &:= \Hmtx_t(w_t), \quad &\hat\bNabla_\tau &:= \bNabla_\tau,
\end{aligned}
\end{equation}
where~$[t] \cup \{-1, \,\dots, -d\}$ is the extended index set (with~$\tau < 0$ corresponding to~$x_{\tau} := e_{-\tau}$). 
Besides,
\begin{equation*}
\lambda_\tau := \lambda^{\ind\{\tau < 0\}}, 
\quad 
\end{equation*}
\begin{equation}
\label{eq:hess-as-a-sum-apx}
\text{so that} \quad
\Hmtx_t(w) = \sum_{\tau \in [t]^+} \lambda_\tau^{\vphantom\top} \bNabla_\tau^{\vphantom\top}(w) \bNabla_\tau^\top(w)
\quad \text{and} \quad
\hat\Hmtx_t = \sum_{\tau \in [t]^+} \lambda_\tau^{\vphantom\top} \hat\bNabla_\tau^{\vphantom\top} \hat\bNabla_\tau^\top. 
\end{equation}
We let~$\bPi(w) \in \Sym^{t+d}_{+}$ be the Gram matrix with rows and columns indexed over~$[t]^+$, and with entries
\begin{equation}
\label{def:Gram-entries-apx}
\pi_{\tau,\nu}(w) := \sqrt{\lambda_\tau \lambda_\nu} \big\langle \bNabla_\tau(w), \bNabla_\nu(w) \big\rangle_{\Hmtx_t(w)^{-1}}.
\end{equation}
We also define~$\hat{\bPi} := \bPi(w_t)$, whose entries are thus
\begin{equation}
\label{def:Gram-hat-entries-apx}
\hat\pi_{\tau,\nu} := \sqrt{\lambda_\tau \lambda_\nu} \big\langle \hat\bNabla_\tau, \hat\bNabla_\nu \big\rangle_{\hat\Hmtx_t^{-1}}. 
\end{equation}
For brevity, we refer to the diagonal entries of both these matrices with a single index. Note that
\begin{equation}
\label{eq:leverage-score-explicit-apx}
\pi_{\tau}(w) = \frac{\lambda_\tau}{(x_\tau^\top w)^{2}} \| A^\top x_\tau \|^2_{\Hmtx_t(w)^{-1}}.
\end{equation}

\subsection{Gram matrix~$\bPi(w)$ and leverage scores}

\begin{lemma}
\label{lem:bound_leverage_score}
For any~$t \in [T]$,~$\tau, \nu \in [t]^+$, and~$w \in \R^d_{++}$, one has the following estimate for~$\pi_{\tau,\nu}$:
\begin{equation}
|\pi_{\tau,\nu}(w)| \le \sqrt{\frac{1}{(1+\lambda)^{\vphantom{\frac{1}{2}}\ind\{\tau > 0\} + \ind\{\nu > 0\}}}}. 
\end{equation}
\end{lemma}
\begin{proof}
By~\eqref{def:Gram-entries-apx} and Cauchy-Schwarz,
$
|\pi_{\tau, \nu}(w)| \le \sqrt{\pi_\tau(w) \pi_\nu(w)}
$
for all~$\tau, \nu \in [t]^+$.
Clearly, for~$i \in [d]$ one has~$\nabla^2 L_t(w) \succeq \lambda \nabla^2 \ell_{-i}(w)$, whence~$\pi_\tau(w) \le 1$ for~$\tau < 0$.
Similarly, but also using Lemma~\ref{lem:rank-one-comparison}, 
\[
\nabla^2 L_t(w) \succeq \nabla^2 \ell_\tau(w) + \lambda  \nabla^2 R(w) \succeq (1+\lambda) \nabla^2 \ell_\tau(w), \quad \forall \tau \in [t].
\]
As a result,~$\pi_{\tau} \le \frac{1}{1+\lambda}$ for~$\tau \in [t]$.
The result follows.
\end{proof}

\begin{lemma}[{Generalization of~\cite[Claim 3]{vaidya1989new}}]
\label{lem:alternate_form_leverage_score}
For~$t \in [T]$,~$\tau_1, \tau_2 \in [t]^+$, and~$w \in \R^d_{++}$, one has:
\begin{align}
\label{eq:leverage-score-identity}
\pi_{\tau_1,\tau_2}(w) 
	= \sum_{\nu \in [t]^+}  \pi_{\tau_1,\nu}(w) \, \pi_{\nu,\tau_2}(w).
\end{align}
In other words,~$\bPi(w)$ is a projection matrix. 
In particular,~$\pi_{\tau}(w) = \sum_{\nu \in [t]^+}  \pi_{\tau,\nu}(w)^2$ for~$\tau \in [t]^+$.
\end{lemma}

\begin{proof}
Using the abridged notation defined in~\eqref{def:abridged-hess-and-grad-apx}, and omitting the dependency on~$w$ for brevity, 
\begin{align*}
\pi_{\tau_1,\tau_2}
\stackrel{\eqref{def:Gram-entries-apx}}{=}
	\sqrt{\lambda_{\tau_1}^{\vphantom\top} \lambda_{\tau_2}^{\vphantom\top}} \bNabla_{\tau_1}^\top \Hmtx_t^{-1} \bNabla_{\tau_2}^{\vphantom\top}
&\;=
	\sqrt{\lambda_{\tau_1}^{\vphantom\top} \lambda_{\tau_2}^{\vphantom\top}} \bNabla_{\tau_1}^\top \Hmtx_t^{-1} \Hmtx_t^{\vphantom{-1}} \Hmtx_t^{-1}  \bNabla_{\tau_2}^{\vphantom A}  \\
&\stackrel{\eqref{eq:hess-as-a-sum-apx}}{=} 
	\sqrt{\lambda_{\tau_1}^{\vphantom\top} \lambda_{\tau_2}^{\vphantom\top}} \bNabla_{\tau_1}^\top \Hmtx_t^{-1} 
\left(
	\sum_{\nu \in [t]^+} \lambda_{\nu}^{\vphantom\top} \bNabla_\nu^{\vphantom\top} \bNabla_\nu^\top 
\right) 
\Hmtx_t^{-1} \bNabla_{\tau_2}^{\vphantom\top} \\
&\stackrel{\eqref{def:Gram-entries-apx}}{=}
	\sum_{\nu \in [t]^+}  \pi_{\tau_1,\nu} \, \pi_{\nu,\tau_2}.
\qedhere
\end{align*}
\end{proof}

\begin{lemma}[Gradient of a leverage score]
\label{lem:leverage-score-grad}
For~$t \in [T]$,~$\tau \in [t]^+$, and~$w \in \R^d_{++}$, it holds that
\begin{align}
\label{eq:leverage-score-grad}
\frac{1}{2} \nabla \pi_\tau(w) 
&= 
    \pi_\tau(w) \, \nabla \ell_\tau(w) \,
        - \sum_{\nu \in [t]^+} \pi_{\tau,\nu}(w)^2 \, \nabla \ell_\nu(w).
\end{align}
\end{lemma}

\begin{proof}
This is a byproduct of an explicit formula for~$\nabla^2 V_t(w)$ in the proof of Lemma~\ref{lem:volumetric-grad-and-hess}. 
\end{proof}

\subsection{Gradient and Hessian of the volumetric barrier}
Following~\cite{vaidya1989new} and~\cite{anstreicher1997volumetric}, we shall now derive an explicit formula for the gradient of function~$V_t(\cdot)$ on~$\Delta_d$, cf.~\eqref{def:VB-regularizer}, estimate its Hessian, 
and verify the expression for~$\nabla \pi_\tau(w)$ in Lemma~\ref{lem:leverage-score-grad}, cf.~\eqref{eq:product-rule-post}--\eqref{eq:newton-decrement-derivative}. 
The lower bound on~$\nabla^2 V_t(w)$ in~\eqref{eq:V-hess-bound} shows, in particular, that~$V_t(\cdot)$ is strictly convex.

\begin{lemma}[{\cite[Lemmas 1--3]{vaidya1989new}}]
\label{lem:volumetric-grad-and-hess}
For~$t \in [T]$ and~$w \in \R^d_{++}$, the gradient of~$V_t(w)$, cf.~\eqref{def:VB-regularizer}, is 
\begin{equation}
\label{eq:V-grad}
\nabla V_t(w) 
=  \sum_{\tau \in [t]^+} \pi_{\tau}(w) \, \nabla \ell_\tau(w),
\end{equation}
and its Hessian satisfies
\begin{equation}
\label{eq:V-hess-bound}
\Qmtx_t(w) \preceq \nabla^2 V_t(w) \preceq 3\Qmtx_t(w) 
\quad \text{where} \quad
\Qmtx_t(w) 
= \sum_{\tau \in [t]^+} \pi_{\tau}(w) \, \nabla^2 \ell_\tau(w).
\end{equation}
\end{lemma}

\begin{remark}
\label{rem:V-hessian}
A weaker form of the upper bound in~\eqref{eq:V-hess-bound}, namely~$\nabla^2 V_t(w) \preceq 5 \Qmtx_t(w)$, was first obtained by Vaidya in~\cite{vaidya1989new}; 
later on,~\cite{anstreicher1996large} improved it to~$\nabla^2 V_t(w) \preceq 3 \Qmtx_t(w)$. (In both cases, it was also assumed that~$\lambda = 1$.) 
For the sake of clarity, here we prove the bound with factor~$5$.
The improved bound follows by applying Schur's product theorem (see e.g.~\cite{bapat1997nonnegative}) to the Hadamard-product representation of the matrix~$\Smtx_t(w) = \frac{3}{2} \Qmtx_t(w) - \frac{1}{2}\nabla^2 V_t(w)$, see~\eqref{eq:S-formula} in the proof of Lemma~\ref{lem:hybrid-hessian-stability}.
\end{remark}

\begin{proof}

\proofstep{1}. 
Note that the entries of~$\Hmtx_t(w) = A^\top \nabla^2 L_t(w) A$, for~$i, j \in [d-1]$, are given by
\begin{equation*}
[\Hmtx_t(w)]_{ij} = \sum_{\tau \in [t]^+} \frac{\lambda_{\tau}}{(x_{\tau}^\top w)^2} \big[ A^\top x_\tau \big]_{i}  \, \big[ A^\top x_\tau \big]_{j}.
\end{equation*}
As such,
\begin{equation}
\label{eq:hess-entries-grad}
\nabla \left( [\Hmtx_t(w)]_{ij} \right) = -\sum_{\tau \in [t]^+} \frac{2\lambda_{\tau}}{(x_{\tau}^\top w)^3} \big[ A^\top x_\tau \big]_{i}  \, \big[ A^\top x_\tau \big]_{j} \, x_{\tau}.
\end{equation}
Using this formula, the composition rule now allows us to verify~\eqref{eq:V-grad} by computing~$\nabla V_t(w)$ directly:
\[
\begin{aligned}
\nabla V_t(w)
&\;= \frac{1}{2} \sum_{i,j \in [d-1]} \left[ \Hmtx_t(w)^{-1} \right]_{ij}  \; \nabla \left(  [\Hmtx_t(w)]_{ij} \right) \\
&\stackrel{\eqref{eq:hess-entries-grad}}{=}
-\sum_{\tau \in [t]^+} \frac{\lambda_{\tau}}{(x_\tau^\top w)^3} \Bigg( \sum_{i,j \in [d-1]} \big[ A^\top x_\tau\big]_{i} \left[ \Hmtx_t(w)^{-1} \right]_{ij} \big[ A^\top x_\tau \big]_{j} \Bigg)\; x_\tau \\
&\;= -\sum_{\tau \in [t]^+} 
\frac{\lambda_{\tau} \| A^\top x_\tau\|_{\Hmtx_t(w)^{-1}}^2}{(x_{\tau}^\top w)^3} x_\tau 
= \sum_{\tau \in [t]^+} \pi_{\tau}(w) \, \nabla \ell_\tau(w).
\end{aligned}
\]

\proofstep{2}. 
Our next goal is to derive the following explicit formula:
\begin{equation}
\label{eq:V-hess}
\nabla^2 V_t(w) =
3\Qmtx_t(w) 
- 2 \Smtx_t(w)
\quad \text{where} \quad
\Smtx_t(w) := 
\sum_{\tau, \nu \in [t]^+}  \pi_{\tau, \nu}(w)^2 \, \nabla \ell_\tau(w) \, \nabla \ell_{\nu}(w)^\top.
\end{equation}
But before, let us show how~\eqref{eq:V-hess-bound} follows from it.
Indeed, recall that~$\pi_{\tau}(w) = \sum_{\nu \in [t]^+} \pi_{\tau,\nu}(w)^2$, cf.~Lemma~\ref{lem:alternate_form_leverage_score}.
Whence from~\eqref{eq:V-hess} and the definition of~$\Qmtx_t(w)$, cf.~\eqref{eq:V-hess-bound}, we obtain
\[
\begin{aligned}
\nabla^2 V_t(w) - \Qmtx_t(w)
&= 2\Qmtx_t(w) - 2\Smtx_t(w) \\
&= \sum_{\tau,\nu \in [t]^+} \pi_{\tau,\nu}(w)^2 \; \left( 2\nabla^2 \ell_\tau(w)  - 2\nabla \ell_\tau(w)  \, \nabla \ell_{\nu}(w)^\top \right) \\
%
%
&=  \sum_{\tau, \nu \in [t]^+}  \pi_{\tau,\nu}(w)^2 \, \Big( \nabla \ell_\tau(w) - \nabla \ell_\nu(w) \Big) \, \Big( \nabla \ell_\tau(w) - \nabla \ell_\nu(w) \Big)^\top
\succeq 0,
\end{aligned}		
\]
the lower bound in~\eqref{eq:V-hess-bound}. 
We obtain the upper bound~$\nabla^2 V_t(w) \preceq 5 \Qmtx_t(w)$ in a similar fashion:
\[
\begin{aligned}
5\Qmtx_t(w) - \nabla^2 V_t(w) 
&= 2\Qmtx_t(w) + 2\Smtx_t(w) \\
&=  \sum_{\tau, \nu \in [t]^+}  \pi_{\tau,\nu}(w)^2 \; \Big( \nabla \ell_\tau(w) + \nabla \ell_\nu(w) \Big) \, \Big( \nabla \ell_\tau(w) + \nabla \ell_\nu(w) \Big)^\top 
\succeq 0.
\end{aligned}
\]

\proofstep{3}.
It remains to verify~\eqref{eq:V-hess}. 
First, direct differentiation of~\eqref{eq:V-grad} shows that
\begin{align}
\nabla^2 V_t(w) = \Qmtx_t(w) + \Rmtx_t(w) 
\quad \text{where} \quad
\Rmtx_t(w)  := \sum_{\tau \in [t]^+} \nabla \pi_{\tau}(w) \nabla \ell_{\tau}(w)^\top, 
\label{eq:product-rule}
\end{align}
granted that~$\Rmtx_t(w)$ is symmetric. 
Thus, to get~\eqref{eq:V-hess} we must show that~$\Rmtx_t(w) = 2\Qmtx_t(w) - 2\Smtx_t(w)$.
To this end, 
observe that, by~\eqref{def:abridged-hess-and-grad-apx},~\eqref{eq:leverage-score-explicit-apx}, and the product rule,
\begin{align}
\nabla \pi_{\tau}(w)
&= \frac{\lambda_{\tau}}{(x_\tau^\top w)^{2}} \nabla \left( \| A^\top x_\tau \|^2_{\Hmtx_t(w)^{-1}} \right) 
	\; - \,\frac{2\lambda_{\tau} \| A^\top x_\tau \|^2_{\Hmtx_t(w)^{-1}}}{(x_\tau^\top w)^{3}}  \; x_{\tau} \notag\\
&= \frac{\lambda_{\tau}}{(x_\tau^\top w)^{2}} \nabla \left( \| A^\top x_\tau \|^2_{\Hmtx_t(w)^{-1}} \right) 
	\; + \; 2\pi_{\tau}(w) \nabla \ell_\tau(w).
\label{eq:product-rule-post}
\end{align}
Furthermore, by~Lemma~\ref{lem:inv-hessian-derivative} and the composition rule,
\[
\begin{aligned}
\nabla \left( \|A^\top x_\tau \|^2_{\Hmtx_t(w)^{-1}} \right) 
&\;= -\sum_{i,j \in [d-1]} \left[ \Hmtx_t(w)^{-1}  A^\top x_\tau^{\vphantom\top} x_\tau^\top A \Hmtx_t(w)^{-1} \right]_{ij} \nabla \big([\Hmtx_t(w)]_{ij}\big) \\
&\stackrel{\eqref{eq:hess-entries-grad}}{=} \hspace{-0.1cm} \sum_{\nu \in [t]^+} \hspace{-0.1cm}
	\frac{2\lambda_\nu}{(x_{\nu}^\top w)^3}
	\Bigg( 
		\sum_{i,j \in [d-1]} \big[ A^\top x_\nu \big]_{i} \left[ \Hmtx_t(w)^{-1} A^\top x_\tau^{\vphantom\top} x_\tau^\top A \Hmtx_t(w)^{-1}\right]_{ij} \big[ A^\top x_\nu \big]_{j} \hspace{-0.1cm}
	\Bigg) x_{\nu}\\
&\;= 
    \sum_{\nu \in [t]^+} \frac{2 \lambda_\nu}{(x_{\nu}^\top w)^{3}} \lang A^\top x_\nu, A^\top x_\tau \rang_{\Hmtx_t(w)^{-1}}^2 x_\nu \\
&\;= 
	-\sum_{\nu \in [t]^+} \frac{2\lambda_\nu}{(x_{\nu}^\top w)^{2}} \lang A^\top x_\nu, A^\top x_\tau \rang_{\Hmtx_t(w)^{-1}}^2 \nabla \ell_\nu(w),
\end{aligned}
\]
whence
\begin{align}
\frac{\lambda_{\tau}}{(x_\tau^\top w)^{2}} \nabla \left( \| A^\top x_\tau \|^2_{\Hmtx_t(w)^{-1}} \right) 
&= -2\sum_{\nu \in [t]^+} \lambda_\tau\lambda_\nu \lang A^\top \nabla \ell_\tau(w), A^\top \nabla \ell_\nu(w) \rang_{\Hmtx_t(w)^{-1}}^2 \nabla \ell_\nu(w) \notag\\
&= -2\sum_{\nu \in [t]^+} {\pi_{\tau,\nu}(w)}^2 \nabla \ell_\nu(w).
\label{eq:newton-decrement-derivative}
\end{align}
Combining~\eqref{eq:product-rule}--\eqref{eq:newton-decrement-derivative}, we verify the identity~$\Rmtx_t(w) = 2\Qmtx_t(w) - 2\Smtx_t(w)$, and hence also~\eqref{eq:V-hess}.
%
%
\end{proof}

%


\subsection{Self-concordance of the hybrid barrier}

The following result can be understood as an extension of Corollary~\ref{cor:SC-of-L} from~$L_t(\cdot)$ to~$P_t$ with~$\mu > 0$.


\begin{proposition}
\label{prop:SC-of-P}
Assume~$\lambda \ge 1$ and~$\mu > 0$, then the restriction~$P_t^\res: \R^{d-1} \to \wb\R$ of~$P_t$ to~$\Aff_d$, i.e.
\begin{equation*}
P_t^{\res}(v) = P_t(Av + e_d),
\end{equation*}
cf.~\eqref{def:VR-potential} and~\eqref{def:coordinate-matrix}--\eqref{def:affine-reparametrization}, is~$21$-self-concordant according to~Definition~\ref{def:SC-function} with domain~$\Int(\bDelta_{d-1})$, cf.~\eqref{def:solid-simplex}.
\end{proposition}

\begin{proof}
The requirements on the domain are verified in the same way as in the proof of~Corollary~\ref{cor:SC-of-L}. 
Thus, it remains to show that when~$\lambda,\mu$ are as in the premise, for any~$u \in \R^d$ and~$w \in \Delta_{d}$ one has
\begin{equation}
\label{eq:SC-bound-for-P}
|\nabla^3 P_t(w)[u,u,u]| \le 42 \left( \nabla^2 P_t(w)[u,u] \right)^{3/2}.
\end{equation}
Indeed, this implies that~$\left|\nabla^3 P_t^\res(v)[u^\res,u^\res,u^\res]\right| \le 42 (\nabla^2 P_t^\res(v)[u^\res, u^\res])^{3/2}$
for~$v \in \bDelta_{d-1}, u^\res \in \R^{d-1}$ by affine invariance, i.e.~since for~$w = Av + e_d$ and~$u = A u^\res$ one has~$\nabla^2 P_t^\res(v)[u^\res,u^\res] = \nabla^2 P_t(w)[u,u]$ and~$\nabla^3 P_t^\res(v)[u^\res,u^\res,u^\res] = \nabla^3 P_t(w)[u,u,u]$.
In turn,~\eqref{eq:SC-bound-for-P} can be derived from the following result.
\begin{lemma}[{Adaptation of~\cite[Thm.~4.1]{anstreicher1997volumetric}}]
\label{lem:hybrid-hessian-stability}
Let~$w, \wb w \in \R^d_{++}$ be such that~$\delta_t(w,\wb w) \le 0.1$, where 
\[
\delta_t(w, \wb w) = \max_{\tau \in [t]^+} \left| \frac{x_\tau^\top (\wb w - w)}{x_{\tau}^\top w}\right|.
\]
Then for~$\veps_t(w, \wb w) \hspace{-0.02cm} = \hspace{-0.02cm} \frac{42\delta_t(w, \scriptsize\wb w)}{[1-\delta_t(w, \scriptsize\wb w)]^{4}}$ one has that
\[
\begin{aligned}
	-\veps_t(w,\wb w) \nabla^2 V_t(w) 
\hspace{-0.04cm} \preceq \hspace{-0.04cm}
	\nabla^2 V_t(\wb w) - \nabla^2 V_t(w) 
\hspace{-0.04cm} \preceq \hspace{-0.04cm}
	 \veps_t(w,\wb w) \nabla^2 V_t(w), \\
	-\veps_t(w,\wb w) \nabla^2 P_t(w) 
\hspace{-0.04cm} \preceq \hspace{-0.04cm}
	\nabla^2 P_t(\wb w) - \nabla^2 P_t(w) 
\hspace{-0.04cm} \preceq \hspace{-0.04cm}
	 \veps_t(w,\wb w) \nabla^2 P_t(w).
\end{aligned}
\]

\end{lemma}
\noindent
Now, fix~$u \in \R^{d}$ and note that
$
\nabla^3 P_t(w)[u,u,u] 
= \lim_{\alpha \to 0} \frac{1}{\alpha} \left(\nabla^2 P_t(w + \alpha u)[u,u] - \nabla^2 P_t(w)[u,u] \right).
$
On the other hand, when~$|\alpha|$ is small enough we can apply Lemma~\ref{lem:hybrid-hessian-stability} to~$\wb w = w+\alpha u$. 
Indeed:
\[
\delta_t(w,w+\alpha u) = |\alpha| \, M_t(u,w) \;\; \text{where} \;\; M_t(u,w) := \max_{\tau \in [t]^+} \left| \frac{x_\tau^\top u}{x_{\tau}^\top w}\right|,
\]
therefore~$\delta_t(w,w+\alpha u) \to 0$ when~$|\alpha| \to 0$ with fixed~$u,w$.
By Lemma~\ref{lem:hybrid-hessian-stability} we then conclude that
\[
\begin{aligned}
|\nabla^3 P_t(w)[u,u,u]|
&\le \lim_{\alpha \to 0} \frac{42 \, \delta_t(w,w+\alpha u)}{|\alpha| \, [1-\delta_t(w,w+\alpha u)]^4} \nabla^2 P_t(w)[u,u]
= 42 M_t(u,w) \, \nabla^2 P_t(w)[u,u].
\end{aligned}
\]
Finally, note that~$\nabla^2 P_t(w)[u,u] \ge \nabla^2 L_t(w)[u,u] \ge M_t(u,w)^2 \min\{\lambda,1\} = M_t(u,w)^2$ as~$\lambda \ge 1$. 
\end{proof}

\paragraph{Proof of Lemma~\ref{lem:hybrid-hessian-stability}.}
Our proof closely follows the one~\cite{anstreicher1997volumetric} while allowing for~$\lambda \ge 1$ and~$\mu > 0$.
In the remainder of this section, we let~$\delta := \delta_t(w,\wb w)$ and also introduce some abbreviated notation:
\begin{align*}
\eta_\tau := x_\tau^\top w, \quad
\wb \eta_\tau := x_\tau^\top \wb w, \quad
\pi_\tau := \pi_\tau(w), \quad
\wb\pi_\tau := \pi_\tau(\wb w), \quad
\bNabla_{\tau} := \bNabla_{\tau}(w), \quad
\wb\bNabla_{\tau} := \bNabla_{\tau}(\wb w), \\
\Hmtx := \Hmtx_t(w), \quad
\xb\Hmtx := \Hmtx_t(\wb w), \quad
\Qmtx := \Qmtx_t(w), \quad
\xb\Qmtx := \Qmtx_t(\wb w), \quad
\Smtx := \Smtx_t(w), \quad
\xb\Smtx := \Smtx_t(\wb w).
\end{align*}
\begin{lemma}
\label{lem:simple-approximations}
Assuming that~$\delta < 1$, for any~$\tau \in [t]^+$ the following bounds hold:
\begin{align}
\label{eq:eta-bound}
1-\delta 
&\le \frac{\wb \eta_\tau}{\eta_{\tau}} \le 
1 + \delta, 
\quad\quad\quad\quad\quad\;\;\,
\frac{1}{(1+\delta)^2} \Hmtx
	\preceq \xb\Hmtx \preceq
\frac{1}{(1-\delta)^2} \Hmtx, \\
\label{eq:pi-bound}
\left( \frac{1 - \delta}{1+\delta} \right)^2 
	&\le \frac{\wb \pi_\tau}{\pi_{\tau}} \le
\left( \frac{1 + \delta}{1-\delta} \right)^2, 
\quad\quad\quad\;\;
\frac{(1-\delta)^2}{(1+\delta)^4} \Qmtx
	\preceq \xb\Qmtx \preceq
\frac{(1+\delta)^2}{(1-\delta)^4} \Qmtx.
\end{align}
\end{lemma}
\begin{proof}
The bounds for~$\wb\eta$ are obvious from the definition of~$\delta$, and immediately imply those for~$\xb \Hmtx$:
\[
\xb\Hmtx
= \sum_{\tau \in [t]^+} \frac{\lambda_\tau}{\wb\eta_\tau^2} A^\top x_\tau x_\tau^\top A 
\preceq \frac{1}{(1-\delta)^2} \sum_{\tau \in [t]^+} \frac{\lambda_\tau}{\eta_\tau^2} A^\top x_\tau x_\tau^\top A 
= \frac{1}{(1-\delta)^2}\Hmtx,
\]
and similarly for the corresponding lower bound.
Furthermore,~\eqref{eq:eta-bound} implies
\[
\frac{\wb \pi_\tau}{\pi_\tau} 
= \frac{\lambda_\tau \|\wb\bNabla_\tau\|_{\scriptsize\xb\Hmtx^{-1}}^2}{\lambda_\tau \|\bNabla_\tau\|_{\Hmtx^{-1}}^2}
= \frac{\|\wb\bNabla_\tau\|_{\scriptsize\xb\Hmtx^{-1}}^2}{\|\bNabla_\tau\|_{\Hmtx^{-1}}^2}
\le (1+\delta)^2 \frac{\|\wb\bNabla_\tau\|_{\Hmtx^{-1}}^2}{\|\bNabla_\tau\|_{\Hmtx^{-1}}^2}
= (1+\delta)^2 \frac{\eta_{\tau}^2}{\wb \eta_{\tau}^2} 
\le \left(\frac{1+\delta}{1-\delta} \right)^2,
\]
and similarly for the corresponding lower bound. Finally, from the above results we conclude that
\[
\xb\Qmtx
= \sum_{\tau \in [t]^+} \frac{\wb\pi_\tau}{\wb\eta_\tau^2} x_\tau x_\tau^\top
\preceq \frac{(1+\delta)^2}{(1-\delta)^4}\sum_{\tau \in [t]^+} \frac{\pi_\tau}{\eta_\tau^2} x_\tau x_\tau^\top
= \frac{(1+\delta)^2}{(1-\delta)^4}\Qmtx,
\]
and similarly for the corresponding lower bound.
\end{proof}

\proofstep{1}.
Returning to the proof of Lemma~\ref{lem:hybrid-hessian-stability}, recall that
\[
\nabla^2 P_t(\wb w) - \nabla^2 P_t(w) = 
\nabla^2 L_t(\wb w) - \nabla^2 L_t(w) + \mu \left[ \nabla^2 V_t(\wb w) - \nabla^2 V_t(w) \right],
\]
and let us consider the two differences in the right-hand side separately. For the first difference we observe that~$\nabla^2 L_t(w)$ and~$\nabla^2 L_t(\wb w)$ satisfy the same relative bounds as~$\Hmtx$ and \,$\xb\Hmtx$\, in~\eqref{eq:eta-bound}, so
\[
\nabla^2 L_t(\wb w) - \nabla^2 L_t(w) 
\preceq \left(\frac{1}{(1-\delta)^2} - 1 \right) \nabla^2 L_t(w) 
\preceq \frac{2\delta}{(1-\delta)^4} \nabla^2 L_t(w)
\]
where we used that~$0 \le \delta \le 1$. Similarly, 
\[
\nabla^2 L_t(w)  - \nabla^2 L_t(\wb w)
\preceq \left(1 - \frac{1}{(1+\delta)^2} \right) \nabla^2 L_t(w) 
\preceq \frac{3\delta}{(1-\delta)^4} \nabla^2 L_t(w).
\]
Thus, to prove the lemma it suffices to show the first of the claimed inequalities, that is
\begin{equation}
\label{eq:volumetric-hessian-stability}
-\frac{42\delta}{(1-\delta)^4} \nabla^2 V_t(w) 
\preceq \nabla^2 V_t(\wb w) - \nabla^2 V_t(w) \preceq 
\frac{42\delta}{(1-\delta)^4} \nabla^2 V_t(w). 
\end{equation}

\proofstep{2}. 
To this end, we first observe that
\[
\nabla^2 V_t(\wb w) - \nabla^2 V_t(w) 
\stackrel{\eqref{eq:V-hess}}{=} 3(\,\xb \Qmtx - \Qmtx) + 2(\Smtx - \xb\Smtx),
\]
and Lemmas~\ref{lem:simple-approximations}--\ref{lem:volumetric-grad-and-hess} already allow us to control the difference \,$\xb \Qmtx - \Qmtx$, namely:
\begin{equation}
\label{eq:Qbar-Q-bound}
\begin{aligned}
\xb \Qmtx - \Qmtx 
&\stackrel{\eqref{eq:pi-bound}}{\preceq} \left( \frac{(1+\delta)^2}{(1-\delta)^4} - 1 \right) \Qmtx 
= \frac{(3\delta - \delta^2)(2 - \delta + \delta^2)}{(1-\delta)^4} \Qmtx 
\stackrel{(a)}{\preceq} \frac{6\delta}{(1-\delta)^4} \Qmtx,\\
\Qmtx - \xb \Qmtx 
&\stackrel{\eqref{eq:pi-bound}}{\preceq} \left(1 -  \frac{(1-\delta)^2}{(1+\delta)^4} \right) \Qmtx 
\stackrel{(b)}{\preceq} \left( \frac{(1+\delta)^2}{(1-\delta)^4} - 1 \right) \Qmtx 
\preceq \frac{6\delta}{(1-\delta)^4} \Qmtx.
\end{aligned}
\end{equation}
Here in~$(a)$ we used that~$0 \le \delta < 1$, and~$(b)$ holds since the function~$\frac{(1+\delta)^2}{(1-\delta)^4} + \frac{(1-\delta)^2}{(1+\delta)^4}$ on~$[0,1)$ is minimized at~$\delta = 0$.
Let us now study the difference\, $\xb \Smtx - \Smtx$.\\

\proofstep{3}.
For brevity, we define~$\bPi := \bPi_t(w)$ and~$\wb\bPi := \bPi_t(\wb w)$. We also define some auxiliary matrices: 
\begin{itemize}
\item matrix~$X \in \R^{d \times (t + d)}$ with~$x_{\tau}$ in its $\tau^{\text{th}}$ column;
\item diagonal matrices~$\bLambda, \Dmtx, \xb \Dmtx,\Zmtx \in \Sym^{t+d}_{++}$ with~$\lambda_{\tau}^{\vphantom 0}, \eta_\tau^{\vphantom 0}, \wb\eta_\tau,\frac{\eta_\tau^{\vphantom 0}}{\scriptsize\wb\eta_\tau}$ at~$\tau^{\text{th}}$ diagonal position, respectively.
\end{itemize}
As earlier for~$\bPi_t(w)$, we use the convention that indexing is over~$[t]^+$ in each dimension where the size of a matrix is~$t + d$. 
Finally, let~$\boldsymbol{A} \circ \boldsymbol{B}$ be the Hadamard (i.e. entrywise) product of two matrices with the same dimensions, and let~$\Mmtx^{(2)} = \Mmtx \circ \Mmtx$.  
In terms of this matrix notation, we have that
\begin{align}
\label{eq:Q-formula}
\Qmtx &= X \Dmtx^{-1} \;\, \bPsi \;\, \Dmtx^{-1} X^\top,\\
\label{eq:S-formula}
\Smtx &= X \Dmtx^{-1} \bPi^{(2)} \hspace{-0.05cm} \Dmtx^{-1} X^\top,
\end{align}
where~$\bPsi = \Diag(\bPi)$ is the diagonal matrix with~$\pi_{\tau}$ at~$\tau^{\text{th}}$ diagonal position. 
On the other hand,
\begin{align}
\xb\Smtx 
	&= X \; \xb\Dmtx^{-1}  \wb\bPi^{(2)}  \; \xb\Dmtx^{-1} X^\top \notag\\
	&= X \Dmtx^{-1}  \Zmtx \; \wb\bPi^{(2)} \hspace{-0.05cm} \Zmtx \Dmtx^{-1} X^\top.
\label{eq:Sbar-aux}
\end{align}
where in the last step we used that~$\Zmtx = \xb \Dmtx^{-1} \hspace{-0.06cm}\Dmtx$. 
In a similar way, 
$
\bPi =  \Dmtx^{-1} \bLambda^{1/2} X^\top \Hmtx^{-1} X \bLambda^{1/2} \Dmtx^{-1}
$
and \,$\wb\bPi 
	= \, \xb \Dmtx^{-1} \hspace{-0.06cm}\bLambda^{1/2} X^\top \xb\Hmtx^{-1} X \; \bLambda^{1/2} \xb\Dmtx^{-1}  \hspace{-0.1cm}
	= \Zmtx  \Dmtx^{-1} \bLambda^{1/2} X^\top \xb\Hmtx^{-1} X \bLambda^{1/2} \Dmtx^{-1} \Zmtx,$
whence by~\eqref{eq:eta-bound} we get
\begin{equation}
\label{eq:Pibar-vs-Pi}
(1-\delta)^2 \Zmtx \bPi \Zmtx \preceq \wb \bPi \preceq (1+\delta)^2 \Zmtx \bPi \Zmtx. 
\end{equation}
Since~$\Zmtx$ is diagonal,~$(\Zmtx \bPi \Zmtx)^{(2)} = \Zmtx^2 \bPi^{(2)} \hspace{-0.06cm} \Zmtx^2$.  
\odima{Moreover: by the Schur product theorem~\cite{bapat1997nonnegative}, for two symmetric matrices~$\boldsymbol{A}, \boldsymbol{B}$ such that~$\boldsymbol{A} \succeq  \boldsymbol{B} \succeq 0$,
one has
$
\boldsymbol{A}^{(2)}  - \boldsymbol{B}^{(2)} = \boldsymbol{A}  \circ (\boldsymbol{A}  - \boldsymbol{B}) + (\boldsymbol{A}  - \boldsymbol{B}) \circ \boldsymbol{B} \succeq 0.
$}
In combination with~\eqref{eq:Pibar-vs-Pi}, these two facts imply
\[
(1-\delta)^4 \Zmtx^2 \bPi^{(2)} \Zmtx^2 \preceq \wb \bPi^{(2)} \preceq (1+\delta)^4 \Zmtx^2 \bPi^{(2)} \hspace{-0.05cm} \Zmtx^2.
\]
Whence by~\eqref{eq:Sbar-aux} we conclude that
\begin{equation}
\label{eq:Sbar-estimate}
(1-\delta)^4 X \Dmtx^{-1}  \Zmtx^3 \bPi^{(2)} \hspace{-0.05cm} \Zmtx^3 \Dmtx^{-1} X^\top \hspace{-0.1cm}
\preceq 
\,\; \xb\Smtx \;
\preceq 
(1+\delta)^4 X \Dmtx^{-1}  \Zmtx^3 \bPi^{(2)} \hspace{-0.05cm} \Zmtx^3 \Dmtx^{-1} X^\top.
\end{equation}
We are now going to use identities~\eqref{eq:Q-formula}--\eqref{eq:S-formula} and inequality~\eqref{eq:Sbar-estimate} to bound~$\xb\Smtx-\Smtx$ in terms of~$\Qmtx$.\\

\proofstep{4}. 
Defining~$\Jmtx := \bPsi^{-1/2} \bPi^{(2)} \bPsi^{-1/2}$ and~$U := X \Dmtx^{-1} \bPsi^{1/2}$ we rewrite~\eqref{eq:Q-formula},~\eqref{eq:S-formula},~\eqref{eq:Sbar-estimate}:
\begin{equation}
\label{Q-and-S-explicitly}
\begin{aligned}
\Qmtx 
	= U U^\top, \quad
\Smtx 
	&= U \Jmtx U^\top, \\
(1-\delta)^4 U \Zmtx^3 \Jmtx \Zmtx^3 U^\top
\preceq \; \xb \Smtx \,
	&\preceq (1+\delta)^4 U \Zmtx^3 \Jmtx \Zmtx^3 U^\top.
\end{aligned}
\end{equation}
We are now going to bound~$\Mmtx_1 := (1+\delta)^4 \Zmtx^3 \Jmtx \Zmtx^3 - \Jmtx$ and~$\Mmtx_2 := \Jmtx - (1-\delta)^4 \Zmtx^3 \Jmtx \Zmtx^3$ from above.
Since~$\bPsi,\Zmtx$ commute as diagonal matrices, and by similarity,~$\Mmtx_1$ has the same eigenvalues as
\[
\bPsi^{-1/2} \Mmtx_1 \bPsi^{1/2}
= 
\bPsi^{-1} \left( (1+\delta)^4 \Zmtx^3 \bPi^{(2)} \Zmtx^3 - \bPi^{(2)} \right).
\]
The entries of the latter matrix are given by
\begin{equation}
\label{eq:Gershgorin-entries}
[\bPsi^{-1/2} \Mmtx_1 \bPsi^{1/2}]_{\tau,\nu} = \frac{\pi_{\tau,\nu}^2}{\pi_{\tau}} \left( (1+\delta)^4 z_{\tau}^3 z_{\nu}^3 - 1   \right)
\end{equation}
where~$z_\tau = \frac{\eta_\tau^{\vphantom 0}}{\scriptsize\wb\eta_\tau}$ is the~$\tau^{\text{th}}$ diagonal entry of~$\Zmtx$.
By the Gershgorin circle theorem~\cite{golub2013matrix}, each eigenvalue of the matrix~$\bPsi^{-1/2} \Mmtx_1 \bPsi^{1/2}$---and hence also of~$\Mmtx_1$---belongs to at least one of the segments 
\[
\left[ [\bPsi^{-1/2} \Mmtx_1 \bPsi^{1/2}]_{\tau,\tau} - r_\tau, \;\; [\bPsi^{-1/2} \Mmtx_1 \bPsi^{1/2}]_{\tau,\tau} + r_\tau \right] 
\;\; \text{with} \;\; 
r_{\tau} := \sum_{\nu \ne \tau} \big| [\bPsi^{-1/2} \Mmtx_1 \bPsi^{1/2}]_{\tau,\nu} \big|, \;\; \tau \in [t]^+.
\]
\vspace{-0.2cm}
Whence by~\eqref{eq:Gershgorin-entries},~\eqref{eq:eta-bound}:\vspace{-0.1cm}
\[
\begin{aligned}
\Mmtx_1 
\preceq
\max_{\tau \in [t]^+} \left\{ \frac{1}{\pi_\tau} \sum_{\nu \in [t]^+} \pi_{\tau,\nu}^2 \right\} 
\left( \frac{(1+\delta)^4}{(1-\delta)^6}  - 1 \right) \Imtx  
\stackrel{\eqref{eq:leverage-score-identity}}{\preceq}
\left( \frac{(1+\delta)^4}{(1-\delta)^6}  - 1 \right) \Imtx. 
\end{aligned}
\]
Proceeding with~$\Mmtx_2$ in a similar fashion as we did with~$\Mmtx_1$, we get
\[
\Mmtx_2 
\preceq \left( 1 - \frac{(1-\delta)^4}{(1+\delta)^6}\right) \Imtx  
\preceq \left( \frac{(1+\delta)^4}{(1-\delta)^6}  - 1 \right) \Imtx.
\]
(The second step uses that the function~$\frac{(1+\delta)^4}{(1-\delta)^6} + \frac{(1-\delta)^4}{(1+\delta)^6}$ increases on~$[0,1)$.)
Finally, observe that
\[
\frac{(1+\delta)^4}{(1-\delta)^6}  - 1 
=
\frac{4\delta(1+\delta)^2}{(1-\delta)^6} 
+ \frac{(1+\delta)^2}{(1-\delta)^4}  - 1
\le
\frac{\delta}{(1-\delta)^4} 
\bigg( 4 \bigg(\frac{1+\delta}{1-\delta} \bigg)^2 + 6 \bigg)
\le
\frac{12\delta}{(1-\delta)^4} 
\]
where the first estimate repeats~$(\ref{eq:Qbar-Q-bound}.a)$, and the second uses that~$(\frac{1+\delta}{1-\delta})^2 < 1.5$ for~$0 \le \delta \le 0.1$. 
Plugging this estimate into~\eqref{Q-and-S-explicitly} results in
\[
-\frac{12\delta}{(1-\delta)^4} \Qmtx \preceq \, \Smtx - \xb \Smtx \preceq \frac{12\delta}{(1-\delta)^4} \Qmtx. 
\]
Joining this with~\eqref{eq:volumetric-hessian-stability},~\eqref{eq:Qbar-Q-bound}, and the bound~$\Qmtx \preceq \nabla^2 V_t(w)$, cf.~\eqref{eq:V-hess-bound}, we 
get~\eqref{eq:volumetric-hessian-stability}, as required.
\qed

\begin{remark}
\label{rem:SC-constant}
The constant factor~$42$ in Lemma~\ref{lem:hybrid-hessian-stability} can be improved to~$30$, exactly recovering the result in~\cite{anstreicher1997volumetric}, through a somewhat more delicate argument following~\cite{anstreicher1997volumetric}, where one separately bounds~$\Qmtx_t(\wb w) - \Qmtx_t(w)$ and~$\Rmtx_t(\wb w) - \Rmtx_t(w)$ for~$\Rmtx(w) = \nabla^2 V_t(w) - \Qmtx_t(w) = 2\Qmtx_t(w) - 2\Smtx_t(w)$. Here we avoided this, as it would burden the proof with linear algebra and obscure the high-level picture.

\end{remark}
\section{Stability lemma for Algorithm~\ref{alg:efficient-ftrl-vr}}
\label{apx:stability-lemma}

We are now about to prove Lemma~\ref{lem:approx-stability} about the stability of approximation of the updates according to~\eqref{def:VB-FTRL} by those corresponding to~\eqref{def:VB-FTRL-Newton}. 
For convenience, let us repeat it here.

\begin{lemma}
\label{lem:approx-stability-apx}
Let~$\lambda, \mu, S$ be as in the premise of Theorem~\ref{th:volumetric-newton}. Then for all~$t \in \{0\} \cup [T]$, one has
\begin{equation}
\label{eq:Newton-stability-invariant-apx}
\| \wt w_t - w_t \|_{\nabla^2 L_{t}(w_t)} \le \min \left\{ (T+d+1)^{-2}, \; 10^{-4} (1 + 3\mu)^{-1/2}\right\}.
\end{equation}
\end{lemma}


\begin{proof}
Let~$v_t := A^+ (w_t - e_d)$, ~$\wt v_t := A^+ (\wt w_t - e_d)$, and~$v_t^{(s)} := A^+ (w_t^{(s)} - e_d)$ for all~$t \in [T]$ and~$s \in [S]$.
We proceed by induction over~$t \in [T]$. 
The base is obvious:~$\wt w_0 = w_0$. 
%
For the induction step, assume that~\eqref{eq:Newton-stability-invariant} holds for some~$0 \le t \le T-1$.
By~Proposition~\ref{prop:SC-of-P} and Lemma~\ref{lem:SC-sum},~$\sqrt{21} P_{t}^\res(\cdot)$ is a 1-self-concordant function minimized at~$v_{t+1}$. 
On the other hand, since function~$\wbb P_t^\res(\cdot)$ defined in~\eqref{eq:potential-affine-lower-bound} satisfies~$\nabla \wbb P_t^\res(v_t) = \nabla P_t^\res(v_t)$ and~$\nabla^2 \wbb P_t^\res(v_t) \preceq \nabla^2 P_t^\res(v_t)$, we estimate 
$
\Dec_t := \left\|\nabla P_t^\res(v_t)\right\|_{{\nabla^2 P_t^\res(v_t)}^{-1}}
$
in the same way in the proof of Theorem~\ref{th:volumetric}, except for using a stronger condition on~$\lambda,\mu$ in~\eqref{eq:volumetric-newton-conditions}:
\begin{equation}
\label{eq:dec-bound-ugly}
\sqrt{21} \Dec_t^2 \le \sqrt{21} \uDec_t^2
\stackrel{\eqref{eq:decrement-score-bound}}{\le} \frac{\sqrt{21}}{\lambda} \left(1 + \frac{2\mu}{\lambda} \right)^2 
\stackrel{\eqref{eq:volumetric-newton-conditions}}{\le} \frac{\sqrt{21}}{556} < \frac{1}{121}. 
\end{equation}
Since~$\sqrt{21} \Dec_t^2$ is the squared Newton decrement of~$\sqrt{21} P_{t}^\res$ at~$v_{t}$, by Lemmas~\ref{lem:bound-minimum-SC}--\ref{lem:SC-sandwich} we get
\[
\omega \left( 21^{1/4}  \| w_t - w_{t+1} \|_{\nabla^2 P_t (w_{t+1})} \right)
\stackrel{\eqref{def:affine-reparametrization-grad-and-hess}}{=} 
\omega \left( 21^{1/4}  \| v_t - v_{t+1} \|_{\nabla^2 P_t^\res (v_{t+1})} \right)
\le \psi \left(21^{1/4} \Dec_t \right)
\]
in terms of~$\omega, \psi$ defined in~\eqref{def:nesterov-functions}. 
By monotonicity of these two functions, using~\eqref{eq:dec-bound-ugly} we arrive at
\[
21^{1/4}  \| w_t - w_{t+1} \|_{\nabla^2 P_t (w_{t+1})} \le \omega^{-1} \left(  \psi \left( \tfrac{1}{11} \right) \right) < 0.097,
\]
and then by Lemma~\ref{lem:SC-hessian-bounds} (very concervatively):
\[
\nabla^2 P_t(w_{t+1}) \preceq 4 \nabla^2 P_t(w_{t}).
\]
Now, observe that~$\nabla^2 P_t(w) \preceq \nabla^2 L_t(w) + 3\mu \Qmtx_t(w) \preceq (1+3\mu) \nabla^2 L_t(w)$, the last step being due to~$\lambda \ge 1$ and~$\pi_{\tau}(w_\tau) \le 1$ by~Lemma~\ref{lem:bound_leverage_score}. Combining this with the previous two inequalities gives
\[
\begin{aligned}
21^{1/4} \| \wt w_t - w_{t+1} \|_{\nabla^2 P_t(w_{t+1})} 
&\le 21^{1/4} \left( \| w_t - w_{t+1} \|_{\nabla^2 P_t(w_{t+1})} + 2 \| \wt w_t - w_t \|_{\nabla^2 P_t(w_{t})} \right) \\
&< 0.097 + 4.3 \sqrt{1+3\mu} \, \| \wt w_t - w_t \|_{\nabla^2 L_t(w_{t})} 
\stackrel{\eqref{eq:Newton-stability-invariant}}{<} 0.098.
\end{aligned}
\]
(In the last step we used the induction hypothesis.)
Now, let~$\wt \Dec_t := \left\|\nabla P_t^\res(\wt v_t)\right\|_{{\nabla^2 P_t^\res(\wt v_t)}^{-1}}$, so that~$21^{1/4} \wt\Dec_t$ is the Newton decrement of~$\sqrt{21} P_t^\res$ at~$\wt v_t$.
By Lemmas~\ref{lem:bound-minimum-SC}--\ref{lem:SC-sandwich} and~\eqref{def:affine-reparametrization-grad-and-hess} we get
\begin{equation}
\label{eq:small-dec-newton}
21^{1/4} \wt\Dec_t
\le 
\omega^{-1} \left( \psi \left( 21^{1/4}  \| \wt w_t - w_{t+1} \|_{\nabla^2 P_t (w_{t+1})} \right) \right)
\le \omega^{-1}(\psi(0.098)) < \frac{1}{9}. 
\end{equation}
We are now in the situation of Corollary~\ref{cor:quasi-Newton-convergence} with~$f \equiv \sqrt{21} P_t^\res$,~$z^\star = v_{t+1}$,~$z_0 = \wt v_t$,~$z_S = \wt v_{t+1}$, and~$\sfc = \frac{1}{3}$, cf.~\eqref{eq:hess-approx-validity}; 
in particular,~$\| \nabla f(z_0) \|_{\nabla^2 f(z_0)^{-1}} < \frac{\sfc}{3}$ holds by~\eqref{eq:small-dec-newton}. 
Whence by Corollary~\ref{cor:quasi-Newton-convergence}:
\[
21^{1/4}  \| \wt w_{t+1} - w_{t+1} \|_{\nabla^2 P_t (w_{t+1})}
\stackrel{\eqref{eq:volumetric-newton-conditions}}{\le} 
	\min \left\{ (T+d+1)^{-2}, \; 10^{-4}(1+3\mu)^{-1/2} \right\}.
\] 
It remains to combine this result with the observation that, due to Lemma~\ref{lem:rank-one-comparison}, for any~$w \in \Delta_d$ 
\[
\nabla^2 L_{t+1}(w) \preceq \nabla^2 L_{t}(w) + \nabla^2 R(w) \preceq \left(1+\tfrac{1}{\lambda} \right) L_{t}(w) \preceq \left(1 + \tfrac{1}{\lambda} \right) \nabla^2 P_t(w),
\]
and use that~$(1+\frac{1}{\lambda}) \stackrel{\eqref{eq:volumetric-newton-conditions}}{\le} (1+\frac{1}{2e}) < 21^{-1/4}$. 
\end{proof}
\odima{
\section{Discussion: alleged~$\wt O(d^2)$ bottleneck for ONS-derived algorithms}
\label{apx:adabarron}

Here, we demonstrate {why}~$\wt O(d^2)$ seems to be the fundamental limit for the regret of ``ONS-derived'' algorithms (e.g.~\cite{luo2018efficient,zimmert2022pushing,mhammedi2022damped}) in online portfolio selection.\footnote{\odima{The work~\cite{gatmiry2024projection}, most recent in this avenue, is no exception. They attain~$\wt O(dL^2 + \max(L,1)(d+L\sqrt{d}))$ regret for~$1$-exponentially concave and~$L$-Lipschitz losses on a convex domain~$W$ contained in the unit~$\ell_2$-ball~\cite[Lem.~6]{gatmiry2024projection}. 
When applied to portfolios this gives a regret bound of order~$d^2$, as in this case~$L = \sup_{x \in \R^d, w \in \Delta_d} {\|x\|_2}{|\langle x, w \rangle|^{-1}} = \sqrt{d}$.}}
We mainly focus on Ada-BARRONS~\cite{luo2018efficient}, but this disussion seems to be relevant to other algorithms combining ONS with log-barrier penalties.

%

Following Online Newton Step (ONS) itself~\cite{hazan2007logarithmic}, the algorithms derived from it approximate the Hessian of~$\sum_{\tau \in [t]} \ell_{\tau}(w)$ at~$w = w_t$ with a constant matrix
$
\Amtx_t := \sum_{\tau \in [t]} \hat\nabla_\tau \hat\nabla_\tau{}^\top
$
where~$\hat\nabla_\tau := \nabla \ell_\tau(w_\tau)$. 
In the case of ONS, the update at round~$t$ is
\begin{equation}
\label{eq:ONS-update}
w_{t+1} := w_{t} - \frac{1}{\beta} (\Amtx_{t} + \varepsilon \Imtx)^{-1} \hat \nabla_{t} 
\end{equation}
where~$\beta$ is the inverse stepsize parameter ($\varepsilon > 0$ is nonessential, entering the regret under the logarithm).
The analysis of~\cite{hazan2007logarithmic} leads to the regret~$O(\tfrac{1}{\beta}d\log(T))$ against any CRP~$w \in \Delta_d$ for which
\begin{equation}
\label{eq:ONS-step-condition}
\beta \gsim \alpha_t(w) := \min_{\tau \in [t]} \frac{1}{|\langle \hat\nabla_\tau, w - w_\tau \rangle|} = \min_{\tau \in [t]} \frac{\langle x_\tau, w_\tau \rangle}{|\langle x_\tau, w_\tau - w \rangle|},\;\; \text{for}\;\; t \in [T].
\end{equation}
Then,~\cite{hazan2007logarithmic} gets an~$O(G_\infty d \log(T))$ regret under the gradient Lipschitzness assumption~$\hat\nabla_t \le G_{\infty}$, simply by bounding~$\inf_{w \in \Delta_d} \alpha_T(w) \gsim 1/G_{\infty}$. 
To obtain an affine-invariant bound,~\cite{luo2018efficient} augments~\eqref{eq:ONS-update} with a (time-variable) weighted log-barrier regularizer
\begin{equation}
\label{eq:log-barrier}
\lambda R_t(w) := \sum_{j \in [d]} \lambda_{t}[j] \log \left(\frac{1}{w[j]}\right) 
\;\;\text{with}\;\;\lambda_{t}[j] := \lambda \left( d \min_{\tau \in [t]}  w_\tau[j] \right)^{{1}/{\log(T)}}.
\end{equation}
Specifically,~\cite[Theorem~1]{luo2018efficient} bounds the regret against any comparator~$w \in \Delta$ satisfying~\eqref{eq:ONS-step-condition} with
\begin{equation}
\label{eq:ada-barrons-regret}
c_1\lambda d \log(T) + \frac{c_2}{\beta} d \log(T) - \frac{c_3}{\log(T)} \lambda \sum_{j \in [d]} \max_{t \in [T]} \frac{w[j]}{w_t[j]}
\end{equation}
where the first term has the same genesis as our~$\LogBias_T$, and the second term is the ONS regret. 
As for the final (negative) term, its full derivation is rather technical (cf.~the proof of Lemma 7 in~\cite[p.~7]{luo2018efficient}) and we will not reproduce it here; however, let us give some intuition behind it.
\begin{itemize}

\item[$a.$]
Observe that the regret must include~$\sum_{t \in [T]} \lambda R_{t-1}(w_t) - \lambda R_{t}(w_t),$ a counterpart of~$\sum_{t \in [T]}\Off_t$ in the \OurAlgo{} analysis:
indeed, such a term appears in the regret of FTRL with {\em any} time-variable regularizer.
For~\OurAlgo{}, $\Off_{t} = \mu V_{t-1}(w_t) - \mu V_{t}(w_t)$ itself is {\em negative}, as~$\det(\Hmtx_t(w))$ increases in~$t$.
However, this is not the case for the regularizer~\eqref{eq:log-barrier}, since~$\lambda_t[j]$ {\em decreases} in~$t$. 
Instead, negative terms are obtained by joining~$\sum_{t \in [T]} \lambda R_{t-1}(w_t) - \lambda R_{t}(w_t)$ with other terms.

\item[$b.$] Namely, exp-concavity (following~\cite{hazan2007logarithmic}) and the standard telescoping for mirror descent (the proof of Lemma 5 in~\cite[p.~13]{luo2018efficient}) 
allows to ``bring over'' the comparator and, instead of bounding~$\lambda R_{t-1}(w_t) - \lambda R_{t}(w_t)$, focus on bounding the difference of the Bregman divergences:
\[
D_{\lambda R_t}(w,w_t) - D_{\lambda R_{t-1}}(w,w_{t-1}) = \sum_{j \in [d]} \left(\lambda_{t}[j] - \lambda_{t-1}[j] \right) \left(\frac{w[j]}{w_{t}[j]} - \log\left(\frac{w[j]}{w_{t}[j]} \right) - 1\right).
\]
Note that~$R_t$ and~$R_{t-1}$ conveniently switched their positions here, and the above difference is {\em negative} (indeed,~$\lambda_t[j] > \lambda_{t-1}[j]$ and the last factor is clearly positive).
Observe also that the right-hand side explains the origin of the ratio~${w[j]}/{w_t[j]}$ in~\eqref{eq:ada-barrons-regret}: {it appears due to the gradient term~$\lambda \langle \nabla R_t(w_t) - \nabla R_{t-1}(w_t), w_t - w \rangle $, along with the fact that~$\log(z)' = {1}/{z}$.}

\item[$c.$]
Furthermore, the fact that this ratio is {\em maximized} over~$t \in [T]$ in~\eqref{eq:ada-barrons-regret} has a simple intuitive explanation: the algorithm increases the learning rate for stock~$j$ (i.e.~$\lambda_{t}[j] < \lambda_{t-1}[j]$) when the weight of the stock decreases ($w_{t}[j] < w_{t-1}[j]$), which can be interpreted as hedging against this decision by more thoroughly exploring a ``curious'' stock (cf.~\eqref{eq:log-barrier}).
This rule turns out to be flexible enough for the resulting strategy to successfully "track'' the adversary, and benefit from the {\em largest} ratio~$\max_{t \in [T]}\frac{w[j]}{w_t[j]}$ over the trajectory.
Inspection of the analysis in~\eqref{eq:log-barrier}~\cite[p.~7]{luo2018efficient} suggests that the negative term~in~\eqref{eq:ada-barrons-regret} cannot be improved with a barrier of the form~\eqref{eq:log-barrier}, even with some other scheme for tuning~$\lambda_{t}[j]$.

\item[$d.$] 
Finally, Ada-BARRONS~\cite{luo2018efficient} is a combination of the scheme described above---i.e.~ONS with a time-variable log-barrier type regularizer~\eqref{eq:ada-barrons-regret}---with an Armijo-type routine for tuning~$\beta$. This routine allows to ensure that the ONS stepsize condition in~\eqref{eq:ONS-step-condition} is satisfied for~$w = w_{t+1}$, then also that~$\beta \approx \alpha_T^{\vphantom*}(w_T^*)$ for the in-hindsight best CRP~$w_T^*$.

\end{itemize}
Combining \eqref{eq:ONS-step-condition} and~\eqref{eq:ada-barrons-regret} gives the regret bound for~Ada-BARRONS (see also~\cite[p.~4]{luo2018efficient})
\[
c_1\lambda d \log(T) + {c_2}d\log(T) \left(1 +  \max_{t \in [T]} \sum_{j \in [d]} \frac{w[j]}{w_t[j]} \right) - c_3\frac{\lambda}{\log(T)} \sum_{j \in [d]} \max_{t \in [T]} \frac{w[j]}{w_t[j]}.\vspace{-0.1cm}
\]
To cancel out the second term, {one must choose~$\lambda = cd \log^2(T)$}, and the resulting regret is~$\wt O(d^2)$.\\

\paragraph{Geometric view of the ONS stepsize condition~\eqref{eq:ONS-step-condition}.}
It appears that condition~\eqref{eq:ONS-step-condition} is {\em essential} for ONS to admit any nontrivial regret bound, not just technical.
Indeed, it is equivalent to
\[
\frac{1}{\beta}\max_{\tau \in [t]} \|\nabla \ell_\tau(w) - \hat\nabla_\tau \|_{(\hat \nabla_\tau \hat\nabla_\tau + \veps \Imtx)^{-1}} \lsim 1,
\]
i.e.~that~$w$ is in the~$O(\beta)$-radius Dikin ellipsoid of each loss~$\ell_{\tau}$, centered at~$w_\tau$. 
By self-concordance and in view of~\eqref{eq:ONS-update}, the above condition guarantees that~$\Amtx_t \approx \sum_{\tau \in [t]} \nabla^2 \ell_\tau(w_{t+1})$ up to a constant factor -- in other words, the quasi-Newton update is done according to the local geometry of~$\sum_{\tau \in [t]} \ell_{\tau}$.
In the absense of this property, {\em any} affine-invariant regret guarantee seems to be extremely unlikely.
}
%
%

\bibliographystyle{plain}
\bibliography{biblio} 

\begin{thebibliography}{10}

\bibitem{agarwal2006algorithms}
Amit Agarwal, Elad Hazan, Satyen Kale, and Robert~E. Schapire.
\newblock Algorithms for portfolio management based on the {N}ewton method.
\newblock In {\em Proceedings of the 23rd International Conference on Machine
  learning}, pages 9--16. PMLR, 2006.

\bibitem{anstreicher1996large}
Kurt~M. Anstreicher.
\newblock Large step volumetric potential reduction algorithms for linear
  programming.
\newblock {\em Annals of Operations Research}, 62(1):521--538, 1996.

\bibitem{anstreicher1997volumetric}
Kurt~M. Anstreicher.
\newblock Volumetric path following algorithms for linear programming.
\newblock {\em Mathematical Programming}, 76(1):245--263, 1997.

\bibitem{bapat1997nonnegative}
Ravindra~B. Bapat and T.~E.~S. Raghavan.
\newblock {\em Nonnegative matrices and applications}.
\newblock Cambridge University Press, 1997.

\bibitem{bellec2017optimal}
Pierre~C. Bellec.
\newblock Optimal exponential bounds for aggregation of density estimators.
\newblock {\em Bernoulli}, 23(1):219--248, 2017.

\bibitem{bellec2018optimal}
Pierre~C. Bellec.
\newblock Optimal bounds for aggregation of affine estimators.
\newblock {\em The Annals of {S}tatistics}, 46(1):30--59, 2018.

\bibitem{blum1999universal}
Avrim Blum and Adam Kalai.
\newblock Universal portfolios with and without transaction costs.
\newblock {\em Machine Learning}, 35(3):193--205, 1999.

\bibitem{boyd2004convex}
Stephen~P. Boyd and Lieven Vandenberghe.
\newblock {\em Convex optimization}.
\newblock Cambridge University Press, 2004.

\bibitem{brosse2017sampling}
Nicolas Brosse, Alain Durmus, {\'E}ric Moulines, and Marcelo Pereyra.
\newblock Sampling from a log-concave distribution with compact support with
  proximal {L}angevin {M}onte {C}arlo.
\newblock In {\em Proceedings of the 30th Conference On Learning Theory}, pages
  319--342. PMLR, 2017.

\bibitem{bubeck2018sampling}
S{\'e}bastien Bubeck, Ronen Eldan, and Joseph Lehec.
\newblock Sampling from a log-concave distribution with projected {L}angevin
  {M}onte {C}arlo.
\newblock {\em Discrete \& Computational Geometry}, 59(4):757--783, 2018.

\bibitem{cesa2006prediction}
Nicol{\`o} Cesa-Bianchi and G{\'a}bor Lugosi.
\newblock {\em Prediction, {L}earning, and {G}ames}.
\newblock Cambridge University Press, 2006.

\bibitem{cover1991universal}
Thomas~M. Cover.
\newblock Universal portfolios.
\newblock {\em Mathematical Finance}, 1(1):1--29, 1991.

\bibitem{cover1999elements}
Thomas~M. Cover.
\newblock {\em Elements of information theory}.
\newblock John Wiley \& Sons, 1999.

\bibitem{cover1996universal}
Thomas~M. Cover and Erik Ordentlich.
\newblock Universal portfolios with side information.
\newblock {\em IEEE Transactions on Information Theory}, 42(2):348--363, 1996.

\bibitem{dalalyan2012sharp}
Arnak~S. Dalalyan and Joseph Salmon.
\newblock Sharp oracle inequalities for aggregation of affine estimators.
\newblock {\em The {A}nnals of {S}tatistics}, 40(4):2327--2355, 2012.

\bibitem{dalalyan2008aggregation}
Arnak~S. Dalalyan and Alexandre~B. Tsybakov.
\newblock Aggregation by exponential weighting, sharp {PAC-B}ayesian bounds and
  sparsity.
\newblock {\em Machine Learning}, 72(1):39--61, 2008.

\bibitem{dwivedi2018log}
Raaz Dwivedi, Yuansi Chen, Martin~J. Wainwright, and Bin Yu.
\newblock Log-concave sampling: {M}etropolis-{H}astings algorithms are fast!
\newblock In {\em Proceedings of the 31st Conference On Learning Theory}, pages
  793--797. PMLR, 2018.

\bibitem{gatmiry2024projection}
Khashayar Gatmiry and Zakaria Mhammedi.
\newblock Projection-free online convex optimization via efficient {N}ewton
  iterations.
\newblock {\em Advances in Neural Information Processing Systems}, 36, 2024.

\bibitem{golub2013matrix}
Gene~H. Golub and Charles~F. Van~Loan.
\newblock {\em Matrix computations}.
\newblock JHU Press, 2013.

\bibitem{hazan2016introduction}
Elad Hazan.
\newblock Introduction to online convex optimization.
\newblock {\em Foundations and Trends{\textregistered} in Optimization},
  2(3-4):157--325, 2016.

\bibitem{hazan2007logarithmic}
Elad Hazan, Amit Agarwal, and Satyen Kale.
\newblock Logarithmic regret algorithms for online convex optimization.
\newblock {\em Machine Learning}, 69(2):169--192, 2007.

\bibitem{helmbold1998line}
David~P. Helmbold, Robert~E. Schapire, Yoram Singer, and Manfred~K. Warmuth.
\newblock On-line portfolio selection using multiplicative updates.
\newblock {\em Mathematical Finance}, 8(4):325--347, 1998.

\bibitem{kalai2002efficient}
Adam~T. Kalai and Santosh Vempala.
\newblock Efficient algorithms for universal portfolios.
\newblock {\em Journal of Machine Learning Research}, pages 423--440, 2002.

\bibitem{kelly1956new}
John~L. Kelly.
\newblock A new interpretation of information rate.
\newblock {\em The Bell System Technical Journal}, 35(4):917--926, 1956.

\bibitem{khachiyan1979polynomial}
Leonid~G. Khachiyan.
\newblock A polynomial algorithm in linear programming.
\newblock In {\em Doklady Akademii Nauk}, volume 244, pages 1093--1096. Russian
  Academy of Sciences, 1979.

\bibitem{kivinen1997exponentiated}
Jyrki Kivinen and Manfred~K. Warmuth.
\newblock Exponentiated gradient versus gradient descent for linear predictors.
\newblock {\em Information and Computation}, 132(1):1--63, 1997.

\bibitem{leung2006information}
Gilbert Leung and Andrew~R. Barron.
\newblock Information theory and mixing least-squares regressions.
\newblock {\em IEEE Transactions on Information Theory}, 52(8):3396--3410,
  2006.

\bibitem{li2014online}
Bin Li and Steven C.~H. Hoi.
\newblock Online portfolio selection: {A} survey.
\newblock {\em ACM Computing Surveys (CSUR)}, 46(3):1--36, 2014.

\bibitem{li2018online}
Bin Li and Steven C.~H. Hoi.
\newblock {\em Online portfolio selection: {P}rinciples and algorithms}.
\newblock CRC Press, 2018.

\bibitem{luo2018efficient}
Haipeng Luo, Chen-Yu Wei, and Kai Zheng.
\newblock Efficient online portfolio with logarithmic regret.
\newblock In {\em Proceedings of the 33rd international Conference on Neural
  Information Processing Systems}, volume~31, pages 8245--8255, 2018.

\bibitem{maclean2011kelly}
Leonard~C. MacLean, Edward~O. Thorp, and William~T. Ziemba.
\newblock {\em The {K}elly capital growth investment criterion: {T}heory and
  practice}, volume~3.
\newblock World Scientific, 2011.

\bibitem{markowitz1952}
Harry~M. Markowitz.
\newblock Portfolio selection.
\newblock {\em The Journal of Finance}, 7(1):77--91, 1952.

\bibitem{markowitz1991foundations}
Harry~M. Markowitz.
\newblock Foundations of portfolio theory.
\newblock {\em The Journal of Finance}, 46(2):469--477, 1991.

\bibitem{marteau2019globally}
Ulysse Marteau-Ferey, Francis Bach, and Alessandro Rudi.
\newblock Globally convergent newton methods for ill-conditioned generalized
  self-concordant losses.
\newblock {\em Advances in Neural Information Processing Systems}, 32, 2019.

\bibitem{mhammedi2022damped}
Zakaria Mhammedi and Alexander Rakhlin.
\newblock Damped online {N}ewton step for portfolio selection.
\newblock {\em arXiv preprint arXiv:2202.07574}, 2022.

\bibitem{narayanan2017efficient}
Hariharan Narayanan and Alexander Rakhlin.
\newblock Efficient sampling from time-varying log-concave distributions.
\newblock {\em The Journal of Machine Learning Research}, 18(1):4017--4045,
  2017.

\bibitem{nemirovskii1983problem}
Arkadii~S. Nemirovskii and David~B. Yudin.
\newblock Problem complexity and method efficiency in optimization.
\newblock 1983.

\bibitem{nesterov2003introductory}
Yurii Nesterov.
\newblock {\em Introductory lectures on convex optimization: {A} basic course},
  volume~87.
\newblock Springer Science \& Business Media, 2003.

\bibitem{nesterov1994interior}
Yurii Nesterov and Arkadii~S. Nemirovskii.
\newblock {\em Interior-point polynomial algorithms in convex programming}.
\newblock SIAM, 1994.

\bibitem{orseau2017soft}
Laurent Orseau, Tor Lattimore, and Shane Legg.
\newblock Soft-{B}ayes: Prod for mixtures of experts with log-loss.
\newblock In {\em International Conference on Algorithmic Learning Theory},
  pages 372--399. PMLR, 2017.

\bibitem{ostrovskii2021finite}
Dmitrii~M. Ostrovskii and Francis Bach.
\newblock Finite-sample analysis of ${M}$-estimators using self-concordance.
\newblock {\em Electronic Journal of Statistics}, 15(1):326--391, 2021.

\bibitem{paris2021online}
Quentin Paris.
\newblock Online learning with exponential weights in metric spaces.
\newblock {\em arXiv preprint arXiv:2103.14389}, 2021.

\bibitem{rigollet2012sparse}
Philippe Rigollet and Alexandre~B. Tsybakov.
\newblock Sparse estimation by exponential weighting.
\newblock {\em Statistical Science}, 27(4):558--575, 2012.

\bibitem{rubinstein2002markowitz}
Mark Rubinstein.
\newblock Markowitz's ``portfolio selection:'' a fifty-year retrospective.
\newblock {\em The Journal of Finance}, 57(3):1041--1045, 2002.

\bibitem{shor1971minimization}
Naum~Z. Shor and Nikolay~G. Zhurbenko.
\newblock A minimization method using the operation of extension of the space
  in the direction of the difference of two successive gradients.
\newblock {\em Cybernetics}, 7(3):450--459, 1971.

\bibitem{vaidya1989new}
Pravin~M. Vaidya.
\newblock A new algorithm for minimizing convex functions over convex sets.
\newblock In {\em 30th Annual Symposium on Foundations of Computer Science},
  pages 338--343. IEEE Computer Society, 1989.

\bibitem{van2020open}
Tim Van~Erven, Dirk Van~der Hoeven, Wojciech Kot{\l}owski, and Wouter~M.
  Koolen.
\newblock Open problem: {F}ast and optimal online portfolio selection.
\newblock In {\em Proceedings of the 33rd Conference On Learning Theory}, pages
  3864--3869. PMLR, 2020.

\bibitem{van2016probability}
Ramon van Handel.
\newblock {\em Probability in High Dimension: Lecture Notes}.
\newblock Princeton University, 2016.

\bibitem{wu2021minimax}
Keru Wu, Scott Schmidler, and Yuansi Chen.
\newblock Minimax mixing time of the {M}etropolis-adjusted {L}angevin algorithm
  for log-concave sampling.
\newblock {\em arXiv preprint arXiv:2109.13055}, 2021.

\bibitem{zimmert2022pushing}
Julian Zimmert, Naman Agarwal, and Satyen Kale.
\newblock Pushing the efficiency-regret {P}areto frontier for online learning
  of portfolios and quantum states.
\newblock {\em arXiv preprint arXiv:2202.02765}, 2022.

\bibitem{zinkevich2003online}
Martin Zinkevich.
\newblock Online convex programming and generalized infinitesimal gradient
  ascent.
\newblock In {\em Proceedings of the 20th International Conference on Machine
  Learning}, pages 928--936, 2003.

\end{thebibliography}

\end{document}